\documentclass[notitlepage,reqno,11pt]{amsart}

\usepackage{amsmath}
\usepackage{amsfonts,amssymb,latexsym,epsfig,amsthm, amscd}
\usepackage{color}
\usepackage[foot]{amsaddr}

\usepackage[colorlinks, allcolors=blue]{hyperref}

\usepackage[margin=1in]{geometry}

\numberwithin{equation}{section}
 \newtheorem{assumption}{Assumption}[section]
\newtheorem{lemma}{Lemma}[section]
\newtheorem{theorem}{Theorem}[section]

\newtheorem{remark}{Remark}[section]

\newcommand{\E}        {{ {\mathbb E}}}
\renewcommand{\P}      {{ {\mathbb P}}}

\newcommand{\R}        {{\mathbb R}}
\newcommand{\N}        {{ {\mathbb N}}}
\newcommand{\Z}        {{\mathbb Z}}
\renewcommand\AA{{\mathcal A}}

\newcommand\LL{{\mathcal L}}

\renewcommand{\S}  {\bar{S}}
\newcommand{\I}  {\bar{I}}
\newcommand{\rR} {\bar{R}}

\newcommand\afrak  {\mathfrak{a}}

\newcommand{\eps}  {{\varepsilon}}

\newcommand{\ep}{\epsilon}

\newcommand{\Cov}{\text{\rm Cov}}
\newcommand{\wt}{\widetilde}

\newcommand{\qandq}{\quad\mbox{and}\quad}

\newcommand{\qforq}{\quad\mbox{for}\quad}

\newcommand{\qasq}{\quad\mbox{as}\quad}
\newcommand{\qinq}{\quad\mbox{in}\quad}

\newcommand{\non}{\nonumber}
\newcommand{\RA}{\Rightarrow}

\newcommand{\ttl}{\Large Multi-patch epidemic models with general exposed \\[5pt] and infectious periods}

\begin{document}

\allowdisplaybreaks

\title[]{\ttl}

\author[Guodong \ Pang]{Guodong Pang$^*$}
\address{$^*$Department of Computational Applied Mathematics and Operations Research,
George R. Brown College of Engineering,
Rice University,
Houston, TX 77005}
\email{gdpang@rice.edu}

\author[{\'E}tienne \ Pardoux]{{\'E}tienne Pardoux$^\dag$}
\address{$^\dag$Aix--Marseille Universit{\'e}, CNRS, Centrale Marseille, I2M, UMR \ 7373 \ 13453\ Marseille, France}
\email{etienne.pardoux@univ.amu.fr}

\date{\today}

\begin{abstract} 
We study  multi-patch epidemic models where individuals may migrate from one patch to another in either of the susceptible, exposed/latent, infectious and recovered states. 
We assume that infections occur both locally with a rate that depends on the patch as well as
``from distance" from all the other patches. 
The migration processes among the patches in either of the four states are assumed to be Markovian, and independent of the exposed and infectious periods.   These periods have general distributions, and are not affected by the possible migrations of the individuals. 
The infection ``from distance'' aspect introduces a new formulation of the infection process, which, together with the migration processes, brings technical challenges in proving the functional limit theorems. Generalizing the methods in \cite{PP-2020}, we establish a functional law of large number (FLLN) and a function central limit theorem (FCLT) for the susceptible, exposed/latent, infectious and recovered processes. 
In the FLLN, the limit is determined by a set of Volterra integral equations. In the special case of deterministic exposed and infectious periods, the limit becomes a system of ODEs with delays. In the FCLT, the limit is given by a set of stochastic Volterra integral equations driven by a sum of independent Brownian motions and  continuous Gaussian processes with an explicit covariance structure.
\end{abstract}


\maketitle

\section{Introduction}

Multi-patch epidemic models have been used to study various infectious diseases, for example, nosocomial infection \cite{magal2013two}, vector-borne diseases \cite{iggidr2016dynamics}, HIV/AIDS transmission \cite{huang1992stability}, SARS epidemic \cite{magal2018final}, and so on.   Patches  refer to different locations, for example, a densely populated city and a less populated rural area, and thus such models capture geographic heterogeneity. 
 It also helps to study the effect of migrations or lock-down measures among different population groups or locations. In the Covid-19 pandemic, it has been observed that the infectivity in different regions may vary and is impacted by various social-distance and lock-down measures \cite{prague2020,shah2020modelling}. 

Compartment ordinary differential equation (ODE)  models are often used to study the dynamics of such multi-patch epidemic models. It is well known that the ODE dynamics arises from the Markovian assumptions in the stochastic multi-patch epidemic model, that is, the infection process is Poisson, the infectious (and/or exposed/latent) periods are exponentially distributed and the migration processes are also Markovian \cite{andersson2012stochastic,britton2018stochastic, magal2016final,magal2018final,bichara2018multi,iggidr2016dynamics}. 

In this paper, we study multi-patch  SEIR models, in which  each individual may experience successively Susceptible (S), Exposed (E), Infectious (I) and Recovered (R) periods, and   the exposed/latent and  infectious  periods have a general joint distribution (possibly correlated), while the migration processes are Markovian. 
SEIR models are widely used to study the propagation of various epidemics in that they capture both exposed/latent and infectious stages, see, e.g., \cite{andersson2012stochastic,britton2018stochastic}.  
The infection in the multi-patch model is assumed to be both local,  and from distance. That is, the infection rate in a given patch depends on the susceptible population in that patch, and on the infectious population in all the patches.  Individuals may migrate from one patch to another in each of the Susceptible, Exposed (Latent), Infected and Recovered stages. The reason for infection at distance is twofold. 
First, if  there were no migration among a subset of the patches (i.e., the migration rates among them are zero),
we could consider some of the patches as substructures of the population, like age classes, which infect each other. Second, some of the movements of the population should not be considered as migrations, but visits from one patch to another, during a week--end or holiday, with a return at home at the end of a short period. Such movements may produce infection of a susceptible individual in patch $i$ by an
infectious individual from patch $i'\not=i$. Those displacements would be very complicated to model as such. We think that infection at distance is a reasonable way to model infections due to such movements.  Mathematically, this requires a much more sophisticated model for the instantaneous infection rate, see \eqref{eqn-Upsilon}. The formulation is new in the literature, even in the Markovian setting. Moreover, it brings new challenges in the proofs of the limit theorems (see Lemmas \ref{le:lowbd}--\ref{le:uniq} and Lemmas \ref{estimhatUps}--\ref{lem-estimate-SIR}), compared to those of the one-patch non-Markovian models in \cite{PP-2020}.


We describe the evolution dynamics by tracking the time epochs of becoming exposed and/or infectious and the location of an individual at these event times. Specifically, 
in the multi-patch SEIR model, each individual tracks the time epochs of becoming exposed, infected and recovered, and is associated with two Markov chains that are used to track their movement starting when the individual becomes  exposed and infectious, respectively.  For the initially exposed and/or infected individuals, we also assume that their remaining exposing and/or infectious periods have general distributions, which may be different from those of the newly exposed/infected individuals. For these initially exposed/infectious individuals, we also track their movement among the patches using Markov chains while being exposed/infectious. 
Although the idea of tracking time epochs of each individual is analogous to the one-patch SIR/SEIR models in \cite{PP-2020}, the migrations among different patches makes the system dynamics much more challenging to describe (see the equations \eqref{S-rep-SEIR}--\eqref{R-rep-SEIR}) and analyze, despite the independence of the Markovian migration processes from the exposed/infectious durations.  The formulation and proofs constitute non-trivial generalizations of the one-patch SEIR model. 

Given the representations with these time epochs and location processes, we show a functional law of large numbers (FLLN) and a functional central limit theorem (FCLT).   More precisely, we divide by $N$ (which is the total population size)  the equations for the evolution of the numbers of individuals in each patch and compartment, thus obtaining equations for the proportions of individuals in each patch and compartment. Those proportions are shown to converge in probability, locally uniformly in $t$, towards the solution of a system of integral equations. This is our FLLN. Next, if we multiply by $\sqrt{N}$ the difference between the proportions in the $N$ stochastic model and the limiting proportion, that renormalized difference is shown to converge in distribution to the solution of a set of linear integral equations driven by Gaussian noise. This is the FCLT.  Needless to say, the deterministic system of equations obtained for the FLLN limit is much simpler than the $N$ stochastic model. It can be used for predictions of an epidemic which affects a population of reasonably large size, at least away form the very early and very final stages, when the number of infected individuals is not of the order of $N$. Note that the FCLT tells us that the error made on the proportions by using the deterministic model is of the order of $N^{-1/2}$.

The FLLN limits are determined by a set of Volterra integral equations.
When the infectious (and exposed/latent) periods are deterministic, we can write the fluid integral equations as a set of ODEs with delay (Remark \ref{rem-FLLN-det-SEIR}). The limit processes in the FCLT are determined by a set of stochastic Volterra integral equations, driven by a sum of independent Brownian motions and continuous Gaussian processes with a certain covariance structure. When the infectious   (and exposed/latent) periods are deterministic, the limits become stochastic differential equations with linear drifts and delay (Remark \ref{rem-FCLT-det-SEIR}).   In both FLLN and FCLT limits, 
the effects of migrations  are exhibited through the transition probability functions and transition rates of the migration Markov processes. 
We discuss how the results simplify in the SIR model as a special case, and 
 also how the approach and results can be extended to study multi-patch SIS and SIRS models (Section \ref{sec-special}). 

In the proofs of these results,  we employ  Poisson random measures (PRMs) that
are constructed as the sums of the Dirac masses at the time epochs of becoming exposed and infectious, the infectious and exposing periods and  the Markov chain 
starting from the location of each individual at those epochs. 
While to the Markov migration process have naturally associated martingales, which are easily proved to be tight in the appropriate path space, the non--Markovian nature of the epidemic process does not produce obvious martingales. However, as in our previous work \cite{PP-2020}, we are able to find various martingales attached to various and non--standard filtrations, which are constructed from our representation of the epidemic by integrals with respect to PRMs.   We use the martingale properties and convergence theorems as critical tools in the proofs.  
For the single-patch SIR and SEIR models with general infectious and exposing periods,  an approach using PRMs that are constructed  at the time epochs of becoming infectious (and/or exposed),  was developed in \cite{PP-2020}. The approach is further developed in this paper for multi-patch SIR and SEIR models, to track the locations of each individual at each event epoch.
Incorporating infection from distance in addition to local infections in the model  also brings in new technical challenges in the proofs of both the FLLN and FCLT. 

This paper contributes to the limited literature on stochastic epidemic models with general exposed/infectious periods. 
We refer the readers to the overview in Chapter 3.4 of \cite{britton2018stochastic} on the common approaches to study non-Markovian epidemic models and
the limit theorems for the final sizes of the epidemic; see also the recent method using piecewise Markov deterministic processes in \cite{clancy2014sir}  and \cite{gomez2017sir} for the SIR model.  
FLLNs and FCLTs are proved for some age and density dependent population models in  \cite{wang1975limit, wang1977central,wang1977gaussian}, which includes the SIR model with the infection rate depending on the number of infectious individuals, and general infectious period as a special case.  
Reinert \cite{reinert1995asymptotic} proves a FLLN for the empirical measure of the SIR epidemic dynamics using Stein's method, while no FCLT has been proved with that approach. 
In \cite{PP-2020}, both FLLN and FCLT were established for the epidemic dynamics in the classical models (SIR, SIS, SEIR, SIRS) where the PRM representations of the dynamics plays a fundamental role in the proofs. 
The FCLT limit for the SIR model in \cite{PP-2020} is similar to that in  \cite{wang1975limit, wang1977central,wang1977gaussian}, however, the proof approaches are completely different; in addition, the distribution function of the infectious periods is assumed to be continuously differentiable in  \cite{wang1975limit, wang1977central,wang1977gaussian} while no condition is imposed in \cite{PP-2020}. 
We  highlight that the distribution functions of the exposed and infectious periods in this paper are general without requiring any conditions. 
The integral equations for the SEIR model in \cite{PP-2020} are also used to estimate the state of the Covid-19 epidemic in \cite{FPP2021}.
For SIR and SEIR models with varying infectivity, where each individual is associated with an i.i.d. random infectivity, which is a function of the time elapsed since infection, FLLN and FCLT have recently been established in \cite{FPP2020b,PP2020-FCLT-VI}. 
Although Volterra integral equations were used to describe the proportion of infectious population in the SIS, SIR or SEIR model without proving an FLLN (see \cite{brauer1975nonlinear,BCF-2019,cooke1976epidemic,diekmann1977limiting,hethcote1995sis,van2000simple}), 
as far as we know no Volterra integral equations have been proposed so far for multi-patch epidemic models with general  infectious (and/or exposed) periods. 
Our work shows both FLLN and FCLT for non-Markovian multi-patch models, and identify (stochastic) multidimensional Volterra integral equations as their limits. 

 It is also worth mentioning the multi-type epidemic models where the population splits up into multiple groups of individuals and each group may infect any other group in addition to itself (no migration), see Chapters 6.1 and 6.2 in \cite{andersson2012stochastic} and \cite{ball1986unified,ball1993final}. The special case of our model  with zero migration rates  covers that situation. 
 In those models, proportionate mixing taking into account control measures like social distance or lockdowns can also be incorporated. 



%

\subsection{Notation}
Throughout the paper, $\N$ denotes the set of natural numbers, and $\R^k (\R^k_+)$ denotes the space of $k$-dimensional vectors
with  real (nonnegative) coordinates, with $\R (\R_+)$ for $k=1$.  For $x,y \in\R$, denote $x\wedge y = \min\{x,y\}$ and $x\vee y = \max\{x,y\}$. 
Let $D=D([0,T], \R)$ denote the space of $\R$--valued c{\`a}dl{\`a}g functions defined on $[0,T]$. Throughout the paper, convergence in $D$ means convergence in the  Skorohod $J_1$ topology, see chapter 3 of \cite{billingsley1999convergence}. 
 Also, $D^k$ stands for the $k$-fold product equipped with the product topology.  Let $C$ be the subset of $D$ consisting of continuous functions.   Let $C^1$ consist of differentiable functions whose derivative is continuous. 
 For any function $x\in D$, we use the notation $\|x\|_T= \sup_{t\in [0,T]} |x(t)|$. For two functions $x,y \in D$, we use $x\circ y(t) = x(y(t))$ denote their composition.
 All random variables and processes are defined on a common complete probability space $(\Omega, \mathcal{F}, \P)$. The notation $\RA$ means convergence in distribution. We use ${\bf1}_{\{\cdot\}}$ for indicator function, and occasionally we shall write {\bf1}\{.\} in case the first notation is not readable enough. 
 We use small-$o$ notation for real-valued functions $f$ and non-zero $g$: $f(x)=o(g(x))$ if $\limsup_{x\to\infty} |f(x)/g(x)| =0$. We use $\hat\imath$ to denote the unit imaginary number. 
 We write $F(t) = \int_0^t F(ds)$ for a cumulative distribution function (c.d.f.) $F$ on $\R_+$.  
For any measure $\mu$ on $\R$ and $f$ a measurable and $\mu$--integrable function, the integral 
 $\int_a^bf(t)\mu(dt)$ will mean $\int_{(a,b]}f(t)\mu(dt)$.

\bigskip


\section{Model description} \label{sec-model}

We consider a multi-patch epidemic model, where individuals in each patch experience the Susceptible--Exposed (Latent)--Infectious--Recovered (SEIR) process. The patches may refer to populations in different locations, for example, a densely populated city and a less populated rural area. As explained in the introduction, susceptible
individuals in each patch are infected both locally, by infectious individuals located in the same patch, and at distance, by infectious individuals from other patches. 
The rate of infection is different in each patch (because of the differences in the 
density of population or in the type of available public transportations), while the law of the infectious period is the same (due to the same illness). 

Let $N$ be the total population size and $L$ be the number of patches. The set of patches will be denoted 
$\LL=\{1,\ldots,L\}$. (We use indices $i,i',\ell,\ell'$ for elements in $\LL$, and occasionally $i'', \ell''$.) 
For each patch $i\in\LL$, let $S^N_i(t), E^N_i(t), I^N_i(t), R^N_i(t)$ count the numbers of individuals that are susceptible, exposed (latent), infectious and recovered 
in patch $i$ at time $t$, respectively. 
We have the balance equation: 
\begin{equation}\label{eq:balance}
N \;=\; \sum_{i=1}^L (S^N_i(t) + E^N_i(t) + I^N_i(t)  + R^N_i(t) )\,, \quad t \ge 0\,. 
\end{equation}
Assume that $S^N_i(0)>0, \sum_{i=1}^L(E^N_i(0)+I^N_i(0))>0$ and $R^N_i(0)=0$ for each $i\in \mathcal{L}$. It is straightforward to allow $R^N_i(0)$ to be nonzero, however the initially recovered individuals can be removed from the population under consideration from the beginning.

\subsection{The infection process} 

Let  $\lambda_i$, a positive constant, be the infection rate of patch $i\in \mathcal{L}$.  
Define the following processes, for some $0\le\gamma\le1$,
\begin{align} \label{eqn-Upsilon}
\Upsilon^N_i(t)&\;=\;\frac{S^N_i(t)\sum_{\ell=1}^L\kappa_{i\ell}I^N_\ell(t)}{N^{1-\gamma}(S^N_i(t)+ E^N_i(t)+I^N_i(t)+R^N_i(t))^\gamma}\,,\quad i \in \mathcal{L}\,, 
\end{align}
where $\kappa_{ii}=1$  and $\kappa_{i\ell} \ge0$ for $i\neq \ell$ represent the infectivity from distance. Let $\bar\kappa_i :=\sum_{\ell=1}^L \kappa_{i\ell} $ and $\bar\kappa:= \max_{i\in \mathcal{L}} \bar\kappa_i$. 
The rate of new infections in patch $i$ at time $t$ is $\lambda_i\Upsilon^N_i(t)$. Let us explain the role of the parameter $\gamma$.

In the homogeneous population model, where $L=1$, \eqref{eq:balance} tells us that in the unique patch,
$S^N(t) + E^N(t) + I^N(t)  + R^N(t)=N$, hence  $\Upsilon^N(t)$ is the same, irrespective of the value of 
$\gamma$. The rationale of this form of the infection rate is as follows. Each infectious individual meets others at rate
$\beta$.  Since we assume that the individual who is met is chosen uniformly at random in the whole population, he/she is susceptible with probability $S^N(t)/N$. In that case, the encounter results in a new infection with probability $p$. If we let $\lambda=\beta\times p$, we find the above formula $\lambda_i\Upsilon^N_i(t)$ for the rate of new infections in case $L=1$. Now, consider the case $L>1$. We do not factorize $\lambda$ into $\beta\times p$ anymore, or equivalently do as if $p=1$.

In order to make the role of the parameter $\gamma$ transparent, let us define $B^N_i(t) =S_i^N(t) + E_i^N(t) + I_i^N(t)  + R_i^N(t)$, the total population in the patch $i$, and rewrite
\[\lambda_i\Upsilon^N_i(t)=\lambda_i\left(\frac{B^N_i(t)}{N}\right)^{1-\gamma}\frac{S^N_i(t)}{B^N_i(t)}\sum_{\ell=1}^L\kappa_{i\ell}I^N_\ell(t)\,.\]
 In the case $\gamma=1$, the rate of encounters of individuals in patch $i$ by a given infectious is given as $\lambda_i$ for an infectious of the same patch, and equal to $\lambda_i\kappa_{ii'}$ for an infectious from patch $i'$, whatever the total population in patch $i$ at time $t$ may be. This factor gets multiplied by the probability that a randomly chosen individual in patch $i$ be susceptible, which equals $S^N_i(t)/B^N_i(t)$. In the case $\gamma=0$, the same rate is proportional to $B^N_i(t)$, the total population of patch $i$ at time $t$. In the intermediate cases, the rate lies between those two extremes. The case $\gamma=1$ seems to be used in most 
 spatial epidemics models. The values of $\lambda_i$'s can correct for the different densities of population of the various patches, resulting in more or less encounters. Indeed, we believe that the rate of encounters by any individual is very different in a densely populated area, from what it is in a desert. However, especially in the stochastic model, the population size in each patch may fluctuate significantly, which we believe is a good motivation for using a model with $\gamma<1$. Our model is probably new in the cases $0<\gamma<1$. It is one possible way of interpolating between the two extreme cases $\gamma=1$ and $\gamma=0$.

We shall prove the FLLN for any value of $\gamma\in[0,1]$, and the FCLT  only for $\gamma\in[0,1)$ in the general case, and for all $\gamma\in[0,1]$ in the case that infections are only local, i.e.,
$\kappa_{i\ell} =0$ for $i\neq \ell$. The reason for this restriction is that in the case $\gamma=1$ and 
$\sum_{\ell\not=i}\kappa_{i\ell}>0$, we are not able to establish the estimate \eqref{eq:esimhatUps} in Lemma \ref{estimhatUps} below.

Note that $\Upsilon^N_i(t)\le \left(\frac{S^N_i(t)}{N}\right)^{1-\gamma}\sum_{\ell=1}^L \kappa_{i\ell}I^N_\ell(t)$, so that in the case $\gamma<1$, $\Upsilon^N_i(t)=0$ whenever $S^N_i(t)+I^N_i(t)+R^N_i(t)=0$. By convention, we shall assume that the same holds in case $\gamma=1$, i.e.,  $\frac{0}{0}=0$. $\lambda_i\Upsilon^N_i(t)$ is the rate of new infections in patch $i$ at time $t$. It is of course $0$ if patch $i$ is empty.

Let $A^N_i(t)$ be the cumulative counting process of individuals in patch $i$ that get infected on the time interval $(0,t]$. 
Then we can give a representation of the process $A^N_i(t)$ via the standard Poisson random measure (PRM) $Q_{i}$ on $\R^2_+$ (with mean measure $dsd\afrak$), the various  $\{Q_i,\ i\in\LL\}$ being mutually independent, 
\begin{align}\label{An-rep-1}
A^N_i(t)\;=\;\int_0^t\int_0^\infty{\bf1}_{\afrak\le \lambda_i\Upsilon^N_i(s^-)}Q_{i}(ds,d\afrak)\,,\quad t\ge 0\,. 
\end{align}
(We write $dA^N_i(s)$ to denote $\int_0^\infty{\bf1}_{\afrak\le \lambda_i\Upsilon^N_i(s^-)}Q_{i}(ds,d\afrak)$ so that $A^N_i(t) = \int_0^t d A^N_i(s)$.) 
Equivalently, we could write 
\begin{align} \label{An-rep-2}
A^N_i(t) \;=\; P_{A,i} \left( \lambda_i \int_0^t  \Upsilon^N_i(s) ds \right) \,,\quad t\ge 0\,,
\end{align}
where $P_{A,i}$ is a unit-rate Poisson process, and independent from each other for $i\in \mathcal{L}$. But the first description will be more useful for us.
We let $\{\tau^N_{j,i},\ j\ge1\}$ denote the successive jump times of the process
$A^N_i$, for $i\in \mathcal{L}$. (Note that all the analysis and results can be easily extended to a deterministic time-dependent rate function $\lambda_i(t)$. For example, in the expression above, we have an integral  $ \int_0^t \lambda_i(s) \Upsilon^N_i(s) ds$ instead.)

\subsection{ On the exposed and infectious periods} 

The $E^N_i(0)$ initially exposed individuals  experience the exposed and  infectious periods before recovery. 
Let $\{\eta^0_{k,i}: k=1,\dots, E^N_{i}(0)\}$ be the remaining exposed periods of the initially exposed individuals in patch $i$.
After the exposed period, let $\{\zeta_{-k,i}:k=1,\dots,E^N_{i}(0)\}$ be the durations of their infectious periods. 
The $I^N_i(0)$ initially  infectious
 individuals  experience a remaining infectious period before recovery, and   let $\zeta^0_{k,i}$, $k=1,\dots, I^N_i(0)$, denote their remaining infectious periods. 
The $A^N_i(t)$ newly infected individuals in patch $i$ experience the exposed and infectious periods. Let $\{\eta_{j,i}:j \in \N\}$ and $\{\zeta_{j,i}:j \in \N\}$ be the associated exposing and infectious periods. 

Assume that $\{\zeta^0_{k,i}\}$, $\{(\eta^0_{k,i},\zeta_{-k,i}) \}$ and $\{(\eta_{j,i}, \zeta_{j,i})\}$ are all i.i.d. sequences of  random variables having distribution functions $F_0$, $H_0(du,dv)$ and $H(du,dv)$, respectively, and they are also mutually independent. Note that $\zeta_{j,i}$ is defined for $j\in\Z$ and $i\in \mathcal{L}$ (those with $j<0$ code the infectious periods of the initially exposed individuals, while those with $j>0$ code the infectious periods of the newly exposed individuals). Let $G_0$ and $ F$ be the marginals of $H_0$ for $\eta^0_{k,i}$ and $ \zeta_{-k,i}$, and $G$ and $ F$ be the marginals of $H$ for $\eta_{k,i}$ and $ \zeta_{k,i}$, respectively.  (It is reasonable to assume that the marginal distributions of $\zeta_{-k,i}$ and $\zeta_{k,i}$ are the same.) Also let $F_0(\cdot|u)$ and $F(\cdot|u)$ be the conditional c.d.f.'s of $\zeta_{-k,i}$ and $\zeta_{k,i}$, given that
$\eta^0_{k,i}=u$ and $\eta_{k,i}=u$, respectively.   
Let $G^c_0=1-G_0$, $G^c=1-G$, $F_0^c = 1- F_0$ and $F^c=1-F$.

\subsection{ The migration processes}
Individuals may migrate from patch $\ell$ to $\ell'$ in any of the four epidemic stages, with rates $\nu_{S,\ell,\ell'}$, $\nu_{E,\ell, \ell'}$, $\nu_{I,\ell,\ell'}$ and $\nu_{R,\ell,\ell'}$ for the susceptible, exposed, infectious and recovered ones, respectively, for $\ell,\ell'\in\mathcal{L}$. For each individual, the times between migrations in each of the stages are exponentially distributed. 

In order to track the location/patch of the $j$--th individual who got exposed in patch $\ell$ at time $\tau^N_{j,\ell}$, 
 we first use the Markov process $X^j_\ell$, taking values in $\mathcal{L}$,  associated to the rates $\nu_{E,\cdot,\cdot}$. It takes effect from the time $\tau^N_{j,\ell}$ of becoming exposed, until the time 
 $\tau^N_{j,\ell}+\eta_{j,\ell}$ when this individual becomes infectious. Given that this individual has migrated to patch $\ell'$ at the end of the exposed period (she/he may have done several migrations to other patches during the exposed period), that is $X^j_\ell(\eta_{j, \ell}) = \ell'$, we then use another Markov process $Y^{j,\ell}_{\ell'}$ to track the location/patch of the individual during the infectious period $\zeta_{j, \ell}$, starting from $\ell'$ and associated to the rates $\nu_{I,\cdot,\cdot}$.
This process  $Y^{j,\ell}_{\ell'}$  only takes effect from the time of becoming infectious  $\tau^N_{j,\ell}+\eta_{j,\ell}$, until the time of recovery $\tau^N_{j,\ell}+\eta_{j,\ell}+\zeta_{j,\ell}$. 
 Suppose that the individual has migrated to patch $\ell''$ at the end of the infectious period, that is, $Y^{j,\ell}_{\ell'}(\zeta_{j,\ell})=\ell''$. The individual will then belong to the compartment of recovered individuals, and will migrate among patches according to the rates $\nu_{R,\cdot,\cdot}$. 
 Similarly we use $X^{0,k}_\ell$ and $Y^{-k,\ell}_{\ell'}$ for the initially exposed individuals $k=1,\dots, E^N_\ell(0)$ that have been exposed at time 0 in patch $\ell$. 
 $X^{0,k}_\ell$  takes effect from time 0 to $\eta^0_{k,\ell}$. 
They are again associated with the rates $\nu_{E,\cdot,\cdot}$  and $\nu_{I,\cdot,\cdot}$, respectively.
 In addition, we also use $Y^{0,k}_{\ell}$  for the initially infectious individuals $k=1,\dots, I^N_\ell(0)$ that have been infectious at time 0 in patch $\ell$. It is again associated with the rates $\nu_{I,\cdot,\cdot}$. 
  $Y^{0,k}_\ell$  takes effect from time 0 to $\zeta^0_{k,\ell}$.  We assume that for each $j$, $X^j_\ell$ and $Y^{j,\ell}_{\ell'}$ are independent for $\ell, \ell'\in \LL$, and for each $k$, $X^{0,k}_\ell$ and $ Y^{-k,\ell}_{\ell'}$ are independent for $\ell, \ell'\in \LL$. We also  assume that all these Markov processes $\{X^j_\ell, Y^{j,\ell}_{\ell'}\}_{j,\ell,\ell'}$, $\{X^{0,k}_\ell, Y^{-k,\ell}_{\ell'}\}_{k,\ell,\ell'}$  and  $\{Y^{0,k}_{\ell}\}_{k,\ell}$  are mutually independent.

Let  $p_{\ell,\ell'}(t) = P(X^j_\ell(t)=\ell')$ and 
$q_{\ell',\ell''}(t) = \P(Y^{j,\ell}_{\ell'}(t)=\ell'')$ for $\ell,\ell', \ell''\in \mathcal{L}$, $j \ge 1$ and $t\ge 0$. 
For each $\ell$,
the processes $\{X^{0,k}_\ell\}_k$ have the same transition function $(p_{\ell,\ell'}(\cdot))_{\ell'\in \LL}$ as  $\{X^{j}_\ell\}_j$, and
 the process  $\{Y^{-k,\ell}_{\ell'}\}_k$ has the same transition function $(q_{\ell',\ell''}(\cdot))_{\ell',\ell''\in \LL}$ as  $\{Y^{j,\ell}_{\ell'}\}_j$, and the process $\{Y^{0,k}_\ell\}_k$ also has the transition function $(q_{\ell,\ell'}(\cdot))_{\ell,\ell'\in \LL}$.

\subsection{ Epidemic evolution dynamics}

The multi-patch SEIR epidemic evolution dynamics can be described as follows: 
\begin{align} \label{S-rep-SEIR}
S^N_i(t)&\;=\;S^N_i(0)-A^N_i(t) + 
\sum_{\ell =1, \ell \neq i}^L \left( P_{S,\ell,i}\left(\nu_{S,\ell,i}\int_0^tS^N_\ell(s)ds\right) - P_{S,i,\ell}\left(\nu_{S,i,\ell} \int_0^tS^N_i(s)ds)\right)\right)\,,\\  
 \label{E-rep-SEIR}
E^N_i(t)&\;=\;\sum_{\ell=1}^L \sum_{k=1}^{E^N_\ell (0)} {\bf1}_{t<\eta_{k,\ell}^0}{\bf1}_{X^{0,k}_\ell(t)=i}
+\sum_{\ell=1}^L \sum_{j=1}^{A^N_\ell(t)} {\bf 1}_{\tau^N_{j,\ell} + \eta_{j,\ell} >t} {\bf1}_{X^j_\ell(t-\tau^N_{j,\ell})=i}\,\,\,, 
\\
I^N_i(t) &\;=\; \sum_{\ell=1}^L \sum_{k=1}^{I^N_\ell (0)} {\bf1}_{t<\zeta_{k,\ell}^0}{\bf1}_{Y^{0,k}_\ell(t)=i}   +\sum_{\ell=1}^L \sum_{k=1}^{E^N_\ell (0)} {\bf1}_{\eta_{k,\ell}^0 \le t} \left( \sum_{\ell'=1}^L {\bf1}_{X^{0,k}_\ell(\eta_{k,\ell}^0)=\ell'}  {\bf1}_{\eta_{k,\ell}^0 + \zeta_{-k,\ell}>t } {\bf1}_{Y^{-k,\ell}_{\ell'}(t- \eta_{k,\ell}^0)=i} \right) \non \\
& \quad +  \sum_{\ell=1}^L \sum_{j=1}^{A^N_\ell(t)} {\bf 1}_{\tau^N_{j,\ell} + \eta_{j,\ell} \le t} \left(\sum_{\ell'=1}^L  {\bf1}_{X^j_\ell( \eta_{j,\ell})= \ell'}  {\bf 1}_{\tau^N_{j,\ell} + \eta_{j,\ell}  + \zeta_{j,\ell}  > t} {\bf1}_{Y^{j,\ell}_{\ell'}(t-\tau^N_{j,\ell} - \eta_{j,\ell} )=i}\right)\,,  \label{I-rep-SEIR}\\
R^N_i(t) &\;=\;  \sum_{\ell=1}^L \sum_{k=1}^{I^N_\ell (0)} {\bf1}_{\zeta_{k,\ell}^0 \le t}{\bf1}_{Y^{0,k}_\ell(\zeta_{k,\ell}^0)=i}  +\sum_{\ell=1}^L \sum_{k=1}^{E^N_\ell (0)} \left( \sum_{\ell'=1}^L {\bf1}_{X^{0,k}_\ell(\eta_{k,\ell}^0)=\ell'}  {\bf1}_{\eta_{k,\ell}^0 + \zeta_{-k,\ell} \le t} {\bf1}_{Y^{-k,\ell}_{\ell'}( \zeta_{-k,\ell})=i}\right) \non \\
& \quad +  \sum_{\ell=1}^L \sum_{j=1}^{A^N_\ell(t)}  \left( \sum_{\ell'=1}^L {\bf1}_{X^j_\ell( \eta_{j,\ell})=\ell'}  {\bf 1}_{\tau^N_{j,\ell} + \eta_{j,\ell}  + \zeta_{j,\ell}  \le t} {\bf1}_{Y^{j,\ell}_{\ell'}( \zeta_{j,\ell} )=i}  \right) \non \\
& \quad + 
\sum_{\ell=1, \,\ell\neq i}^L \bigg( P_{R,\ell,i}\left(\nu_{R,\ell,i} \int_0^tR^N_\ell(s)ds\right) - P_{R,i,\ell}\left(\nu_{R,i,\ell} \int_0^tR^N_i(s)ds\right) \bigg) \,, \label{R-rep-SEIR}
\end{align}
where $P_{S,i,\ell}, P_{R,i,\ell}$ , $i,\ell\in \mathcal{L}$, are mutually independent unit-rate Poisson processes, which are globally independent of the $Q_{i}$'s. 
 The dynamics of $S^N_i(t)$ is straightforward since it is simply equal to the number of initially susceptible individuals minus the number of exposed ones in patch $i$ and then take into account the migrations. 
For the dynamics of $E^N_i(t)$, the first term represents the number of initially exposed individuals from patch $\ell$ that remain exposed and are in patch $i$ at time $t$, and the second term represents the number of newly exposed individuals from patch $\ell$ that remain exposed and are in patch $i$ at time $t$. 
In the expression for $I^N_i(t)$,  
the first term counts the number of initially infectious individuals from all the patches that remain infectious and are in patch $i$ at time $t$, and
the second term counts the numbers of initially exposed individuals from all the patches that have become infectious and are in patch $i$ at time $t$ (for tracking purposes, the location at the epochs of becoming infectious  is recorded). Also note that we use  the Markov process $Y^{0,k}_\ell$ to indicate that these are for the initially exposed individuals. The third term counts the number of newly exposed individuals at all patches that have become infectious and are in patch $i$ at time $t$, and we also track the patch in which each individual has become infectious. 
In the expression for $R^N_i(t)$, the first term represents the 
number of initially infectious individuals from patch $\ell$ that have recovered by time $t$ and were in patch $i$ at the time of recovery,  the second term represents the number of initially exposed individuals from patch $\ell$ that have recovered by time $t$, and were  in patch $i$ at the time of recovery, while becoming infectious in patch $\ell'$, 
 the third term represents the number of newly exposed individuals from patch $\ell$ that have recovered by time $t$, and were in patch $i$ at the time of recovery while becoming infectious in patch $\ell'$. 

It is not easy to take the limit as $N\to\infty$ in the formulas of $E^N_i$ and $I^N_i$ above.  We now derive the following representations, which will be very helpful in the proofs of our results. 
\begin{lemma} \label{lem-I1-2-rep}
We have
\begin{align} \label{eqn-E-rep-2-SEIR}
E^N_i(t)&\;=\; E^N_i(0) - \sum_{\ell=1}^L  \sum_{k=1}^{E^N_\ell (0)} {\bf1}_{\eta_{k,\ell}^0 \le t}{\bf1}_{X^{0,k}_\ell(\eta_{k,\ell}^0)=i}  + A^N_i(t) - \sum_{\ell=1}^L\sum_{j=1}^{A^N_\ell(t)} {\bf 1}_{\tau^N_{j,\ell} + \eta_{j,\ell} \le t} {\bf1}_{X^j_\ell(\eta_{j,\ell})=i} \non \\
& \quad + \sum_{\ell \neq i} P_{E,\ell,i}\left(\nu_{E,\ell,i} \int_0^t E^N_\ell (s)ds\right)  -\sum_{\ell\neq i}P_{E,i,\ell}\left(\nu_{E,i,\ell} \int_0^tE^N_i(s)ds\right)\,,
\end{align}
\begin{align} \label{I-rep-2-SEIR}
I^N_i(t) &\;=\;   I^N_i (0) -  \sum_{\ell=1}^L \sum_{k=1}^{I^N_\ell (0)} {\bf1}_{\zeta_{k,\ell}^0 \le t}{\bf1}_{Y^{0,k}_\ell(\zeta_{k,\ell}^0)=i} \non \\
& \quad + \sum_{\ell=1}^L  \sum_{k=1}^{E^N_\ell (0)} {\bf1}_{\eta_{k,\ell}^0 \le t}{\bf1}_{X^{0,k}_\ell(\eta_{k,\ell}^0)=i}  -\sum_{\ell=1}^L \sum_{k=1}^{E^N_\ell (0)}\left( \sum_{\ell'=1}^L {\bf1}_{X^{0,k}_\ell(\eta_{k,\ell}^0)=\ell'}  {\bf1}_{\eta_{k,\ell}^0 + \zeta_{-k,\ell} \le t} {\bf1}_{Y^{-k,\ell}_{\ell'}( \zeta_{-k,\ell})=i}\right)  \non \\
& \quad  + \sum_{\ell=1}^L\sum_{j=1}^{A^N_\ell(t)} {\bf 1}_{\tau^N_{j,\ell} + \eta_{j,\ell} \le t} {\bf1}_{X^j_\ell(\eta_{j,\ell})=i}  -   \sum_{\ell=1}^L \sum_{j=1}^{A^N_\ell(t)}  \left( \sum_{\ell'=1}^L {\bf1}_{X^j_\ell( \eta_{j,\ell})=\ell'}  {\bf 1}_{\tau^N_{j,\ell} + \eta_{j,\ell}  + \zeta_{j,\ell}  \le t} {\bf1}_{Y^{j,\ell}_{\ell'}( \zeta_{j,\ell} )=i}  \right) \non  \\
& \quad + \sum_{\ell \neq i} P_{I,\ell,i}\left(\nu_{I,\ell,i} \int_0^t I^N_\ell (s)ds\right)  -\sum_{\ell\neq i}P_{I,i,\ell}\left(\nu_{I,i,\ell} \int_0^t I^N_i(s)ds\right)\,. 
\end{align}
where $P_{E,i,\ell}$ and $P_{I,i,\ell}$, $i,\ell\in \mathcal{L}$, are all unit-rate Poisson processes, mutually independent, and also independent of $P_{A,i}$, $P_{S,i,\ell}$ and  $P_{R,i,\ell}$. 
\end{lemma}
Before turning to the proof, let us comment on these formulas. 
In the expression of $E^N_i(t)$, the first and third term count the number of initially exposed individuals in patch $i$, and the number of those whose became exposed in patch $i$ on the time interval $[0,t]$. The second and fourth terms subtract the numbers of initially and newly exposed individuals in any patch who have become infectious before time $t$ in patch $i$. Finally the last two term count the numbers of migrations of exposed individuals to and from patch $i$.
The expression of $I^N_i(t)$ is similar, except that the second term subtracts the number of initially infectious individuals in any patch who have recovered before time $t$ in patch $i$, and the fourth and sixth terms subtract the numbers of initially and newly exposed individuals in any patch who have recovered before time $t$ in patch $i$, where the patch in which they became infectious is also tracked.

\begin{proof}
 In the representation of $E^N_i(t)$, we observe that 
\begin{align*}
\sum_{k=1}^{E^N_i (0)} {\bf1}_{t<\eta_{k,i}^0}{\bf1}_{X^{0,k}_i(t)=i}
&= E^N_i(0) - \sum_{k=1}^{E^N_i (0)} {\bf1}_{\eta_{k,i}^0 \le t }{\bf1}_{X^{0,k}_i(\eta_{k,i}^0)=i}  - \sum_{\ell \neq i} V^{N,0}_{i, \ell}(t)
\end{align*}
and for $\ell \neq i$, 
\begin{align*}
\sum_{k=1}^{E^N_{\ell} (0)} {\bf1}_{t<\eta_{k,\ell}^0}{\bf1}_{X^{0,k}_{\ell}(t)=i}
=  V^{N,0}_{\ell, i}(t) -  \sum_{k=1}^{E^N_\ell (0)} {\bf1}_{\eta_{k,\ell}^0 \le t}{\bf1}_{X^{0,k}_\ell(\eta_{k,\ell}^0)=i}
\end{align*}
where  
$ V^{N,0}_{i,\ell}(t) = \sum_{k=1}^{E^N_i(0)} {\bf 1}_{X^{0,k}_i(t\wedge \eta^0_{k,i})=\ell} $ is the number of initially exposed  individuals from patch $i$ that are in patch $\ell$
at the time  $t\wedge \eta^0_{k, i}$ for $k=1,\dots, E^N_i(0)$. 
We also observe that 
\begin{align*}
\sum_{j=1}^{A^N_i(t)} {\bf 1}_{\tau^N_{j,i} + \eta_{j,i} >t}  {\bf1}_{X^j_i(t-\tau^N_{j,i})=i} 
&= A^N_i(t) - \sum_{j=1}^{A^N_i(t)} {\bf 1}_{\tau^N_{j,i} + \eta_{j,i} \le t} {\bf1}_{X^j_i( \eta_{j,i})=i} - \sum_{\ell \neq i} V^N_{i,\ell} (t) 
\end{align*}
and for $\ell \neq i$,
\begin{align*}
 \sum_{j=1}^{A^N_\ell(t)} {\bf 1}_{\tau^N_{j,\ell} + \eta_{j,\ell} >t} {\bf1}_{X^j_\ell(t-\tau^N_{j,\ell})=i}
& =  V^N_{\ell,i}(t) -  \sum_{j=1}^{A^N_\ell(t)} {\bf 1}_{\tau^N_{j,\ell} + \eta_{j,\ell} \le t} {\bf1}_{X^j_\ell(\eta_{j,\ell})=i}
\end{align*}
where $V^N_{i,\ell}(t)  = \sum_{j=1}^{A^N_i(t)} {\bf 1}_{X^j_i((t-\tau^N_{j,i})\wedge \eta_{j,i}=\ell}  $ denotes the number of individuals who became exposed at time $ \tau^N_{j, \ell}\in(0,t)$ in patch $i$, and are in patch $\ell$ at time $(t-\tau^N_{j,i}) \wedge \eta_{j,i}$ for $j =1,\dots, A^N_i(t)$. 


It is clear that 
\begin{align*}
& \sum_{\ell \neq i} V^N_{\ell,i} (t) -\sum_{\ell \neq i} V^N_{i,\ell} (t) + \sum_{\ell \neq i} V^{N,0}_{\ell,i} (t) -\sum_{\ell \neq i} V^{N,0}_{i,\ell} (t)  \\
&= \sum_{\ell\neq i}  P_{E,\ell,i}\left(\nu_{E,\ell,i} \int_0^t E^N_\ell (s)ds\right)  -\sum_{\ell\neq i}P_{E,i,\ell}\left(\nu_{E,i,\ell} \int_0^tE^N_i(s)ds\right) \,.
\end{align*}
Thus, using the above identities, we obtain the expression in \eqref{eqn-E-rep-2-SEIR}. A similar argument gives the expression in \eqref{I-rep-2-SEIR}. 
\end{proof} 

\subsection{Using PRMs in order to represent some terms of the model}
We end this section of presentation of our model with the description of the representations of some of the key components in the dynamics of the above model 
via PRMs. Those will play an important role in the proofs below. The infection process $A^N_\ell$ has the representation \eqref{An-rep-1}, which makes use of the PRM $Q_{\ell}$. 
Define a PRM 
$\check{Q}_{\ell}(ds,d\afrak, du, d\theta)$ on $\R_+^3\times\mathcal{L}$, which is the sum of the Dirac masses at the points  
$(\tau_{j,\ell}^N,\mathfrak{A}^N_{j,\ell}, \eta_{j,\ell}, X_\ell^j(\eta_{j,\ell}))$ with mean measure $ds\times d\afrak\times G(du)\times \mu^X_\ell(u,d\theta)$, where
for each $u>0$, $\mu^X_\ell(u,\{\ell'\})=p_{\ell,\ell'}(u)$, and an infection occurs at time $\tau^N_{j,\ell}$ if and only if $\mathfrak{A}^N_{j,\ell}\le\lambda_\ell\Upsilon^N(\tau^N_{j,\ell})$. 
We can then write for $ \ell, \ell' \in \mathcal{L}$, 
\begin{align} \label{PRM-rep-E-SEIR}
& \sum_{j=1}^{A^N_\ell(t)}  {\bf 1}_{\tau^N_{j,\ell} + \eta_{j,\ell} \le t} {\bf1}_{X^j_\ell( \eta_{j,\ell})=\ell'}\; =\; \int_0^t \int_0^\infty \int_0^{t-s}\int_{\{\ell'\}} {\bf 1}_{\afrak \le \lambda_\ell \Upsilon^N_\ell(s^-)} \check{Q}_{\ell}(ds,d\afrak, du, d\theta)\,. 
\end{align}
We denote the corresponding compensated PRM $\overline{Q}_{\ell}(ds,d\afrak, du, d\theta) = \check{Q}_{\ell}(ds,d\afrak, du, d\theta)- ds\times d\afrak\times G(du)\times \mu^X_\ell(u,d\theta)$ for $\ell,\ell'\in \mathcal{L}$.

Define another PRM 
$\breve{Q}_{\ell}(ds,d\afrak,du, d \theta, dv, d\vartheta)$ 
on $\R_+^4\times \mathcal{L}^2 $, which is the sum of the Dirac masses at the points  
$(\tau_{j,\ell}^N,\mathfrak{A}^N_{j,\ell}, \eta_{j,\ell},    \zeta_{j,\ell},  X^j_\ell(\eta_{j,\ell}),  Y_{\ell'}^j(\zeta_{j,\ell}))$ with mean measure $ds\times d\afrak\times H(du, dv)\times \mu^X_\ell(u, d\theta) \times \mu^Y_\theta(v,d\vartheta)$, where
for each $u>0$, $\mu^X_\ell(u,\{\ell'\})=p_{\ell,\ell'}(u)$, and for each $v>0$, $\mu^Y_\ell(v,\{\ell'\})=q_{\ell,\ell'}(v)$,
 and again an infection occurs at time $\tau^N_{j,\ell}$ if and only if $\mathfrak{A}^N_{j,\ell}\le\lambda_\ell\Upsilon^N((\tau^N_{j,\ell})^-)$. 
We can then write for $ \ell,i \in\LL$, 
\begin{align} \label{PRM-rep-I-SEIR}
& \sum_{j=1}^{A^N_\ell(t)} {\bf 1}_{\tau^N_{j,\ell} + \eta_{j,\ell} \le t}  \left( \sum_{\ell'=1}^L {\bf1}_{X^j_\ell( \eta_{j,\ell})=\ell'}  {\bf 1}_{\tau^N_{j,\ell} + \eta_{j,\ell}  + \zeta_{j,\ell}  \le t} {\bf1}_{Y^{j,\ell}_{\ell'}( \zeta_{j,\ell} )=i}  \right) \non \\
&\;=\; \int_0^t \int_0^\infty \int_0^{t-s} \int_0^{t-s-u}  \int_{\LL}   \int_{\{i\}}  {\bf 1}_{\afrak \le \lambda_\ell \Upsilon^N_\ell(s^-)}  \breve{Q}_{i}(ds,d\afrak, du, dv, d\theta,  d \vartheta)\,. 
\end{align}
We denote the corresponding compensated PRM 
$\wt{Q}_{\ell}(ds,d\afrak,du, dv, d \theta, d\vartheta) = \breve{Q}_{\ell}(ds,d\afrak,du, dv, \\ d \theta,  d\vartheta) 
-  ds\times d\afrak\times H(du,dv)\times \mu^X_\ell(u, d\theta)  \times \mu^Y_\theta(v,d\vartheta)$ for $\ell\in \mathcal{L}$.

\section{Functional limit theorems} \label{sec-theorems}

\subsection{FLLN}

For any process $Z^N$, let $\bar{Z}^N:=N^{-1}Z^N$.   

\begin{assumption} \label{AS-FLLN-SEIR}
There exist constants $0<\S_i(0) \le 1, 0 \le \bar{E}_i(0)<1, 0 \le \I_i(0)<1$ with $\sum_{i=1}^L [\bar{E}_i(0)+\bar{I}_i(0)]>0$ such that  $\sum_{i=1}^L (\S_i(0)+ \bar{E}_i(0)+ \I_i(0)) =1$ and  $(\bar{S}^N_i(0), \bar{E}^N_i(0),  \bar{I}^N_i(0),\,  i\in \mathcal{L})   \to (\bar{S}_i(0), \bar{E}_i(0), \bar{I}_i(0),\, i\in \mathcal{L})$ in probability in $\R^{3L}$ as $N\to\infty$. 

\end{assumption}

The following FLLN is proved in Section \ref{sec-proof-LLN-SEIR}. 

\begin{theorem} \label{thm-FLLN-SEIR} 
Under Assumption  \ref{AS-FLLN-SEIR},  and assume that $F_0$ and $G_0$ are continuous,  
 \begin{equation} \label{eqn-FLLN-conv-SEIR}
 (\S^N_i, \bar{E}^N_i, \I^N_i,\rR^N_i,\, i\in \mathcal{L}) \;\to\; (\S_i,\ \bar{E}_i, \bar{I}_i,\rR_i,\, i\in \mathcal{L}) \qinq D^{4L} \qasq N \to \infty\, , 
 \end{equation}
in probability, locally uniformly  on $[0,T]$, where $(\S_i(t), \bar{E}_i(t), \I_i(t),\rR_i(t),\, i\in \mathcal{L})\in C^{4L}$ is the unique  solution of the following system of  deterministic integral equations: 
\begin{align} \label{barS-SEIR}
\bar{S}_i(t) \;= \; \S_i(0) - \lambda_i \int_0^t \bar\Upsilon_i(s) ds  +\sum_{\ell=1,\ell\neq i}^L \int_0^t (\nu_{S,\ell,i}\S_\ell(s)- \nu_{S,i,\ell}\S_i(s)) ds\,,
\end{align}
\begin{align} \label{barE-SEIR}
\bar{E}_i(t) &\;=\; \bar{E}_i(0)  -   \sum_{\ell=1}^L \bar{E}_\ell(0) \int_0^t p_{\ell,i}(u) d G_0(u)    +\lambda_i \int_0^t \bar\Upsilon_i(s) ds \non\\
& \quad
  - \sum_{\ell=1}^L \lambda_\ell \int_0^t  \int_0^{t-s} p_{\ell,i}(u) d G(u) \bar\Upsilon_\ell(s)ds    +\sum_{\ell=1,\ell\neq i}^L \int_0^t (\nu_{E,\ell,i}\bar{E}_\ell(s)- \nu_{E,i,\ell}\bar{E}_i(s)) ds\,, 
\end{align}
\begin{align}\label{barI-SEIR}
\bar{I}_i(t) &\;=\;  \bar{I}_i(0) -  \sum_{\ell=1}^L \bar{I}_\ell(0) \int_0^t  q_{\ell,i}(s) d F_0(s) +   \sum_{\ell=1}^L \bar{E}_\ell(0) \bigg(\int_0^t p_{\ell,i}(u) d G_0(u)    -  \Phi^0_{\ell,i}(t)\bigg)  \non\\
& \quad + \sum_{\ell=1}^L \lambda_\ell  \bigg( \int_0^t  \int_0^{t-s} p_{\ell,i}(u) d G(u) \bar\Upsilon_\ell(s)ds  -  \int_0^t  \Phi_{\ell,i}(t-s)\bar\Upsilon_\ell(s)  ds \bigg)  \non \\
& \quad +\sum_{\ell=1,\ell\neq i}^L \int_0^t \big(\nu_{I,\ell,i}\bar{I}_\ell(s)- \nu_{I,i,\ell}\bar{I}_i(s)\big) ds\, ,
\end{align}
and
\begin{align} \label{barR-SEIR}
\bar{R}_i(t) &\;=\;   \sum_{\ell=1}^L  \bar{I}_\ell(0) \int_0^t  q_{\ell,i}(s) d F_0(s) +  \sum_{\ell=1}^L \bar{E}_\ell(0)\Phi^0_{\ell,i}(t) +  \sum_{\ell=1}^L \lambda_\ell \int_0^t   \Phi_{\ell,i}(t-s)\bar\Upsilon_\ell(s)  ds \non\\
& \quad +\sum_{\ell=1,\ell\neq i}^L \int_0^t (\nu_{R,\ell,i}\bar{R}_\ell(s)- \nu_{R,i,\ell}\bar{R}_i(s)) ds\,,
\end{align}
with 
\begin{align} \label{eqn-Phi-SEIR}
\bar\Upsilon_i(t)\;:=\;\frac{\S_i(t)\sum_{j=1}^L\kappa_{ij}\I_j(t)}{(\S_i(t)+ \bar{E}_i(t)+\I_i(t)+\rR_i(t))^\gamma}\,,
\end{align}
\begin{equation} \label{eqn-H0}
\Phi^0_{\ell,i}(t)\;:=\;   \int_0^{t}  \sum_{\ell'=1}^L p_{\ell, \ell'}(u) \int_0^{t-u} q_{\ell',i}(v)  H_0(du,dv) \,, 
\end{equation}
and
\begin{equation} \label{eqn-H}
\Phi_{\ell,i}(t)\;:= \;  \int_0^{t} \sum_{\ell'=1}^L p_{\ell, \ell'}(u) \int_0^{t-u} q_{\ell',i}(v) H(du,dv)\,. 
\end{equation}

\end{theorem}

Note that if the exposed and infectious periods are independent for each individual,  we have
\begin{equation} \label{eqn-H0-ind}
\Phi^0_{\ell,i}(t)\;:=\;   \int_0^{t}\left(  \sum_{\ell'=1}^L p_{\ell, \ell'}(u) \int_0^{t-u} q_{\ell',i}(v)  F(dv)  \right) G_0(du)\,, 
\end{equation}
and
\begin{equation} \label{eqn-H-indp}
\Phi_{\ell,i}(t)\;:= \;  \int_0^{t}\left(  \sum_{\ell'=1}^L p_{\ell, \ell'}(u) \int_0^{t-u} q_{\ell',i}(v)  F(dv)  \right)  G(du)\,. 
\end{equation}

\begin{remark} \label{rem-FLLN-det-SEIR}
Suppose the exposed and infectious periods are deterministic, taking values of $t_e>0$ and $t_o>0$. Also, assume that the remaining exposed and infectious periods of the initially exposed and infectious are uniformly distributed over $(0,t_e)$ and $(0,t_o)$, respectively. These are the  corresponding equilibrium distributions of the deterministic ones. Recall that for any c.d.f. $F$ on $\R_+$, the equilibrium distribution $F_e(x) := \int_0^x (1-F(t)) dt/ \int_0^\infty (1-F(t)) dt$ for $x\ge 0$.
 Then the FLLN equations of $\bar{E}_i,\bar{I}_i, \bar{R}_i$ become
\begin{align} 
\bar{E}_i(t) &\;=\; \bar{E}_i(0)  -   \sum_{\ell=1}^L \bar{E}_\ell(0) \frac{1}{t_e}  \int_0^{t\wedge t_e}   p_{\ell,i}(u) du    +\lambda_i \int_0^t \bar\Upsilon_i(s) ds \non\\
& \quad
  - \sum_{\ell=1}^L \lambda_\ell  p_{\ell,i}(t_e) \int_0^{t-t_e}   
 \bar\Upsilon_\ell(s)ds    +\sum_{\ell=1,\ell\neq i}^L \int_0^t (\nu_{E,\ell,i}\bar{E}_\ell(s)- \nu_{E,i,\ell}\bar{E}_i(s)) ds\,, \non 
\end{align}
\begin{align}
\bar{I}_i(t) &\;=\;  \bar{I}_i(0) -  \sum_{\ell=1}^L \bar{I}_\ell(0) \frac{1}{t_o}  \int_0^{t\wedge t_o} q_{\ell,i}(s) d s +   \sum_{\ell=1}^L \bar{E}_\ell(0)  \frac{1}{t_e} \bigg( \int_0^{t\wedge t_e}  p_{\ell,i}(u) d u  -   \int_0^{(t-t_o)\wedge t_e}   \sum_{\ell'=1}^L p_{\ell, \ell'}(s) q_{\ell', i}(t_o)   d s \bigg) \non\\
& \quad  +  \sum_{\ell=1}^L \lambda_\ell \bigg( p_{\ell,i}(t_e)   \int_0^{t-t_e}     
 \bar\Upsilon_\ell(s)ds -  \sum_{\ell'=1}^L p_{\ell, \ell'}(t_e)   
 q_{\ell'i}(t_o)  \int_0^{t-t_e-t_o}  \bar\Upsilon_\ell(s)  ds \bigg)  \non\\
 & \quad  +\sum_{\ell=1,\ell\neq i}^L \int_0^t (\nu_{I,\ell,i}\bar{I}_\ell(s)- \nu_{I,i,\ell}\bar{I}_i(s)) ds \,, \non
\end{align}
and
\begin{align} 
\bar{R}_i(t) &\;=\;   \sum_{\ell=1}^L \bar{I}_\ell(0) \frac{1}{t_o}  \int_0^{t\wedge t_o}   q_{\ell,i}(s) d s + \sum_{\ell=1}^L  \bar{E}_\ell(0)   \frac{1}{t_e}  \int_0^{(t-t_o)\wedge t_e}   \sum_{\ell'=1}^L p_{\ell, \ell'}(s) q_{\ell', i}(t_o)   d s \non \\
& \quad +  \sum_{\ell=1}^L \lambda_\ell \sum_{\ell'=1}^L p_{\ell, \ell'}(t_e)   
 q_{\ell'i}(t_o)  \int_0^{t-t_e-t_o}  \bar\Upsilon_\ell(s)  ds   +\sum_{\ell=1,\ell\neq i}^L \int_0^t (\nu_{R,\ell,i}\bar{R}_\ell(s)- \nu_{R,i,\ell}\bar{R}_i(s)) ds\,.  \non
\end{align}
It is easy to see that we obtain a set of ODEs with delay after taking derivative. 
\end{remark}

\subsection{FCLT}

For a process $Z^N=S^N, E^N, I^N, R^N, \Upsilon^N$, let $\hat{Z}^N:=\sqrt{N} ( \bar{Z}^N - \bar{Z})$ be the diffusion-scaled process where $ \bar{Z}^N:=N^{-1}Z^N$ and $\bar{Z}$ is its limit in probability, see Theorem \ref{thm-FLLN-SEIR}.

\begin{assumption} \label{AS-FCLT-SEIR}
There exist constants $0<\S_i(0) \le 1, 0 \le \bar{E}_i(0)<1, 0 \le \I_i(0)<1$ with $\sum_{i=1}^L [\bar{E}_i(0)+\bar{I}_i(0)]>0$ such that  $\sum_{i=1}^L (\S_i(0)+ \bar{E}_i(0)+ \I_i(0)) =1$,   and random variables $\hat{S}_i(0)$,   $\hat{E}_i(0)$ and $\hat{I}_i(0)$, $i\in \mathcal{L}$, such that 
 $(\hat{S}^N_i(0),\hat{E}^N_i(0),  \hat{I}^N_i(0), \, i\in \mathcal{L}) \RA (\hat{S}_i(0), \hat{E}_i(0), \hat{I}_i(0),\, i\in \mathcal{L}) $ in $\R^{3L}$ as $N\to\infty$. In addition,  $\sup_N\E\big[(\hat{Z}^N(0))^2\big]<\infty$ for $\hat{Z}^N(0) = \hat{S}^N_i(0), \hat{E}^N_i(0), \hat{I}^N_i(0), \hat{R}^N_i(0)$, $i \in \LL$.
\end{assumption}

\begin{theorem} \label{thm-FCLT-SEIR} 
Under Assumption  \ref{AS-FCLT-SEIR}, if $F_0$ and $G_0$ are continuous,  in the two cases (i) $\gamma\in[0,1)$ or (ii)  $\gamma\in[0,1]$ and $\sum_{j\not=i}\kappa_{ij}=0$, 
 \begin{equation} \label{eqn-FCLT-conv-SEIR}
 (\hat{S}^N_i,\hat{E}^N_i,\hat{I}^N_i,\hat{R}^N_i,\, i\in \mathcal{L})\;\to\;(\hat{S}_i(t),\hat{E}_i,\hat{I}_i(t),\hat{R}_i(t),\, i\in \mathcal{L}) \qinq D^{4L} \qasq N \to \infty,
 \end{equation}
where the limit is the unique solution of the following system of stochastic Volterra  integral equations driven by continuous Gaussian processes: 
 \begin{align} \label{eqn-hatS-1-SEIR}
\hat{S}_i(t) &\;=\; \hat{S}_i(0) -  \lambda_i \int_0^t \hat{\Upsilon}_i(s) ds + \sum_{\ell=1, \ell\neq i}^L \int_0^t (\nu_{S,\ell,i}\hat{S}_\ell(s)- \nu_{S,i, \ell}\hat{S}_i(s)) ds  \non\\
& \qquad - \hat{M}_{A,i}(t)  +\sum_{\ell=1, \ell\neq i}^L \big( \hat{M}_{S,\ell,i}(t) - \hat{M}_{S,i,\ell}(t) \big) \,,
\end{align}
\begin{align} \label{eqn-hatE-1-SEIR}
\hat{E}_i(t) &\;=\; \hat{E}_i(0) - \sum_{\ell=1}^L \hat{E}_\ell(0)\int_0^t p_{\ell,i}(s) G_0(ds) + \lambda_i \int_0^t \hat{\Upsilon}_i(s) ds 
-  \sum_{\ell=1}^L \lambda_\ell \int_0^t \int_0^{t-s} p_{\ell,i}(u) G(du)  \hat\Upsilon_\ell(s) ds  \non\\
 & \qquad  + \sum_{\ell=1, \ell\neq i}^L \int_0^t (\nu_{I,\ell,i}\hat{E}_\ell(s)- \nu_{I,i, \ell}\hat{E}_i(s)) ds  - \sum_{\ell=1}^L\big( \hat{E}^{0}_{\ell,i}(t)+ \hat{E}_{\ell,i}(t) \big) \non\\
&\qquad + \hat{M}_{A,i}(t)   +   \sum_{\ell=1, \ell \neq i}^L \big( \hat{M}_{E,\ell,i}(t) - \hat{M}_{E,i,\ell}(t) \big)\,, 
\end{align}
\begin{align} \label{eqn-hatI-1-SEIR}
\hat{I}_i(t)&\;=\; \hat{I}_i(0) -  \sum_{\ell=1}^L\hat{I}_\ell(0) \int_0^t q_{\ell,i}(s) F_0(ds)   + \hat{E}_i(0) \sum_{\ell=1}^L  \left( \int_0^t p_{\ell,i}(s) G_0(ds) - \Phi^0_{\ell,i}(t) \right)   \non\\
& \qquad -  \sum_{\ell=1}^L \lambda_\ell \int_0^t \bigg( \int_0^{t-s} p_{\ell,i}(u) G(du) + \Phi_{\ell,i}(t-s) \bigg)  \hat\Upsilon_\ell(s) ds   \non\\
 & \qquad    +   \sum_{\ell \neq i} \int_0^t \left( \nu_{I,\ell,i}  \hat{I}_\ell(s)-  \nu_{I,i,\ell}\hat{I}_i(s)  \right) ds + \sum_{\ell=1}^L \left( \hat{M}_{I,\ell,i}(t) - \hat{M}_{I,i,\ell}(t) \right) \non\\
& \qquad + \sum_{\ell=1}^{L} \big( \hat{E}^{0}_{\ell,i}(t) +  \hat{E}_{\ell,i}(t)  \big)  - \sum_{\ell=1}^L \left(\hat{I}^{0,1}_{\ell,i}(t)  +\hat{I}^{0,2}_{\ell,i}(t) + \hat{I}_{\ell,i}(t) \right)\,,
\end{align}
\begin{align} \label{eqn-hatR-1-SEIR}
\hat{R}_i(t)&\;=\;   \sum_{\ell=1}^L\hat{I}_\ell(0) \int_0^t q_{\ell,i}(s) F_0(ds)   + \hat{E}^N_i(0) \sum_{\ell=1}^L  \Phi^0_{\ell,i}(t)   \non\\
& \quad+ \sum_{\ell=1}^L \lambda_\ell \int_0^t \bigg( \int_0^{t-s} p_{\ell,i}(u) G(du) + \Phi_{\ell,i}(t-s) \bigg)  \hat\Upsilon_\ell(s) ds  
 +   \sum_{\ell \neq i} \int_0^t \left( \nu_{R,\ell,i}  \hat{R}_\ell(s)-  \nu_{R,i,\ell}\hat{R}_i(s)  \right) ds \non\\
& \quad  + \sum_{\ell=1}^L \left( \hat{M}_{R,\ell,i}(t) - \hat{M}_{R,i,\ell}(t) \right) + \sum_{\ell=1}^L \left(\hat{I}^{0,1}_{\ell,i}(t)  +\hat{I}^{0,2}_{\ell,i}(t) + \hat{I}^N_{\ell,i}(t) \right)\,. 
\end{align}
Here,  with the notation $\bar{I}_{(i)}(t)=\sum_{\ell=1}^L \kappa_{i\ell}\bar{I}_\ell(t)$,
\begin{align} \label{hatPhi-SEIR}
 \hat{\Upsilon}_i(t)
 &\;=\; \frac{1}{( \S_i(t) +\bar{E}_i(t) +\I_i(t)+\rR_i(t))^{(1+\gamma)}} \Big([(1-\gamma) \S_i(t)+\bar{E}_i(t) +\I_i(t)+\rR_i(t)]  \bar{I}_{(i)}(t)\hat{S}_i(t)\non\\
 & \qquad + [\S_i(t)( \S_i(t) +\bar{E}_i(t) +\I_i(t)+\rR_i(t))-\gamma\S_i(t)\bar{I}_{(i)}(t)]\hat{I}_i(t)  -  \gamma \S_i(t)\bar{I}_{(i)}(t)[\hat{E}_i(t)+\hat{R}_i(t)] \Big) \non\\
 & \qquad +\frac{ \bar{S}_i(t)\sum_{j\not=i}\hat{I}_j(t)}{ (\bar{S}_i(t) +\bar{E}_i(t) + \bar{I}_i(t) +\bar{R}_i(t))^\gamma}  \,, 
\end{align}
\begin{align*}
&\hat{M}_{A,i}(t) \;=\; B_{A,i} \left(  \int_0^t  \lambda_i \bar\Upsilon_i(s) ds\right)\,,
\quad
\hat{M}_{S,i,\ell}(t) \;=\; B_{S,i,\ell}\left(\nu_{S,i,\ell} \int_0^t \bar{S}_i(s)ds\right)\,, \\
&
\hat{M}_{E,i,\ell}(t) \;=\; B_{E,i,\ell}\left(\nu_{I,i,\ell} \int_0^t \bar{E}_i(s)ds\right)\,, \quad 
\hat{M}_{I,i,\ell}(t) \;=\; B_{I,i,\ell}\left(\nu_{I,i,\ell} \int_0^t \bar{I}_i(s)ds\right)\,, \\
&\hat{M}_{R,i,\ell}(t) \;=\; B_{R,i,\ell}\left(\nu_{R,i,\ell} \int_0^t \bar{R}_i(s)ds\right)\,, \quad i \neq \ell \,,
\end{align*}
with $B_{A,i}$, $B_{S,i,\ell}$, $B_{E,i,\ell}$,  $B_{I,i,\ell}$, $B_{R,i,\ell}$ being mutually independent standard Brownian motions, and with the deterministic functions $\bar{S}_i, \bar{E}_i, \bar{I}_i, \bar{R}_i$ being the limits in Theorem \ref{thm-FLLN-SEIR}. 
The processes $\big(\hat{E}^{0}_{\ell,i}(t),  \hat{I}^{0,2}_{\ell,i}(t)\big)$,  $\big(\hat{I}^{0,1}_{\ell,i}(t)\big)$,   and $ \big(\hat{E}_{\ell,i}(t), \hat{I}_{\ell,i}(t)\big)$
are   
continuous Gaussian processes, independent of the above Brownian motions,  with mean zero and covariance functions: 
\begin{align*}
\Cov( \hat{E}^{0}_{\ell,i}(t),  \hat{E}^{0}_{\ell',i'}(t') ) & \;=\; \begin{cases}
\bar{E}_\ell(0)  \bigg( \int_0^{t\wedge t'}  p_{\ell,i}(s) G_0(ds) -   \int_0^t  p_{\ell,i}(s) G_0(ds)  \int_0^{t'}  p_{\ell,i}(s) G_0(ds)\bigg)\,,  \\
 \qquad\qquad \qquad \qquad \qquad \qquad \qquad \qquad  \quad   \qforq \ell=\ell',\, i\neq i', \\
- \bar{E}_\ell(0)  \int_0^t  p_{\ell,i}(s) G_0(ds)  \int_0^{t'}  p_{\ell,i}(s) G_0(ds)\,,  \qforq \ell=\ell',\, i\neq i',   \\
 0,\qforq  \ell\neq \ell', \,\\
\end{cases}
\end{align*}
\begin{align*}
\Cov( \hat{I}^{0,1}_{\ell,i}(t),  \hat{I}^{0,1}_{\ell',i'}(t') )  &\;=\;\begin{cases} 
  \I_\ell(0) \bigg( \int_0^{t\wedge t'}  q_{\ell,i}(s) F_0(ds) -  \int_0^t  q_{\ell,i}(s) F_0(ds)  \int_0^{t'}  q_{\ell,i}(s) F_0(ds) \bigg) \,, \\
  \qquad \qquad \qquad \qquad \qquad \qquad \qquad  \qquad \quad   \qforq \ell=\ell',\, i= i', \\
- \I_\ell(0) \int_0^t  q_{\ell,i}(s) F_0(ds)  \int_0^{t'}  q_{\ell,i}(s) F_0(ds)\,,  \qforq \ell=\ell',\, i\neq i', \\
0, \qforq  \ell\neq \ell', \,\\
\end{cases}
\end{align*}
\begin{align*}
& \Cov( \hat{I}^{0,2}_{\ell,i}(t),  \hat{I}^{0,2}_{\ell',i'}(t') )\; =\; \begin{cases}  \bar{E}_\ell(0) \left( \Phi^0_{\ell,i}(t \wedge t') -  \Phi^0_{\ell,i}(t) \Phi^0_{\ell,i}(t')\right)\,,   \qforq \ell=\ell',\, i=i', \\ -\bar{E}_\ell(0)  \Phi^0_{\ell,i}(t) \Phi^0_{\ell,i}(t')\,, \qforq \ell=\ell',\, i\neq i', \\
0,  \qforq \ell\neq \ell',
\end{cases}
\end{align*}
\begin{align*}
\Cov( \hat{E}_{\ell,i}(t),  \hat{E}_{\ell',i'}(t') )&\;=\; \begin{cases}  \lambda_\ell \int_0^{t\wedge t'} \int_0^{t\wedge t'-s} p_{\ell,i}(u) G(du)  \bar\Upsilon_\ell(s) ds,   \qforq  \ell=\ell',\, i= i'\,, \\
0,  \qforq  \ell\neq \ell',  \, \text{and for}\quad \ell=\ell',\, i\neq i',
\end{cases}
\end{align*}
\begin{align*}
\Cov( \hat{I}_{\ell,i}(t),  \hat{I}_{\ell',i'}(t') ) &\;=\; \begin{cases}  \lambda_\ell  \int_0^{t\wedge t'} \Phi_{\ell,i}(t \wedge t'-s)  \bar\Upsilon_\ell(s) ds,  \qforq  \ell=\ell',\, i= i'\,, \\0,   \qforq  \ell\neq \ell',  \, \text{and for}\quad \ell=\ell',\, i\neq i'.  
\end{cases}
\end{align*}
The processes  $\big(\hat{E}^{0}_{\ell,i}(t),  \hat{I}^{0,2}_{\ell,i}(t)\big)$,  $\big(\hat{I}^{0,1}_{\ell,i}(t)\big)$,   and $ \big(\hat{E}_{\ell,i}(t), \hat{I}_{\ell,i}(t)\big)$ are independent from each other, and 
\begin{align*}
\Cov(\hat{E}^{0}_{\ell,i}(t),  \hat{I}^{0,2}_{\ell',i'}(t')\big)&\;=\;\begin{cases} \bar{E}_\ell(0) \bigg( \int_0^t   p_{\ell,i}(u) \int_0^{t'-u} q_{i, i'}(v) H_0(du,dv) -  \int_0^t  p_{\ell,i}(s) G_0(ds)\Phi^0_{\ell,i'}(t')\bigg), \\
\qquad \qquad \qquad \qquad \qquad \qquad \qquad \qquad \qquad  \qquad \qquad  \qforq \ell=\ell', \\
 0,  \qforq  \ell\neq \ell', 
 \end{cases}
\end{align*}
and
\begin{align*}
  \Cov(\hat{E}_{\ell,i}(t), \hat{I}_{\ell',i'}(t')\big) &\;=\; \begin{cases} 
  \lambda_\ell \int_0^{t\wedge t'}  \int_0^{t\wedge t'-s}  p_{\ell,i}(u) \int_0^{t'-s-u} q_{i,i'}(v) H(du, dv)    \bar{\Upsilon}_\ell(s) ds\\
  \quad -  \lambda_\ell \int_0^{t\wedge t'} \int_0^{t\wedge t'-s} p_{\ell,i}(u) G(du) \Phi_{\ell,i}( t'-s)  \bar\Upsilon_\ell(s) ds, \qforq \ell=\ell', \\
   0,\qforq  \ell\neq \ell'. 
   \end{cases}
\end{align*}
\end{theorem}

\begin{remark}
The continuity of $F_0$ and $G_0$ will be important in our proofs of the FLLN and FCLT, in particular for the convergence of the processes associated with the initially exposed and infectious individuals in Lemmas \ref{lem-barEI-0-conv} and  \ref{lem-I0-ij-conv-SEIR}. Note that this assumption is not really restrictive, in the sense that even when $F$ or $G$ is a Dirac measure (deterministic duration), the time in the past when the initially exposed (or infectious) individuals have been infected (or have become infectious) would most naturally be assumed to follow a uniform distribution on some interval dictated by $F$ or $G$ as discussed in Remark \ref{rem-FLLN-det-SEIR}. 
\end{remark}

\begin{remark} \label{rem-FCLT-det-SEIR}
Suppose that the c.d.f.'s $F_0,G_0,F,G$ have the same conditions in Remark \ref{rem-FLLN-det-SEIR}.  Then the limits in Theorem \ref{thm-FCLT-SEIR}  become stochastic differential equations with linear drifts and delay. In particular, 
\begin{align*}
  \sum_{\ell=1}^L \lambda_\ell \int_0^t \int_0^{t-s} p_{\ell,i}(u) G(du)  
  \hat\Upsilon_\ell(s) ds  &\; =\;    \sum_{\ell=1}^L \lambda_\ell  p_{\ell,i}(t_e)  \int_0^{t-t_e}    
  \hat\Upsilon_\ell(s) ds\,, \\
    \sum_{\ell=1}^L \lambda_\ell \int_0^t \Phi_{\ell,i}(t-s)  \hat{\Upsilon}_\ell(s) ds &\;=\; 
     \sum_{\ell=1}^L \lambda_\ell \sum_{\ell'=1}^L p_{\ell, \ell'}(t_e)   
 q_{\ell'i}(t_o)  \int_0^{t-t_e-t_o} \hat{\Upsilon}_\ell(s)  ds\,. 
\end{align*}
\end{remark}

\subsection{On the multi--patch SIR, SIS and SIRS models}  \label{sec-special}
\subsubsection{Multi--patch SIR model }
The multi-patch SEIR model includes the multi-patch SIR model as a special case, without the exposed periods and the associated Markov chain $X$. The infectious process $I^N_i(t)$ becomes 
\begin{align}
I^N_i(t)&\;=\; \sum_{\ell=1}^L \sum_{k=1}^{I^N_\ell(0)} {\bf1}_{\zeta_{k,\ell}^0>t}{\bf1}_{X^{0,k}_\ell(t)=i}
+ \sum_{\ell=1}^L \sum_{j=1}^{A^N_\ell(t)} {\bf 1}_{\tau^N_{j,\ell} + \zeta_{j,\ell} >t} {\bf1}_{X^j_\ell(t-\tau^N_{j,\ell})=i}\,\,, \label{In-1-rep}
\end{align}
which can be also expressed as
\begin{align} \label{I1-rep-2}
I^N_i(t)&\;=\; I^N_i(0)  + A^N_i(t) -   \sum_{\ell=1}^L \sum_{k=1}^{I^N_\ell(0)} {\bf1}_{\zeta_{k,\ell}^0\le t}{\bf1}_{X^{0,k}_\ell(\zeta_{k,\ell}^0)=i}
- \sum_{\ell=1}^L \sum_{j=1}^{A^N_\ell(t)}  {\bf 1}_{\tau^N_{j,\ell} + \zeta_{j,\ell} \le t} {\bf1}_{X^j_\ell( \zeta_{j,\ell})=i} \non
\\& \quad - \sum_{\ell \neq i} P_{I,i,\ell}\left(\nu_{I,i,\ell}\int_0^t I^N_i (s)ds\right)
+\sum_{\ell\neq i} P_{I,\ell,i}\left( \nu_{I,\ell,i} \int_0^t I^N_\ell(s)ds\right) \,,
\end{align}
The process $\Upsilon^N_i(t)$ in \eqref{eqn-Upsilon} becomes
\begin{align*} 
\Upsilon^N_i(t)\;=\;\frac{S^N_i(t)\sum_{\ell=1}^L \kappa_{i\ell}I^N_\ell(t)}{N^{1-\gamma}(S^N_i(t)+I^N_i(t)+R^N_i(t))^\gamma}
\,,\quad i \in \mathcal{L}\,,
\end{align*}

In the FLLN, we obtain 
 \begin{align}
 \S_i(t) &\;=\; \S_i(0) - \lambda_i \int_0^t \bar\Upsilon_i(s) ds  +   \sum_{\ell=1,\ell \neq i}^L\int_0^t \left( \nu_{S,\ell,i}  \S_\ell(s)-  \nu_{S,i,\ell}\S_i(s)  \right) ds\,,  \label{eqn-barS}\\
\I_i(t) &\;=\;  \I_i(0) -  \int_0^t   \sum_{\ell=1}^L \I_\ell(0)  q_{\ell,i}(s) F_0(ds)  + \lambda_i  \int_0^t  \bar\Upsilon_i(s) ds  \non \\
 & \quad  -  \int_0^t  \sum_{\ell=1}^L  \left(  \int_0^{t-s} q_{\ell,i}(u) F(du)  \right)   \lambda_\ell \bar\Upsilon_\ell(s) ds +  \sum_{\ell \neq i} \int_0^t \left( \nu_{I,\ell,i}  \I_\ell(s)-  \nu_{I,i,\ell}\I_i(s)  \right) ds \,, \label{eqn-barI}\\
\rR_i(t) &\;=\; \int_0^t   \sum_{\ell=1}^L \I_\ell(0) q_{\ell,i}(s) F_0(ds)  +  \int_0^t  \sum_{\ell}  \left(  \int_0^{t-s} q_{\ell,i}(u) F(du)  \right)   \lambda_\ell \bar\Upsilon_\ell(s) ds \non\\ 
 & \quad +  \sum_{\ell=1,\ell\neq i}^L  \int_0^t \left( \nu_{R,\ell,i}  \rR_\ell(s)-  \nu_{R,i,\ell}\rR_i(s)  \right) ds\, ,  \label{eqn-barR}
 \end{align}
with $\bar\Upsilon_{i}$ defined by
\begin{align} \label{eqn-Phi}
\bar\Upsilon_i(t)\;=\;\frac{\S_i(t)\sum_{\ell=1}^L\kappa_{i\ell}\I_\ell(t)}{(\S_i(t)+\I_i(t)+\rR_i(t))^\gamma}\,. 
\end{align}

In the FCLT, we obtain 
 \begin{align} \label{eqn-hatS-1}
\hat{S}_i(t) &\;=\; \hat{S}_i(0) -  \lambda_i \int_0^t \hat{\Upsilon}_i(s) ds + \sum_{\ell=1, \ell\neq i}^L \int_0^t (\nu_{S,\ell,i}\hat{S}_\ell(s)- \nu_{S,i, \ell}\hat{S}_i(s)) ds \non\\
& \qquad - \hat{M}_{A,i}(t)  +\sum_{\ell=1, \ell\neq i}^L \big( \hat{M}_{S,\ell,i}(t) - \hat{M}_{S,i,\ell}(t) \big)\,,
\end{align}
\begin{align} \label{eqn-hatI-1}
\hat{I}_i(t) &\;=\;  \hat{I}_i(0) - \sum_{\ell=1}^L \hat{I}_\ell(0)  \int_0^t  q_{\ell,i}(s) F_0(ds) + \lambda_i \int_0^t \hat{\Upsilon}_i(s) ds 
 - \sum_{\ell=1}^L \lambda_\ell \int_0^t \int_0^{t-s} q_{\ell,i}(u) F(du)  \hat\Upsilon_\ell(s) ds \non\\
 & \quad  + \sum_{\ell=1, \ell\neq i}^L \int_0^t (\nu_{I,\ell,i}\hat{I}_\ell(s)- \nu_{I,i, \ell}\hat{I}_i(s)) ds  -\sum_{\ell=1}^L \big( \hat{I}^{0}_{\ell,i}(t) + \hat{I}_{\ell,i}(t) \big)    \non\\
& \quad + \hat{M}_{A,i}(t)+  \sum_{\ell=1, \ell \neq i}^L \big( \hat{M}_{I,\ell,i}(t) - \hat{M}_{I,i,\ell}(t) \big) \,, 
\end{align}
\begin{align} \label{eqn-hatR-1}
\hat{R}_i(t) &\;=\; \sum_{\ell=1}^L \hat{I}_\ell(0) \int_0^t  q_{\ell,i}(s) F_0(ds)   + 
\sum_{\ell=1}^L \lambda_\ell \int_0^t \int_0^{t-s} q_{\ell,i}(u) F(du)  \hat\Upsilon_\ell(s) ds  \non\\
& \quad  +  \sum_{\ell=1, \ell\neq i}^L \int_0^t (\nu_{R,\ell,i}\hat{R}_\ell(s)- \nu_{R,i, \ell}\hat{R}_i(s)) ds  \non\\
& \quad + \sum_{\ell=1}^L \big( \hat{I}^{0}_{\ell,i}(t) + \hat{I}_{\ell,i}(t) \big)  +   \sum_{\ell=1, \ell \neq i}^L \big( \hat{M}_{R,\ell,i}(t) - \hat{M}_{R,i,\ell}(t) \big)\,. 
\end{align}
Here, with the notation $\bar{I}_{(i)}(t)=\sum_{\ell=1}^L \kappa_{i\ell}\bar{I}_\ell(t)$,
\begin{align} \label{eqn-hatPhi-i}
\hat{\Upsilon}_i(t) &=\frac{1}{(\S_i(t)\!+\!\I_i(t)\!+\!\rR_i(t))^{(1+\gamma)}} \Big( [(1-\gamma)\bar{S}_i(t)\!+\!\bar{I}_i(t)\!+\! \bar{R}_i(t)] \bar{I}_{(i)}(t) \hat{S}_i(t)\!  \non\\
& \qquad \qquad +\!
\left[\bar{S}_i(t)(\S_i(t)\!+\!\I_i(t)\!+\!\rR_i(t))\!-\!\gamma\bar{S}_i(t)\bar{I}_{(i)}(t)\right]\hat{I}_i(t) 
\!-\! \gamma\bar{S}_i(t) \bar{I}_{(i)}(t) \hat{R}_i(t) \Big) \non\\
&\quad+\frac{\bar{S}_i(t)\sum_{\ell\not=i}\kappa_{i\ell}\hat{I}_\ell(t)}{(\S_i(t)+\I_i(t)+\rR_i(t))^\gamma}\,, 
\end{align}
\begin{align*}
&\hat{M}_{A,i}(t) \;=\; B_{A,i} \left(  \int_0^t  \lambda_i \bar\Upsilon_i(s) ds\right)\,,
\quad
\hat{M}_{S,i,\ell}(t) \;=\; B_{S,i,\ell}\left(\nu_{S,i,\ell} \int_0^t \bar{S}_i(s)ds\right)\,, \\
&
\hat{M}_{I,i,\ell}(t) \; =\; B_{I,i,\ell}\left(\nu_{I,i,\ell} \int_0^t \bar{I}_i(s)ds\right)\,, \quad 
\hat{M}_{R,i,\ell}(t)\; =\; B_{R,i,\ell}\left(\nu_{R,i,\ell} \int_0^t \bar{R}_i(s)ds\right)\,, \quad i \neq j\,,
\end{align*}
with $B_{A,i}$, $B_{S,i,\ell}$,  $B_{I,i,\ell}$, $B_{R,i,\ell}$ being mutually independent standard Brownian motions, and with the deterministic functions $\bar{S}_i, \bar{I}_i, \bar{R}_i$ given above.  
The processes $ \hat{I}^{0}_{\ell,i}$ and  $\hat{I}_{\ell,i}$ are continuous Gaussian processes with mean zero and covariance functions: 
\begin{align*}
Cov( \hat{I}^0_{\ell,i}(t),  \hat{I}^0_{\ell',i'}(t') ) &\;=\; \begin{cases}
\I_\ell(0)  \Big(\int_0^{t\wedge t'}  q_{\ell,i}(s) F_0(ds) -  \int_0^t  q_{\ell,i}(s) F_0(ds)  \int_0^{t'}  q_{\ell,i}(s) F_0(ds)\Big)\,,   &\text{if}\quad    \ell=\ell', i=i'\,,\\
 -\I_\ell(0)   \int_0^t  q_{\ell,i}(s) F_0(ds)  \int_0^{t'}  q_{\ell,i}(s) F_0(ds)\,,   &\text{if}\quad   \ell=\ell', i\neq i'\,, \\
 0\,, &\text{if}\quad  \ell\neq \ell' ,
\end{cases}
\end{align*}
\begin{align*}
Cov( \hat{I}_{\ell,i}(t),  \hat{I}_{\ell',i'}(t') ) &\;=\; \begin{cases}
\lambda_\ell \int_0^{t\wedge t'} \int_0^{t\wedge t'-s} q_{\ell,i}(u) F(du)  \bar\Upsilon_\ell(s) ds\,,   &\text{if}\quad   \ell=\ell'\,, \, i=i'\,,\\
 0\,,  & otherwise.
\end{cases}
\end{align*}
In addition, $ \hat{I}^{0}_{\ell,i}$ and  $\hat{I}_{\ell,i}$ are independent, and also independent of the Brownian terms.

\subsubsection{Multi--patch SIS model}

The analysis of the multi-patch SIR model can be easily extended to the multi-patch SIS model, where the population  in each patch has susceptible and infectious groups, and when infectious individuals recover, they become susceptible immediately. 
The epidemic evolution dynamics is described as
\begin{align}
S^N_i(t)&\;=\;S^N_i(0)-A^N_i(t)  + \sum_{\ell=1}^L \sum_{k=1}^{I^N_\ell(0)} {\bf1}_{\zeta_{k,\ell}^0\le t}{\bf1}_{X^{0,k}_\ell(\zeta_{k,\ell}^0)=i}
+\sum_{\ell=1}^L \sum_{j=1}^{A^N_\ell(t)}  {\bf 1}_{\tau^N_{j,\ell} + \zeta_{j,\ell} \le t} {\bf1}_{X^j_\ell( \zeta_{j,\ell})=i} \non\\
& \quad - \sum_{\ell=1,\ell \neq i}^L P_{S,i,\ell}\left(\nu_{S,i,\ell}\int_0^t S^N_i (s)ds\right)
+\sum_{\ell=1,\ell\neq i}^L P_{S,\ell,i}\left( \nu_{S,\ell,i} \int_0^t S^N_\ell(s) ds\right) \,,\label{Sn-1-rep-SIS}\\
I^N_i(t)&\;=\; \sum_{\ell=1}^L \sum_{k=1}^{I^N_\ell(0)} {\bf1}_{t<\zeta_{k,\ell}^0}{\bf1}_{X^{0,k}_\ell(t)=i}
+ \sum_{\ell=1}^L \sum_{j=1}^{A^N_\ell(t)} {\bf 1}_{\tau^N_{j,\ell} + \zeta_{j,\ell} >t} {\bf1}_{X^j_\ell(t-\tau^N_{j,\ell})=i} \,\,, \label{In-1-rep-SIS}
\end{align}
where $A^N_i$  is given as in \eqref{An-rep-1} with $
\Upsilon^N_i(t)=\frac{S^N_i(t)\sum_{\ell=1}^L\kappa_{i\ell}I^N_\ell(t)}{(S^N_i(t)+I^N_i(t))^\gamma}$, for $ i \in \mathcal{L}.$
Thus, in the FLLN, we obtain the same limit $\bar{I}_i$ in \eqref{eqn-barI} as in the multi-patch SIR model, and the limit $\bar{S}_i(t)$:
 \begin{align}
 \S_i(t) &\;=\; \S_i(0) - \lambda_i \int_0^t \bar\Upsilon_i(s) ds  \int_0^t   \sum_{\ell} q_{\ell,i}(s) F_0(ds)  +  \int_0^t  \sum_{\ell}  \left(  \int_0^{t-s} q_{\ell,i}(u) F(du)  \right)   \lambda_\ell \bar\Upsilon_\ell(s) ds \non\\  
 & \qquad +   \sum_{\ell=1,\ell \neq i}^L\int_0^t \left( \nu_{S,\ell,i}  \S_j(s)-  \nu_{S,i,\ell}\S_i(s)  \right) ds\,,  \non
 \end{align}
 where $
\bar\Upsilon_i(t):=\frac{\S_i(t)\sum_{\ell=1}\kappa_{i\ell}\I_\ell(t)}{(\S_i(t)+\I_i(t))^\gamma}. $ Similarly in the FCLT, we obtain  the same limit $\hat{I}_i$ as in \eqref{eqn-hatI-1} for the multi-patch SIR model, and the limit $\hat{S}_i(t)$: 
 \begin{align} 
\hat{S}_i(t) &\;=\; \hat{S}_i(0) -  \lambda_i \int_0^t \hat{\Upsilon}_i(s) ds   + \sum_{\ell=1}^L \lambda_\ell \int_0^t \int_0^{t-s} q_{\ell,i}(u) F(du)  \hat\Upsilon_\ell(s) ds  + \sum_{\ell=1}^L \big( \hat{I}^{0}_{\ell,i}(t) + \hat{I}_{\ell,i}(t) \big)  \non\\
& \qquad + \sum_{\ell=1, \ell\neq i}^L \int_0^t (\nu_{S,\ell,i}\hat{S}_\ell(s)- \nu_{S,i, \ell}\hat{S}_i(s)) ds - \hat{M}_{A,i}(t)  +\sum_{\ell=1, \ell\neq i}^L \big( \hat{M}_{S,\ell,i}(t) - \hat{M}_{S,i,\ell}(t) \big)\,, \non 
\end{align}
where 
\begin{align*}\hat{\Upsilon}_i(t) &=\frac{1}{(\bar{S}_i(t)+\bar{I}_i(t))^{(1+\gamma)}} 
\left\{[(1-\gamma)\bar{S}_i(t)+\bar{I}_i(t)]\bar{I}_{(i)}(t) \hat{S}_i(t) + [\bar{S}_i(t)(\bar{S}_i(t)+\bar{I}_i(t))-\bar{S}_i(t) \hat{I}_{(i)}(t)]\hat{I}_i(t)\right\} \\
&\qquad +\frac{\bar{S}_i(t)\sum_{\ell\not=i}\hat{I}_\ell(t) }{ (\S_i(t)+\I_i(t))^\gamma}.
\end{align*}

\subsubsection{Multi-patch SIRS model}

The analysis for the multi-patch SEIR model can be easily extended to multi-patch SIRS model,
where in each patch, the population is grouped into susceptible, infectious, and recovered  individuals
and individuals become susceptible after experiencing a recovery period. In this model, the infectious and recovered processes $I^N_i, R^N_i$ correspond to the exposed and infectious processes $E^N_i, I^N_i$  in the SEIR model. 
In the description of the epidemic dynamics, we need to change the dynamics of $S^N_i$ in \eqref{S-rep-SEIR} 
by adding the individuals that have become susceptible after recovery, i.e., the first three terms in $R^N_i$ in \eqref{R-rep-SEIR}. 
This is similar to the susceptible process $S^N_i$ in \eqref{Sn-1-rep-SIS} for the SIS model. 
Then it is straightforward to write down the limit processes in the FLLN and FCLT for the processes 
$(S^N_i, I^N_i, R^N_i, i \in \mathcal{L})$ (corresponding to $(S^N_i, E^N_i, I^N_i, i \in \mathcal{L})$ in the SEIR model).

\section{Proof of the FLLN} \label{sec-proof-LLN-SEIR}

In this section we prove Theorem \ref{thm-FLLN-SEIR}.
  Here we extend the approach in \cite{PP-2020}. Specifically, we first establish that the process $\{(\bar{A}^N_1, \dots, \bar{A}^N_L): N\ge 1\}$ is tight and then along its convergent subsequence with a given limit, we prove the convergence of $\{(\bar{S}^N_1, \dots, \bar{S}^N_L)\}$, $\{(\bar{E}^N_1, \dots, \bar{E}^N_L)\}$,  $\{(\bar{I}^N_1, \dots, \bar{I}^N_L)\}$ and  $\{(\bar{R}^N_1, \dots, \bar{R}^N_L)\}$. We then identify the limit of  $\{(\bar{A}^N_1, \dots, \bar{A}^N_L)\}$, which allows us to show that the above limits satisfy the system of equations \eqref{barS-SEIR}--\eqref{eqn-Phi-SEIR}. Finally we show that this system of equations  \eqref{barS-SEIR}--\eqref{eqn-Phi-SEIR} has at most one solution, which implies that the whole sequence converges.
    Due to the complications from the migration processes, the proofs of tightness for some key component processes become much more involved, see Lemmas \ref{lem-barEI-0-conv} and \ref{lem-barEI-conv}, and in addition, the formula of $\Upsilon^N_i(t)$ for the infection also brings some new challenges, see Lemma \ref{lem-3.11}. 
Since convergence on $[0,\infty)$ is equivalent to convergence on $[0,T]$ for any $T>0$, it is sufficient to prove convergence on $[0,T]$, with $T$ arbitrary. Hence we fix an arbitrary $T>0$, and study the convergence on $[0,T]$ throughout the section. 

Let us first rewrite the representation \eqref{An-rep-1} of  the processes $A^N_i(t)$ which uses the PRM $Q_{i}(ds, d\afrak)$. 
 Let  $\bar{Q}_{i}(ds, d\afrak)=Q_{i}(ds, d\afrak) - ds d\afrak$ be the compensated PRM. 
Then for each $i \in \mathcal{L}$, 
\begin{align} \label{Ai-rep}
A^N_i(t)
\;=\;\lambda_i\int_0^t\Upsilon^N_i(s)ds+ M^N_{A,i}(t)\,, \quad t \ge 0\,, 
\end{align}
where
\begin{align} \label{M-A-def}
M^N_{A,i}(t) \;:=\; \int_0^t\int_0^\infty{\bf1}_{\afrak\le \lambda_i \Upsilon^N_i(s^-)}\bar{Q}_{i}(ds,d\afrak)\,.
\end{align}
The process $\{M^N_{A,i}(t): t \ge 0\}$ is a square-integrable martingale with respect to the filtration $\{\mathcal{F}^N_{A,i}(t): t \ge 0\}$, defined by
\begin{align*}
\mathcal{F}^N_{A,i}(t) \;:=\; \sigma\big\{S^N_{i}(0), I^N_i(0), i\in \mathcal{L}\big\}\vee \sigma\big\{A^N_i(s):  0 \le s \le t\big\}\,, \quad t \ge 0\,. 
\end{align*}
It has the predictable quadratic variation: 
\begin{align*}
\langle M^N_{A,i} \rangle(t) \;=\; \lambda_i\int_0^t\Upsilon^N_i(s)ds\,, \quad t \ge 0\,. 
\end{align*}

\begin{lemma} \label{lem-barAn-tight}
The sequence $\big\{\big(\bar{A}^N_1,\dots, \bar{A}^N_L\big): N \ge 1\big\}$ is tight in $D^L$.  Each convergent  subsequence of $\big\{\big(\bar{A}^N_1, \dots, \bar{A}^N_L\big)\big\}$ converges in distribution to a limit, 
 denoted as $\big(\bar{A}_1, \dots, \bar{A}_L\big)$, which satisfies 
$$
\bar{A}_i \;=\; \lim_{N\to\infty} \bar{A}^N_i \;=\; \lim_{N\to \infty} N^{-1} \lambda_i\int_0^{\cdot}\Upsilon^N_i(s)ds\,,
$$
and
$$
0 \;\le\; \bar{A}_i(t) - \bar{A}_i(s) \;\le\; \lambda_i \bar\kappa_i (t-s)\,, \qforq 0 < s\le t\,,  \quad w.p.\,1\,.  
$$
\end{lemma}

\begin{proof}
First, since 
\begin{equation} \label{bar-int-Phi-bound}
0 \;\le\; \frac{\lambda_i}{N}\int_s^t\Upsilon^N_i(u)du \;\le\; \lambda_i \bar\kappa_{i}(t-s), \quad \text{w.p.1}\,.  \quad t \ge s \ge 0\,,
\end{equation}
and 
\begin{equation} \label{bar-MA-qv}
\langle \bar{M}^N_{A,i} \rangle(t)  = \frac{\lambda_i}{N^2} \int_0^t  \Upsilon^N_i(u)du,
\end{equation} 
it follows readily from Doob's inequality that, as $N\to\infty$,
\begin{equation} \label{barM-A-conv}
\bar{M}^N_{A,i} \;\to\; 0 \quad \text{in probability, locally uniformly in }t\,.
\end{equation} 
Tightness of $\{(\bar{A}^N_1,\dots,\bar{A}^N_L )\}$ in $D^L$ then follows from  the representation in \eqref{Ai-rep} and the two properties in \eqref{bar-int-Phi-bound} and \eqref{barM-A-conv}. 
\end{proof}

In the following we consider a convergent subsequence of $\{(\bar{A}^N_1,\dots,\bar{A}^N_L )\}$. 
Before we establish the convergence of the $S^N_i$'s, let us establish a simple technical Lemma, which will be useful below.
\begin{lemma}\label{le:EqLin}
Let $A$ be a $d\times d$ matrix. For any $F\in D(\R_+;\R^d)$, the equation
\begin{equation}\label{eq:linear}
U_t=F_t+\int_0^t A U_s ds,\ t\ge0
\end{equation}
has a unique solution  $U \in D(\R_+;\R^d)$, and the mapping
$\AA$ defined by $U:=\AA(F)$ is continuous from $D(\R_+;\R^d)$  equipped with the Skorohod $J_1$ topology, into itself. 
\end{lemma}
\begin{proof}
Let $V_t:=U_t-F_t$. It is easy to verify that $U$ solves \eqref{eq:linear} iff $V$ solves the linear ODE
\begin{equation*}
\frac{dV_t}{dt}=A V_t+A F_t,\ t\ge0;\ V_0=0\,.
\end{equation*}
This ODE has a unique solution, which is given by a
 well--known explicit formula, from which we deduce that
 \[ \AA(F)_t=F_t+\int_0^te^{A(t-s)}AF_s ds\,.\]
 The continuity of the mapping $\AA$ is now clear.
 \end{proof}
 
\begin{lemma} \label{lem-barS-conv}
With the limit $\big(\bar{A}_1, \dots,\bar{A}_L \big)$ of the convergent subsequence of  $\big\{\big(\bar{A}^N_1, \dots,\bar{A}^N_L \big)\big\}$,  under Assumption \ref{AS-FLLN-SEIR},
$$(\bar{S}^N_1,\dots,\bar{S}^N_L ) \;\RA\; (\bar{S}_1,\dots, \bar{S}_L) \qinq D^L \qasq N\to\infty
$$
where the limit $ (\bar{S}_1,\dots, \bar{S}_L)$ is the unique solution to the ODEs: for each $i\in \mathcal{L}$, 
\begin{align}\label{eq:barS}
\bar{S}_i (t)\;=\;\bar{S}_i(0)-\bar{A}_i(t)  +   \sum_{\ell\neq i} \int_0^t \left( \nu_{S,\ell,i}  \S_\ell(s)-  \nu_{S,i,\ell}\S_i(s)  \right) ds\,, \quad t \ge 0\,.
\end{align}
\end{lemma}
\begin{proof}
Starting from \eqref{S-rep-SEIR},  
we can rewrite the processes $\bar{S}^N_i$ as 
\begin{align*}
\bar{S}^N_i(t)&\;=\;\bar{S}^N_i(0) -\bar{A}^N_i(t) + \sum_{\ell\neq i} \int_0^t \left( \nu_{S,\ell,i}  \S^N_\ell(s)-  \nu_{S,i,\ell}\S^N_i(s)  \right) ds +\sum_{\ell\neq i} \big(\bar{M}^N_{S,\ell,i}(t)  - \bar{M}^N_{S,i,\ell}(t) \big)   \,,     
\end{align*}
where 
\begin{align*}
\bar{M}^N_{S,\ell,i}(t) &\;:=\; \frac{1}{N} \left( P_{S,\ell,i}\left(\nu_{S,\ell,i}\int_0^tS^N_\ell(s)ds\right)- \nu_{S,\ell,i}\int_0^t S^N_\ell(s)ds \right)\,. 
\end{align*}
The processes $\bar{M}^N_{S,\ell,i}$  are square integrable martingales with respect to the filtration 
$\{\mathcal{F}^N_{S}(t): t \ge 0\}$, defined by
\begin{align*}
\mathcal{F}^N_{S}(t) &\; :=\; \bigvee_{i=1}^L  \mathcal{F}^N_{A,i}(t) \vee  \sigma \left\{P_{S,\ell,i}\left(\nu_{S, \ell,i}\int_0^s S^N_\ell(u)du\right):  0 \le s \le t, \ell, i\in \mathcal{L}, \, \ell \neq i\right\}\,, \quad t \ge 0\,. 
\end{align*}
They have the predictable quadratic variation: 
\begin{align*}
\langle \bar{M}^N_{S,\ell,i} \rangle(t) &\;=\; \frac{1}{N} \nu_{S,\ell,i}\int_0^t \bar{S}^N_\ell(s)  ds \to 0 \qasq  N \to \infty\,. 
\end{align*}
Thus, as $N\to\infty$, 
$$
\left(\bar{M}^N_{S,\ell,i}(t),\, \ell, i \in \mathcal{L}, \ell\neq i\right) \to 0 \ \text{ locally uniformly in }t\,, 
$$
and under Assumption \ref{AS-FLLN-SEIR}, for $1\le i\le L$, the processes 
$\bar{S}^N_i(0) -\bar{A}^N_i $
jointly converge in distribution
 to $\bar{S}_i(0) -\bar{A}_i$ in $D$. 
Now we exploit the result of Lemma \ref{le:EqLin} with $d=L$ and $A$ is given by 
 $A_{i,j} = \nu_{S,j,i}$ for $j\not=i$ and $A_{i,i}=-\sum_{j\not=i}\nu_{S,i,j}$. With the notations of 
 Lemma \ref{le:EqLin},  the vector 
$\bar{S}^N=\AA\Big(\Big(\bar{S}^N_i(0) -\bar{A}^N_i + \sum_{\ell\neq i}\big(\bar{M}^N_{S,\ell,i} - \bar{M}^N_{S,i,\ell}\big) \Big)_{i=1,\dots,L} \Big)$, while $\bar{S}=\AA\Big(\big(\bar{S}_i(0)-\bar{A}_i \big)_{i=1,\dots,L}\Big)$. 
Hence it follows from the continuous mapping theorem that 
\[ (\bar{S}^N_1,\ldots,\bar{S}^N_L)\Rightarrow(\bar{S}_1,\ldots,\bar{S}_L) \ \text{ in }D^L,\]
where $(\bar{S}_1,\ldots,\bar{S}_L)$ is the unique solution of the system of equations \eqref{eq:barS}.
\end{proof}

For $\ell,i\in \mathcal{L}$, let
\begin{equation} \label{eqn-EI-ij-def-SEIR}
\left.
\begin{aligned}
& E^{N,0}_{\ell,i}(t) \;:=\;  \sum_{k=1}^{E^N_\ell (0)} {\bf1}_{\eta_{k,\ell}^0 \le t}{\bf1}_{X^{0,k}_\ell(\eta_{k,\ell}^0)=i}\,,\\
& E^N_{\ell,i}(t)\;:=\;  \sum_{j=1}^{A^N_\ell(t)} {\bf 1}_{\tau^N_{j,\ell} + \eta_{j,\ell} \le t} {\bf1}_{X^j_\ell(\eta_{j,\ell})=i}\,\,,  \\
&  I^{N,0,1}_{\ell,i}(t) \;:=\; \sum_{k=1}^{I^N_\ell (0)} {\bf1}_{\zeta_{k,\ell}^0 \le t}{\bf1}_{Y^{0,k}_\ell(\zeta_{k,\ell}^0)=i}\,\,,  \\
& I^{N,0,2}_{\ell,i}(t) \;:=\;  \sum_{k=1}^{E^N_\ell (0)} {\bf1}_{\eta_{k,\ell}^0 \le t}\left( \sum_{\ell'=1}^L {\bf1}_{X^{0,k}_\ell(\eta_{k,\ell}^0)=\ell'}  {\bf1}_{\eta_{k,\ell}^0 + \zeta_{-k,\ell} \le t} {\bf1}_{Y^{-k,\ell}_{\ell'}( \zeta_{-k,\ell})=i}\right)\,, \\
& I^{N}_{\ell,i}(t) \;:=\;  \sum_{j=1}^{A^N_\ell(t)} {\bf 1}_{\tau^N_{j,\ell} + \eta_{j,\ell} \le t}  \left( \sum_{\ell'=1}^L {\bf1}_{X^j_\ell( \eta_{j,\ell})=\ell'}  {\bf 1}_{\tau^N_{j,\ell} + \eta_{j,\ell}  + \zeta_{j,\ell}  \le t} {\bf1}_{Y^{j,\ell}_{\ell'}( \zeta_{j,\ell} )=i}  \right)\,. 
\end{aligned}
\right\}
\end{equation}

We first treat the components associated with the initial quantities. 
\begin{lemma} \label{lem-barEI-0-conv}
If Assumption \ref{AS-FLLN-SEIR} holds and $F_0$ and $G_0$ are continuous, then
\begin{align*}
\big(\bar{E}^{N,0}_{\ell,i},\,\I^{N,0,1}_{\ell,i},\,\I^{N,0,2}_{\ell,i}, \, \ell,i\in \mathcal{L}\big) \to \big(\bar{E}^{0}_{\ell,i},\,\I^{0,1}_{\ell,i},\,\I^{0,2}_{\ell,i}, \, \ell,i \in \mathcal{L}\big) 
\quad \text{in} \quad D^{3L^2} \quad \text{as} \quad N\to\infty,
\end{align*}
in probability, uniformly in $t$, 
where for $\ell,i\in \mathcal{L}$ and $t\ge 0$, 
\begin{equation} \label{eqn-barI01-SEIR}
\bar{E}^{0}_{\ell,i}(t) \;:=\; \bar{E}_\ell (0) \int_0^t p_{\ell,i}(s) G_0(ds)\,, \quad
\I^{0,1}_{\ell,i}(t) \;:=\; \I_\ell (0) \int_0^t q_{\ell,i}(s) F_0(ds)\,, 
\end{equation}
and
\begin{equation} \label{eqn-barI02-SEIR}
\I^{0,2}_{\ell,i}(t) \;:=\; \bar{E}_\ell(0) \Phi^0_{\ell,i}(t)\,,
\end{equation}
with $\Phi^0_{\ell,i}(t)$ defined in \eqref{eqn-H0}. 
\end{lemma}

\begin{proof}
We define $\tilde{E}^{N,0}_{\ell,i}$, $\tilde{I}^{N,0,1}_{\ell,i}$ and 
 $\tilde{I}^{N,0,2}_{\ell,i}$ similarly as $\bar{E}^{N,0}_{\ell,i}$, $\I^{N,0,1}_{\ell,i}$ and $\I^{N,0,2}_{\ell,i}$, but with 
  $E^N_\ell(0)$ and $I^N_\ell(0)$ replaced by  $[N\bar{E}_\ell(0)]$ and $[N\bar{I}_\ell(0)]$ respectively. As a consequence of Assumption \ref{AS-FLLN-SEIR}, the differences $\bar{E}^{N,0}_{\ell,i}(t)-\tilde{E}^{N,0}_{\ell,i}(t)$, $\I^{N,0,1}_{\ell,i}(t)-\tilde{I}^{N,0,1}_{\ell,i}(t)$ and $\I^{N,0,2}_{\ell,i}(t)-\tilde{I}^{N,0,2}_{\ell,i}(t)$ are easily shown to tend to $0$ in probability, locally uniformly in $t$, as $N\to\infty$. 
  
 The convergence of $\{\tilde{E}^{N,0}_{\ell,i}\}$ follows the same argument as that of $\{\tilde{I}^{N,0,1}_{\ell,i}\}$, so we establish the convergence of $\{\tilde{I}^{N,0,1}_{\ell,i}\}$.

Note that, since  $\zeta_{k,\ell}^0$ and $Y^{0,k}_\ell$ are independent,
\begin{align*}
\E\left[ {\bf1}_{\zeta_{k,\ell}^0\le t}{\bf1}_{Y^{0,k}_\ell(\zeta_{k,\ell}^0)=i}\right] 
& =  \E\left[ \int_0^t {\bf1}_{Y^{0,k}_\ell(s)=i} d F_0(s) \right]  =  \int_0^t q_{\ell,i}(s)  F_0(ds), 
\end{align*}
where the expectation is taken under the condition that $Y^{0,k}_\ell(0)=\ell$. 
Note that the pairs $(\zeta^0_{k,\ell}, Y^{0,k}_\ell(\cdot))$ are independent over $k$, and have the same distributions. 
Thus, by the LLN of i.i.d. random variables, 
we obtain that for each $t\ge 0$, as $N\to\infty$,
$$
\tilde{I}^{N,0,1}_{\ell,i}(t) \to \bar{I}^{0,1}_{\ell,i}(t) \ \text{ in probability}. 
$$
In order to establish locally uniform convergence in $t$, is suffices to establish tightness in $D$,  which (see the Corollary of Theorem 7.4 in \cite{billingsley1999convergence})
 will follow from the fact that
\begin{equation}\label{crit_cor_bill}
\limsup_{N\to\infty}\sup_{0\le t\le T}\frac{1}{\delta}\P\left(\sup_{0\le u\le \delta}|\tilde{I}^{N,0,1}_{\ell,i}(t+u)-\tilde{I}^{N,0,1}_{\ell,i}(t)|>\eps\right)\to0,\ \text{ as }\delta\to0.
\end{equation}
 By the independence of the pairs
$\{(\zeta_{k,\ell}^0,Y^{0,k}_\ell(\zeta_{k,\ell}^0)),\ k\ge1\}$,
\begin{align*}
&\P\left( \sup_{t  \le s \le t+\delta } \big|\tilde{I}^{N,0,1}_{\ell,i}(s)-\tilde{I}^{N,0,1}_{\ell,i}(t) \big| > \ep \right)  
 = \P\left( N^{-1}\sum_{k=1}^{N \bar{I}_\ell(0)} {\bf1}_{t < \zeta_{k,\ell}^0\le t+\delta}{\bf1}_{Y^{0,k}_\ell(\zeta_{k,1}^0)=i} >\ep  \right)  \\
&\le \P\left( N^{-1}\sum_{k=1}^{N \bar{I}_\ell(0)}\left[ {\bf1}_{t < \zeta_{k,\ell}^0\le t+\delta}{\bf1}_{Y^{0,k}_\ell(\zeta_{k,\ell}^0)=i}-
 \int_t^{t+\delta} q_{\ell,i}(u)F_0(du)\right]>\ep/2  \right) \\
& \qquad \qquad +{\bf1}\bigg\{\int_t^{t+\delta} q_{\ell,i}(u)F_0(du)>\ep/2\bar{I}_1(0)\bigg\}\,. 
\end{align*}
The first term is bounded by
\begin{align*}
& \frac{4}{\ep^2} \E\left[ \left( N^{-1}\sum_{k=1}^{N \bar{I}_\ell(0)}\left[ {\bf1}_{ t < \zeta_{k,\ell}^0\le t+\delta }{\bf1}_{Y^{0,k}_\ell(\zeta_{k,\ell}^0)=i} -\int_t^{t+\delta} q_{\ell,i}(u)F_0(du)\right]\right)^2 \right] \\
&  = \frac{4 \bar{I}_1(0)}{\ep^2 N} \int_t^{t+\delta} q_{\ell,i}(s)  F_0(ds)\left[1-\int_t^{t+\delta} q_{\ell,i}(s) F_0(ds)\right]  \to 0
\end{align*}
as $N\to \infty$, 
so the right hand side converges  as $N\to \infty$ to  $ {\bf1}\big\{\int_t^{t+\delta} q_{\ell,i}(u)F_0(du)>\ep/2\bar{I}_\ell(0)\big\}$ which vanishes 
 for $\delta>0$ small enough, uniformly w.r.t. $t\in[0,T]$. 
 Hence, \eqref{crit_cor_bill} follows.

  We next sketch the proof for the convergence of $\tilde{I}^{N,0,2}_{\ell,i}$ since it follows similar steps as that of $\{\tilde{I}^{N,0,1}_{\ell,i}\}$.  
We have 
\begin{align*}
& \E\left[ {\bf1}_{\eta_{k,\ell}^0 \le t}\left( \sum_{\ell'=1}^L {\bf1}_{X^{0,k}_\ell(\eta_{k,\ell}^0)=\ell'}  {\bf1}_{\eta_{k,\ell}^0 + \zeta_{-k,\ell} \le t} {\bf1}_{Y^{-k,\ell}_{\ell'}( \zeta_{-k,\ell})=i}\right)\right]  \;=\;  \Phi^0_{\ell,i}(t)\,,
\end{align*}
which implies by the LLN of i.i.d. variables that for each $t\ge 0$, 
$
\bar{I}^{N,0,2}_{\ell,i}(t) \RA \bar{I}^{0,2}_{\ell,i}(t)$ as $N\to \infty. 
$
The convergence of finite dimensional distribution is a straightforward extension. 
For tightness we use the same approach as above. 
We start with
\begin{align*}
 \big|\tilde{I}^{N,0,2}_{\ell,i}(t+s)-\tilde{I}^{N,0,2}_{\ell,i}(t) \big|  &=  \frac{1}{N}\sum_{k=1}^{N\bar{E}_\ell (0)} {\bf1}_{t<\eta_{k,\ell}^0 \le t+s}  \left( \sum_{\ell'=1}^L {\bf1}_{X^{0,k}_\ell(\eta_{k,\ell}^0)=\ell'}  {\bf1}_{\eta_{k,\ell}^0 + \zeta_{-k,\ell} \le t+s} {\bf1}_{Y^{-k,\ell}_{\ell'}( \zeta_{-k,\ell})=i}\right) \\ 
& \quad + \frac{1}{N}  \sum_{k=1}^{N\bar{E}_\ell (0)} {\bf1}_{\eta_{k,\ell}^0 \le t}\left( \sum_{\ell'=1}^L {\bf1}_{X^{0,k}_\ell(\eta_{k,\ell}^0)=\ell'}  {\bf1}_{t< \eta_{k,\ell}^0 + \zeta_{-k,\ell} \le t+s} {\bf1}_{Y^{-k,\ell}_{\ell'}( \zeta_{-k,\ell})=i}\right). 
\end{align*}
We next note that each of the two terms on the right hand side is increasing in $s$, so that
\begin{align*}
&\P\left( \sup_{0 \le s \le \delta } \big|\tilde{I}^{N,0,2}_{\ell,i}(t+s)-\tilde{I}^{N,0,2}_{\ell,i}(t) \big| > \ep \right)  \\
& \le  \P \bigg( \frac{1}{N}\sum_{k=1}^{N\bar{E}_\ell (0)} {\bf1}_{t<\eta_{k,\ell}^0 \le t+\delta}  \bigg( \sum_{\ell'=1}^L {\bf1}_{X^{0,k}_\ell(\eta_{k,\ell}^0)=\ell'}  {\bf1}_{\eta_{k,\ell}^0 + \zeta_{-k,\ell} \le t+\delta} {\bf1}_{Y^{-k,\ell}_{\ell'}( \zeta_{-k,\ell})=i}\bigg) > \ep/2 \bigg) \\
& \quad +  \P \bigg(  \frac{1}{N}  \sum_{k=1}^{N\bar{E}_\ell (0)} {\bf1}_{\eta_{k,\ell}^0 \le t}\bigg( \sum_{\ell'=1}^L {\bf1}_{X^{0,k}_\ell(\eta_{k,\ell}^0)=\ell'}  {\bf1}_{t< \eta_{k,\ell}^0 + \zeta_{-k,\ell} \le t+\delta} {\bf1}_{Y^{-k,\ell}_{\ell'}( \zeta_{-k,\ell})=i}\bigg) > \ep/2\bigg) \\
& \le  \P \bigg( \frac{1}{N}\sum_{k=1}^{N\bar{E}_\ell (0)} \bigg[  {\bf1}_{t<\eta_{k,\ell}^0 \le t+\delta}  \bigg( \sum_{\ell'=1}^L {\bf1}_{X^{0,k}_\ell(\eta_{k,\ell}^0)=\ell'}  {\bf1}_{\eta_{k,\ell}^0 + \zeta_{-k,\ell} \le t+\delta} {\bf1}_{Y^{-k,\ell}_{\ell'}( \zeta_{-k,\ell})=i}\bigg) \\
& \qquad \qquad \qquad \qquad\qquad  - \int_t^{t+\delta}   \sum_{\ell'=1}^L p_{\ell, \ell'}(u) \int_0^{t+\delta-u} q_{\ell', i}(v)  H_0(du,dv) \bigg]  > \ep/4 \bigg) \\
& \quad + {\bf1}\bigg\{ \int_t^{t+\delta}   \sum_{\ell'=1}^L p_{\ell, \ell'}(u) \int_0^{t+\delta-u} q_{\ell', i}(v) H_0(du,dv) >\ep/4\bar{E}_\ell(0)\bigg\}\\
& \quad + \P \bigg(  \frac{1}{N}  \sum_{k=1}^{N\bar{E}_\ell (0)} \bigg[ {\bf1}_{\eta_{k,\ell}^0 \le t}\bigg( \sum_{\ell'=1}^L {\bf1}_{X^{0,k}_\ell(\eta_{k,\ell}^0)=\ell'}  {\bf1}_{t< \eta_{k,\ell}^0 + \zeta_{-k,\ell} \le t+\delta} {\bf1}_{Y^{-k,\ell}_{\ell'}( \zeta_{-k,\ell})=i}\bigg)   \\
& \qquad \qquad \qquad \qquad \qquad - \int_0^t  \sum_{\ell'=1}^L p_{\ell, \ell'}(u) \int_{t-u}^{t+\delta-u} q_{\ell', i}(v) H_0(du,dv) \bigg] > \ep/4\bigg)   \\
& \quad +  {\bf1}\bigg\{\int_0^t  \sum_{\ell'=1}^L p_{\ell, \ell'}(u) \int_{t-u}^{t+\delta-u} q_{\ell', i}(v) H_0(du,dv) >\ep/4\bar{E}_\ell(0)\bigg\}\,. 
\end{align*}
The two probability terms on the right hand side are bounded by
\begin{align*}
 \frac{16 \bar{E}_\ell(0)}{\ep^2 N} \int_t^{t+\delta}   \sum_{\ell'=1}^L p_{\ell, \ell'}(u) \int_0^{t+\delta-u} q_{\ell', i}(v) H_0(du,dv)  \bigg( 1-  \int_t^{t+\delta}   \sum_{\ell'=1}^L p_{\ell, \ell'}(u) \int_0^{t+\delta-u} q_{\ell', i}(v) H_0(du,dv) \bigg) 
\end{align*}
and 
\begin{align*}
\frac{16 \bar{E}_\ell(0)}{\ep^2 N}\int_0^t  \sum_{\ell'=1}^L p_{\ell, \ell'}(u) \int_{t-u}^{t+\delta-u} q_{\ell', i}(v) H_0(du,dv)  \bigg(1- \int_0^t  \sum_{\ell'=1}^L p_{\ell, \ell'}(u) \int_{t-u}^{t+\delta-u} q_{\ell', i}(v) H_0(du,dv)   \bigg),
\end{align*} 
respectively. Thus, as $N\to \infty$, the right hand side converges to the sum of the two indicator terms.
And the two terms in the limit both vanish for small enough $\delta$. 
Thus,  for any $\ep>0$, 
\begin{align*}
\lim_{\delta \to 0} \limsup_{N\to\infty} \frac{T}{\delta} \sup_{t \in [0,T]}\P\left( \sup_{t  \le s \le t+\delta } \big|\tilde{I}^{N,0,2}_{\ell,i}(s)-\tilde{I}^{N,0,2}_{\ell,i}(t) \big| > \ep \right) \;\to\; 0\,. 
\end{align*}
Therefore, we can conclude that $\tilde{I}^{N,0,2}_{\ell,i} \to \bar{I}^{0,2}_{\ell,i} $ in probability, uniformly in $t$, as $N\to \infty$.  
Then, since $G_0$ is continuous, we can verify the continuity of the covariance function, and thus the continuity of the limit processes $\bar{I}^{0,2}_{\ell,i}$.
This completes the proof. 
\end{proof}

In the next proof, we will make use of the following result, which is Lemma 4.4 in \cite{FPP2020b}.
 In the next statement, $D_\uparrow(\R_+)$ (resp. $C_\uparrow(\R_+)$) 
 denotes the set of real-valued nondecreasing
 function on $\R_+$, which belong to $D(\R_+)$ (resp. $C(\R_+)$).
\begin{lemma}\label{le:Portmanteau}
Let $f\in D(\R_+)$ and $\{g_N\}_{N\ge1}$ be a sequence of elements of $D_\uparrow(\R_+)$ which is such that
$g_N\to g$ locally uniformly as $N\to\infty$, where $g\in C_\uparrow(\R_+)$. Then, for any $t>0$, as $N\to\infty$,
\[ \int_{[0,t]} f(s) g_N(ds)\to \int_{[0,t]} f(s)g(ds)\,. \]
\end{lemma}

We next prove the convergence of $\bar{E}^N_{\ell,i}$ and $\bar{I}^{N}_{\ell,i}$ for $ \ell,i\in \mathcal{L}$. Define the auxiliary processes: for $\ell,i\in \mathcal{L}$, 
\begin{align*}
\breve{E}^N_{\ell,i}(t) = \E\big[\bar{E}^N_{\ell,i}(t)| \mathcal{F}^N_{A,\ell}(t)\big], \quad \breve{I}^N_{\ell,i}(t) = \E\big[\bar{I}^N_{\ell,i}(t)| \mathcal{F}^N_{A,\ell}(t) \big], \quad t \ge 0.
\end{align*}
We first prove these processes converge to the desired limits in the following lemma, and then show that these processes are asymptotically equivalent to $\bar{E}^N_{\ell,i}$ and $\bar{I}^{N}_{\ell,i}$, $ \ell,i\in \mathcal{L}$. 

\begin{lemma} \label{lem-breve-EI-conv}
With the limit $\big(\bar{A}_1, \dots,\bar{A}_L \big)$ of the convergent subsequence of  $\big\{\big(\bar{A}^N_1, \dots,\bar{A}^N_L \big)\big\}$,  
 under Assumption \ref{AS-FLLN-SEIR}, 
 \begin{align} \label{eqn-barEI-conv}
\big(\breve{E}^N_{\ell,i},\breve{I}^{N}_{\ell,i},\, \ell,i\in \mathcal{L}\big) \to \big(\bar{E}_{\ell,i}, \I_{\ell,i}, \ell,i\in \mathcal{L}\big) \qinq D^{2L^2} \qasq N\to\infty,
\end{align}
in probability, 
where for $\ell,i\in \mathcal{L}$ and $t\ge 0$,
\begin{align} \label{eqn-barE-SEIR}
\bar{E}_{\ell,i}(t) := \int_0^t \left( \int_0^{t-s}p_{\ell,i}(u) G(du) \right) d \bar{A}_\ell(s),
\end{align}
and
\begin{align} \label{eqn-barI-SEIR}
\bar{I}_{\ell,i}(t):=\int_0^t \Phi_{\ell,i}(t-s) d \bar{A}_\ell(s),
\end{align}
with $\Phi_{\ell,i}$ defined in \eqref{eqn-H}. 

\end{lemma} 

\begin{proof}
Observe that for $\ell,i\in \mathcal{L}$, 
 \begin{align} \label{eqn-breveEn-integral-rep}
\breve{E}^{N}_{\ell,i}(t)&= \E\left[\bar{E}^{N}_{\ell,i}(t)|\mathcal{F}^N_{A,\ell}(t) \right] = \frac{1}{N}\sum_{j=1}^{A^N_\ell(t)}  \E\left[ {\bf 1}_{\tau^N_{j,\ell} + \eta_{j,\ell} \le t} {\bf1}_{X^j_\ell( \eta_{j,\ell})=i}| \tau^N_{j,\ell}\right] \non \\
&= \frac{1}{N} \sum_{j=1}^{A^N_\ell(t)} \E\left[  \int_0^{t-\tau^N_{j,\ell}} {\bf1}_{X^j_\ell(u)=i} F(du) \Big| \tau^N_{j,\ell} \right] 
=  \frac{1}{N}\sum_{j=1}^{A^N_\ell(t)}  \int_0^{t-\tau^N_{j,\ell}}p_{\ell,i}(u)  F(du) \non \\
&= \int_0^t \left(  \int_0^{t-s}p_{\ell,i}(u)  F(du) \right) d \bar{A}^N_\ell(s),
\end{align}
and
\begin{align} \label{breveI-rep-SEIR}
\breve{I}^N_{\ell,i}(t)
&=  N^{-1} \sum_{j=1}^{A^N_\ell(t)} \E \left[  {\bf 1}_{\tau^N_{j,\ell} + \eta_{j,\ell} \le t}  \left( \sum_{\ell'=1}^L {\bf1}_{X^j_\ell( \eta_{j,\ell})=\ell'}  {\bf 1}_{\tau^N_{j,\ell} + \eta_{j,\ell}  + \zeta_{j,\ell}  \le t} {\bf1}_{Y^{j,\ell}_{\ell'}( \zeta_{j,\ell} )=i}  \right) \Big| \tau^N_{j,\ell}  \right] \non\\
& = N^{-1} \sum_{j=1}^{A^N_\ell(t)} \Phi_{\ell,i}(t- \tau^N_{j,\ell}) =  \int_0^t \Phi_{\ell,i}(t-s) d \bar{A}^N_\ell(s)\,. 
\end{align}

Then the convergence follows by applying Lemma \ref{le:Portmanteau}.
\end{proof}

\begin{lemma} \label{lem-barbreveEI-0}
Under the assumptions of Lemma \ref{lem-breve-EI-conv}, for $\ell,i\in \mathcal{L}$, and for any $\ep>0$, 
\begin{align} \label{eqn-E-diff-conv-SEIR}
P \left( \sup_{t \in [0,T]} \big|\bar{E}^{N}_{\ell,i}(t) - \breve{E}^{N}_{\ell,i}(t) \big| > \ep\right) \to 0 \quad \text{as} \quad N\to\infty,
\end{align}
and
\begin{align} \label{eqn-I-diff-conv-SEIR}
P \left( \sup_{t \in [0,T]} \big|\I^{N}_{\ell,i}(t) - \breve{I}^{N}_{\ell,i}(t) \big| > \ep\right) \to 0 \quad \text{as} \quad N\to\infty.
\end{align}
\end{lemma}

\begin{proof}
We first consider $\bar{E}^{N}_{\ell,i}$. 
We have
\begin{align*}
\bar{E}^{N}_{\ell,i}(t) - \breve{E}^{N}_{\ell,i}(t)&=  \frac{1}{N}\sum_{j=1}^{A^N_\ell(t)} \chi^N_{j,\ell,i}(t),
\end{align*}
where 
$$
\chi^N_{j,\ell,i}(t) :=  {\bf 1}_{\tau^N_{j,\ell} + \eta_{j,\ell} \le t} {\bf1}_{X^j_\ell( \eta_{j,\ell})=i}  - \int_0^{t-\tau^N_{j,\ell}}p_{\ell,i}(u) G(du). 
$$
Let 
\begin{equation} \label{G-ij-def}
G_{\ell,i}(t) :=  \int_0^{t}p_{\ell,i}(u) G(du), \quad \text{for}\quad \ell,i\in \mathcal{L}.
\end{equation}
Then it is clear that for each $j$, $\E\big[ \chi^N_{j,\ell,i}(t) | \tau^N_{j,\ell} \big]=0$, and 
$
\E\big[ \chi^N_{j,\ell,i}(t)^2 | \tau^N_{j,\ell} \big] = G_{\ell,i}(t-\tau^N_{j,\ell}) (1- G_{\ell,i}(t-\tau^N_{j,\ell})). 
$
And by the independence of the pairs $\big(\eta_{j,\ell}, X^j_\ell(\cdot)\big)$ and $\big(\eta_{j',\ell}, X^{j'}_\ell(\cdot)\big)$, we have  
$\E\big[ \chi^N_{j,\ell,i}(t)  \chi^N_{j',\ell,i}(t)  |  \mathcal{F}^N_{A,\ell}(t) \big]=0$, for $ j\neq j'.$  Thus, we obtain
\begin{align*}
\E\big[ \big(\bar{E}^{N}_{\ell,i}(t) - \breve{E}^{N}_{\ell,i}(t) \big)^2\big| \mathcal{F}^N_{A,\ell}(t) \big] & = \frac{1}{N^2}\sum_{j=1}^{A^N_\ell(t)}   \E\big[ \chi^N_{j,\ell,i}(t)^2 | \tau^N_{j,\ell} \big]  \\
&=  \frac{1}{N} \int_0^t G_{\ell,i}(t-u) (1-G_{\ell,i}(t-u)) d \bar{A}^N_\ell(u) \le \frac{ \bar{A}^N_\ell(t)}{N},\\
\E\big[ \big(\bar{E}^{N}_{\ell,i}(t) - \breve{E}^{N}_{\ell,i}(t) \big)^2\big]&\le \frac{1}{N}\lambda_\ell \bar\kappa_\ell  \, t, 
\end{align*}
which, together with the upper bound in \eqref{bar-int-Phi-bound} for  $\E\big[ \bar{A}^N_\ell(t) \big]$, implies that for any $t>0$ and $\ep>0$,
\begin{align*}
\P \left( \left|\bar{E}^{N}_{\ell,i}(t) - \breve{E}^{N}_{\ell,i}(t)\right|>\ep\right) \le  \frac{1}{N\ep^2} \lambda_\ell \bar\kappa_\ell  \, t   \to 0, \qasq N \to\infty. 
\end{align*}
Next, for $t,u>0$, 
\begin{align} \label{eqn-barIn-diff}
& \big|(\bar{E}_{\ell,i}^N(t+u)- \breve{E}^N_{\ell,i}(t+u)) - (\bar{E}^N_{\ell,i}(t)- \breve{E}^N_{\ell,i}(t)) \big| \non \\
&= \left| \frac{1}{N} \sum_{j=1}^{{A}^N_\ell(t+u)} \chi^N_{j,\ell,i}(t+u) - \frac{1}{N} \sum_{j=1}^{{A}^N_\ell(t)} \chi^N_{j,\ell,i}(t)\right| \non\\
&=  \left| \frac{1}{N} \sum_{j=1}^{{A}^N_\ell(t)} ( \chi^N_{j,\ell,i}(t+u) -  \chi^N_{j,\ell,i}(t)) + \frac{1}{N} \sum_{j= {A}_\ell^N(t)+1}^{{A}^N_\ell(t+u)} \chi^N_{j,\ell,i}(t+u)\right| \non\\
&\le  \frac{1}{N} \sum_{j=1}^{{A}^N_\ell(t)}  {\bf 1}_{t<\tau^N_{j,\ell} + \eta_{j,\ell} \le t+u} {\bf1}_{X^j_\ell( \eta_{j,\ell})=i}  + \int_0^{t} \int_{t-s}^{t+u-s} p_{\ell,i}(v) G(dv) d\bar{A}^N_\ell(s)  \non\\
& \quad + \big| \bar{A}^N_\ell(t+u)- \bar{A}^N_\ell(t) \big|. 
\end{align}
Observe that the three terms on the right hand side are nondecreasing in $u$. Thus we obtain 
\begin{align} \label{eqn-I-11-diff-p1}
&\P \left(\sup_{u \in [0,\delta]} \big|(\bar{E}_{\ell,i}^N(t+u)- \breve{E}^N_{\ell,i}(t+u)) - (\bar{E}^N_{\ell,i}(t)- \breve{E}^N_{\ell,i}(t)) \big| > \ep \right)\non\\
& \le \P \left(\frac{1}{N} \sum_{j=1}^{{A}^N_\ell(t)}  {\bf 1}_{t<\tau^N_{j,\ell} + \eta_{j,\ell} \le t+\delta} {\bf1}_{X^j_\ell( \eta_{j,\ell})=i} > \ep/3 \right)  \non\\
& \quad + \P \left(   \int_0^{t} \int_{t-s}^{t+\delta-s} p_{\ell,i}(v)G(dv) d\bar{A}^N_\ell(s) > \ep/3\right)  + \P \left(  \big| \bar{A}^N_\ell(t+\delta)- \bar{A}^N_\ell(t) \big|> \ep/3 \right). 
\end{align}
Using the PRM $\check{Q}_{\ell}(ds,d\afrak, du, d\theta)$ and its compensated PRM $\overline{Q}_{\ell}(ds,d\afrak, du, d\theta)$, we have 
\begin{align*}
& \E \left[ \left( \frac{1}{N} \sum_{j=1}^{{A}^N_\ell(t)}  {\bf 1}_{t<\tau^N_{j,\ell} + \eta_{j,\ell} \le t+\delta} {\bf1}_{X^j_\ell( \eta_{j,\ell})=i} \right)^2\right] \\
& =  \E \left[ \left( \frac{1}{N} \int_0^t \int_0^\infty \int_{t-s}^{t+\delta-s}\int_{\{i\}} {\bf 1}_{\afrak \le \lambda_\ell \Upsilon^N_\ell(s)} \check{Q}_{\ell}(ds,d\afrak, du, d\theta)  \right)^2\right] \\
& \le  2\E \left[ \left( \frac{1}{N} \int_0^t \int_0^\infty \int_{t-s}^{t+\delta-s}\int_{\{i\}} {\bf 1}_{\afrak \le \lambda_\ell \Upsilon^N_\ell(s)} \overline{Q}_{\ell}(ds,d\afrak, du, d\theta)   
\right)^2\right] \\
& \quad + 2 \E \left[ \left( \int_0^{t} \int_{t-s}^{t+\delta-s} p_{\ell,i}(u) G(du) \lambda_\ell \bar{\Upsilon}^N_\ell(s) ds   \right)^2 \right] \\
& = \frac{2}{N} \E \left[ \int_0^{t} \int_{t-s}^{t+\delta-s} p_{\ell,i}(u) G(du) \lambda_\ell \bar{\Upsilon}^N_\ell(s) ds  \right]  + 2 \E \left[ \left( \int_0^{t} \int_{t-s}^{t+\delta-s} p_{\ell,i}(u) G(du) \lambda_\ell \bar{\Upsilon}^N_\ell(s) ds  \right)^2 \right] \\
& \le \frac{2}{N} \lambda_\ell\bar\kappa_{\ell } \int_0^{t} \int_{t-s}^{t+\delta-s} p_{\ell,i}(u) G(du) ds  + 
2 \left( \lambda_\ell\bar\kappa_{\ell} \int_0^{t} \int_{t-s}^{t+\delta-s} p_{\ell,i}(u) G(du) ds  \right)^2.
\end{align*}
The first term converges to zero as $N\to\infty$, and the second term satisfies
\begin{align} \label{eqn-conv-without-F-condition}
 \frac{1}{\delta} \left( \lambda_\ell\bar\kappa_{\ell}  \int_0^{t} \int_{t-s}^{t+\delta-s} p_{\ell,i}(u) G(du) ds  \right)^2  
 & \le  \frac{1}{\delta} \left( \lambda_\ell  \bar\kappa_{\ell}  \int_0^{t}  (G(t+\delta-s) - G(t-s))  ds  \right)^2  \non \\
&   \le   \frac{1}{\delta} \big(\lambda_\ell\bar\kappa_{\ell} \big)^2 \left(  \int_t^{t+\delta} G(u) du - \int_0^\delta G(u)du  \right)^2 \non \\
&\le \delta  \big(\lambda_\ell \bar\kappa_{\ell}\big)^2 \to 0 \qasq \delta \to 0. 
\end{align}

The second term on the right hand side of \eqref{eqn-I-11-diff-p1} can be treated similarly as the second term right above. 
Now for the third term, using \eqref{Ai-rep},
\begin{align} \label{eqn-barAn-incr-bound}
&\E \left[  \big| \bar{A}^N_\ell(t+\delta)- \bar{A}^N_\ell(t) \big|^2 \right] \le 2\E \left[  \big| \bar{M}^N_{A,\ell}(t+\delta)- \bar{M}^N_{A,\ell}(t) \big|^2 \right]  +2 \E \left[  \Big| \lambda_\ell  N^{-1}\int_t^{t+\delta} \Upsilon^N_\ell (s)ds \Big|^2 \right]\,.
\end{align}
By \eqref{barM-A-conv}, the first term converges to zero as $N\to\infty$. The second term is bounded by  $2 \big(\lambda_\ell \bar\kappa_{\ell}  \delta \big)^2$ by \eqref{bar-int-Phi-bound}. 

Thus, combining the above, we obtain 
\begin{align*}
\lim_{\delta \to 0} \limsup_{N\to\infty} \left[\frac{T}{\delta} \right] \sup_{t \in [0,T]}\P \left(\sup_{u \in [0,\delta]} \big|(\bar{E}_{\ell,i}^N(t+u)- \breve{E}^N_{\ell,i}(t+u)) - (\bar{E}^N_{\ell,i}(t)- \breve{E}^N_{\ell,i}(t)) \big| > \ep \right)  = 0. 
\end{align*}
Therefore, we have proved  \eqref{eqn-E-diff-conv-SEIR} 

\smallskip

We next show \eqref{eqn-I-diff-conv-SEIR}, which 
follows from similar steps as above. We highlight the main differences below. 

For each $t\ge 0$, we have
\begin{align*}
\I^{N}_{\ell,i}(t) - \breve{I}^{N}_{\ell,i}(t) &=  \frac{1}{N}\sum_{j=1}^{A^N_\ell(t)} \chi^N_{j,\ell,i}(t),
\end{align*}
where
\begin{align*}
\chi^N_{j,\ell,i}(t) & :=   {\bf 1}_{\tau^N_{j,\ell} + \eta_{j,\ell} \le t}  \left( \sum_{\ell'=1}^L {\bf1}_{X^j_\ell( \eta_{j,\ell})=\ell'}  {\bf 1}_{\tau^N_{j,\ell} + \eta_{j,\ell}  + \zeta_{j,\ell}  \le t} {\bf1}_{Y^{j,\ell}_{\ell'}( \zeta_{j,\ell} )=i}  \right)  - \Phi_{\ell,i}(t-\tau^N_{j,\ell}).
\end{align*}
It is clear that $\E\big[\chi^N_{j,\ell,i}(t) |  \tau^N_{j,\ell} \big]=0$ and 
$\E\big[\chi^N_{j,\ell,i}(t)^2 |  \tau^N_{j,\ell} \big]= \Phi_{\ell,i}(t- \tau^N_{j,\ell} ) (1-\Phi_{\ell,i}(t- \tau^N_{j,\ell}))$ where $\Phi_{\ell,i}(t)$ is defined in \eqref{eqn-H}. Moreover  $\E\big[\chi^N_{j,\ell,i}(t) \chi^N_{j',\ell,i}(t) |  \mathcal{F}^N_{A,\ell}(t)  \big]=0$ due to the independence of the pairs $\big(\eta_{j,\ell}, \zeta_{j, \ell'}, X^j_\ell(\cdot), Y^{j,\ell}_{\ell'}(\cdot)\big)$ and $\big(\eta_{j',\ell}, \zeta_{j', \ell'}, X^{j'}_\ell(\cdot), Y^{j'}_{\ell'}(\cdot)\big)$.  Thus, we obtain
\begin{align*}
\E\big[ \big(\I^{N}_{\ell,i}(t) - \breve{I}^{N}_{\ell,i}(t) \big)^2\big| \mathcal{F}^N_{A,\ell}(t) \big] & = \frac{1}{N^2}\sum_{j=1}^{A^N_\ell(t)}   \E\big[ \chi^N_{j,\ell,i}(t)^2 | \tau^N_{j,\ell} \big]  \\
&=   \frac{1}{N}\int_0^t \Phi_{\ell,i}(t-u) (1-\Phi_{\ell,i}(t-u)) d \bar{A}^N_\ell(u) 
\le \frac{\lambda_\ell  \bar\kappa_{\ell} t}{N} , 
\end{align*}
which implies that for any $\ep>0$, 
\begin{align*}
\P \left( \big|\I^{N}_{\ell,i}(t) - \breve{I}^{N}_{\ell,i}(t)\big|>\ep\right) \le  \frac{1}{N\ep^2}\bigg( \lambda_\ell  \sum_{\ell'=1}^L \kappa_{\ell, \ell'} \bigg) \, t  \to 0, \qasq N \to\infty. 
\end{align*}
Next, for $t,s>0$, we have
\begin{align*}
& \big|\big(\bar{I}_{\ell,i}^N(t+s)- \breve{I}^N_{\ell,i}(t+s)\big) - \big(\bar{I}^N_{\ell,i}(t)- \breve{I}^N_{\ell,i}(t)\big) \big| \non \\
&=  \left| \frac{1}{N} \sum_{j=1}^{{A}^N_\ell(t)} ( \chi^N_{j,\ell,i}(t+s) -  \chi^N_{j,\ell,i}(t)) + \frac{1}{N} \sum_{j= {A}_\ell^N(t)+1}^{{A}^N_\ell(t+s)} \chi^N_{j,\ell,i}(t+s)\right| \non\\
& \le  \frac{1}{N} \sum_{j=1}^{{A}^N_\ell(t)} 
{\bf 1}_{t< \tau^N_{j,\ell} + \eta_{j,\ell} \le t+s}  \left( \sum_{\ell'=1}^L {\bf1}_{X^j_\ell( \eta_{j,\ell})=\ell'}  {\bf 1}_{\tau^N_{j,\ell} + \eta_{j,\ell}  + \zeta_{j,\ell}  \le t+s} {\bf1}_{Y^{j,\ell}_{\ell'}( \zeta_{j,\ell} )=i}  \right)   \\
& \quad +  \frac{1}{N} \sum_{j=1}^{{A}^N_\ell(t)} 
 {\bf 1}_{\tau^N_{j,\ell} + \eta_{j,\ell} \le t}  \left( \sum_{\ell'=1}^L {\bf1}_{X^j_\ell( \eta_{j,\ell})=\ell'}  {\bf 1}_{t<\tau^N_{j,\ell} + \eta_{j,\ell}  + \zeta_{j,\ell}  \le t+s} {\bf1}_{Y^{j,\ell}_{\ell'}( \zeta_{j,\ell} )=i}  \right)   \\
 & \quad +   \frac{1}{N} \sum_{j=1}^{{A}^N_\ell(t)}   \int_{t-\tau^N_{j,\ell}}^{t+s-\tau^N_{j,\ell}}  \sum_{\ell'=1}^Lp_{\ell, \ell'}(u) 
\int_0^{t+s-\tau^N_{j,\ell} - u}   q_{\ell', i}(v)  H(du, dv) \\
& \quad +  \frac{1}{N} \sum_{j=1}^{{A}^N_\ell(t)}  \int_0^{t-\tau^N_{j,\ell}}  \sum_{\ell'=1}^L p_{\ell, \ell'}(u) 
\int_{t-\tau^N_{j,\ell} - u}^{t+s-\tau^N_{j,\ell} - u}   q_{\ell', i}(v)  H(du,dv)  \\
& \quad + \big| \bar{A}^N_\ell(t+s)- \bar{A}^N_\ell(t) \big|. 
\end{align*}
Observe that each of the five terms on the right hand side is increasing in $s$. Thus, we have
\begin{align}\label{I-diff-inc-p1}
& \P \left(\sup_{s \in [0,\delta]} \big|\big(\bar{I}_{\ell,i}^N(t+s)- \breve{I}^N_{\ell,i}(t+s)\big) - \big(\bar{I}^N_{\ell,i}(t)- \breve{I}^N_{\ell,i}(t)\big) \big|> \ep \right) \non\\
& \le \P\left(  \frac{1}{N} \sum_{j=1}^{{A}^N_\ell(t)} 
{\bf 1}_{t< \tau^N_{j,\ell} + \eta_{j,\ell} \le t+\delta}  \left( \sum_{\ell'=1}^L {\bf1}_{X^j_\ell( \eta_{j,\ell})=\ell'}  {\bf 1}_{\tau^N_{j,\ell} + \eta_{j,\ell}  + \zeta_{j,\ell}  \le t+\delta} {\bf1}_{Y^{j,\ell}_{\ell'}( \zeta_{j,\ell} )=i}  \right)  > \ep/5  \right) \non\\
&\quad + \P\left(  \frac{1}{N} \sum_{j=1}^{{A}^N_\ell(t)} 
 {\bf 1}_{\tau^N_{j,\ell} + \eta_{j,\ell} \le t}  \left( \sum_{\ell'=1}^L {\bf1}_{X^j_\ell( \eta_{j,\ell})=\ell'}  {\bf 1}_{t<\tau^N_{j,\ell} + \eta_{j,\ell}  + \zeta_{j,\ell}  \le t+\delta} {\bf1}_{Y^{j,\ell}_{\ell'}( \zeta_{j,\ell} )=i}  \right) > \ep/5  \right) \non \\
 & \quad + \P\left(  \int_0^t   \int_{t-s}^{t+\delta-s}  \sum_{\ell'=1}^Lp_{\ell, \ell'}(u) 
\int_0^{t+\delta-s - u}   q_{\ell', i}(v) H(du,dv) d \bar{A}^N_\ell(s) > \ep/5\right)\non \\
& \quad + \P\left(   \int_0^t \int_0^{t-s} \sum_{\ell'=1}^L p_{\ell, \ell'}(u) 
\int_{t-s - u}^{t+\delta-s - u}   q_{\ell', i}(v)  H(du,dv)   d \bar{A}^N_\ell(s)> \ep/5\right) \non\\
& \quad + \P\left(    \big| \bar{A}^N_\ell(t+\delta)- \bar{A}^N_\ell(t) \big| > \ep/5\right). 
\end{align}
The last term is treated in the same way as the last term in \eqref{eqn-barIn-diff} using the bound in \eqref{eqn-barAn-incr-bound}. 
For the first two terms, we use the PRM representation in \eqref{PRM-rep-I-SEIR}.
We have 
\begin{align*}
&\E\left[ \left( \frac{1}{N} \sum_{j=1}^{{A}^N_\ell(t)} 
{\bf 1}_{t< \tau^N_{j,\ell} + \eta_{j,\ell} \le t+\delta} \sum_{\ell'=1}^L {\bf1}_{X^j_\ell( \eta_{j,\ell})=\ell'}  {\bf 1}_{\tau^N_{j,\ell} + \eta_{j,\ell}  + \zeta_{j,\ell}  \le t+\delta} {\bf1}_{Y^{j,\ell}_{\ell'}( \zeta_{j,\ell} )=i}  \right)^2 \right] \\
& \le 2 \E \left[ \left( \frac{1}{N}\int_0^t \int_0^\infty \int_{t-s}^{t+\delta-s} \int_0^{t+\delta-s-u} \int_{\LL}     \int_{\{i\}}  {\bf 1}_{\afrak \le \lambda_\ell \Upsilon^N_\ell(s)}  \wt{Q}_{i}(ds,d\afrak, du, dv, d\theta,  d \vartheta) \right)^2 \right] \\
& \quad + 2 \E \left[ \left( \int_0^t  \left( \int_{t-s}^{t+\delta-s}\int_0^{t+\delta-s - u}    \sum_{\ell'=1}^Lp_{\ell, \ell'}(u) 
  q_{\ell', i}(v) H(du,dv)  \right) \lambda_\ell \bar{\Upsilon}^N_\ell(s) ds   \right)^2 \right]\,. 
\end{align*}
The first term is equal to 
\begin{align*}
&  \frac{2}{N}\E \left[  \int_0^t   \left( \int_{t-s}^{t+\delta-s}\int_0^{t+\delta-s - u}    \sum_{\ell'=1}^Lp_{\ell, \ell'}(u) 
  q_{\ell', i}(v) H(du,dv)  \right)  \lambda_\ell \bar{\Upsilon}^N_\ell(s) ds    \right] \\
  & \le    \frac{2\lambda_\ell  \bar\kappa_{\ell}}{N}   \int_0^t   \left( \int_{t-s}^{t+\delta-s}\int_0^{t+\delta-s - u}    \sum_{\ell'=1}^Lp_{\ell, \ell'}(u) 
  q_{\ell', i}(v) H(du,dv)  \right) ds \to 0 \qasq N \to \infty.
\end{align*}
The second term is bounded by 
\begin{align*}
2 \left(\lambda_\ell  \bar \kappa_{\ell}    \int_0^t    \left( \int_{t-s}^{t+\delta-s}\int_0^{t+\delta-s - u}    \sum_{\ell'=1}^Lp_{\ell, \ell'}(u) 
  q_{\ell', i}(v) H(du,dv)  \right) ds   \right)^2.
\end{align*}
Without the constant $2 (\lambda_\ell \bar\kappa_\ell)^2$, it satisfies 
\begin{align*}
& \frac{1}{\delta}  \left(   \int_0^t    \left( \int_{t-s}^{t+\delta-s}\int_0^{t+\delta-s - u}    \sum_{\ell'=1}^Lp_{\ell, \ell'}(u) 
  q_{\ell', i}(v) H(du,dv)  \right)  ds   \right)^2  \\
  & \le  \frac{1}{\delta}  \left(  L  \int_0^t  \int_{t-s}^{t+\delta-s} 
F(t+\delta-s-u|u)  G(du) ds   \right)^2  \\
& \le   \frac{1}{\delta}  \left( L  \int_0^t  \left( G(t+\delta -s) - G(t-s)\right) ds   \right)^2    \to 0 \qasq \delta \to 0, 
\end{align*}
where the last step follows from the same argument as in \eqref{eqn-conv-without-F-condition}.

Similarly, for the second term on the right hand side of \eqref{I-diff-inc-p1}, we have
\begin{align*}
&  \E\left[   \left( \frac{1}{N} \sum_{j=1}^{{A}^N_\ell(t)} 
 {\bf 1}_{\tau^N_{j,\ell} + \eta_{j,\ell} \le t}  \sum_{\ell'=1}^L {\bf1}_{X^j_\ell( \eta_{j,\ell})=\ell'}  {\bf 1}_{t<\tau^N_{j,\ell} + \eta_{j,\ell}  + \zeta_{j,\ell}  \le t+\delta} {\bf1}_{Y^{j,\ell}_{\ell'}( \zeta_{j,\ell} )=i}  \right)^2\right] \\
 & \le  \frac{2\lambda_\ell  \bar\kappa_\ell}{N}   \int_0^t \left( \int_0^{t-s} \int_{t-s - u}^{t+\delta-s - u}   \sum_{\ell'=1}^L p_{\ell, \ell'}(u)   q_{\ell', i}(v) H(du,dv)   \right)    d s \\
& \quad + 2   \left(  \lambda_\ell \bar \kappa_{\ell}  \int_0^t \left( \int_0^{t-s} \int_{t-s - u}^{t+\delta-s - u}   \sum_{\ell'=1}^L p_{\ell, \ell'}(u)   q_{\ell', i}(v) H(du,dv)   \right)   d s\right)^2.
\end{align*}
Here it is clear that the first term converges to zero as $N\to \infty$, and the second term without the constant $2 (\lambda_\ell \bar\kappa_\ell)^2$ satisfies 
\begin{align} \label{I-diff-inc-p1-2}
& \frac{1}{\delta}  \left(    \int_0^t  \left( \int_0^{t-s} \int_{t-s - u}^{t+\delta-s - u}   \sum_{\ell'=1}^L p_{\ell, \ell'}(u)   q_{\ell', i}(v) H(du,dv)   \right)   d s\right)^2 \non \\
& \le  \frac{1}{\delta}    \left(   \int_0^t \int_0^{t-s} 
 (F(t+\delta-s-u|u)  - F(t-s-u|u))     G(du)  d s\right)^2 \non\\
 & = \frac{1}{\delta}    \left(   \int_0^t   \left( \int_\delta^{t+\delta-u} F(s|u)ds - \int_0^{t-u} F(s|u)ds \right) G(du)  \right)^2 \non  \\
 & \le \frac{1}{\delta}    \left(   \int_0^t \int_{t-u}^{t+\delta-u} F(s|u)ds  G(du) - \int_0^t \int_0^\delta F(s|u)ds G(du)   \right)^2 \non   \\
 &   \le   \delta  \to 0, \qasq \delta \to 0.
\end{align}

Now for the third and fourth terms on the right hand side of \eqref{I-diff-inc-p1},
by decomposing  the integral terms with respect to $\bar{A}^N_\ell(s)$  using $\bar{A}^N_\ell(s) = \bar{M}^N_{A,\ell}(s)  + \lambda_\ell\int_0^t \bar{\Upsilon}^N_\ell(s)ds$ and using the martingale property of  $ \bar{M}^N_{A,\ell}$ and  \eqref{bar-int-Phi-bound}--\eqref{bar-MA-qv},  we have 

\begin{align*}
& \E\left[ \left(  \int_0^t    \left( \int_{t-s}^{t+\delta-s}\int_0^{t+\delta-s - u}    \sum_{\ell'=1}^Lp_{\ell, \ell'}(u) 
  q_{\ell', i}(v) H(du,dv)  \right) d \bar{A}^N_\ell(s) \right)^2 \right] \\
& \le   \frac{2}{N}  \lambda_\ell  \bar \kappa_{\ell}  \int_0^t  \left( \int_{t-s}^{t+\delta-s}\int_0^{t+\delta-s - u}    \sum_{\ell'=1}^Lp_{\ell, \ell'}(u) 
  q_{\ell', i}(v) H(du,dv)  \right)^2 ds  \\
& \quad + 2 \left( \lambda_\ell  \bar \kappa_{\ell}  \int_0^t  \left( \int_{t-s}^{t+\delta-s}\int_0^{t+\delta-s - u}    \sum_{\ell'=1}^Lp_{\ell, \ell'}(u) 
  q_{\ell', i}(v) H(du,dv)  \right)  ds   \right)^2\,, 
\end{align*}
and
\begin{align*}
& \E \left[ \left(   \int_0^t  \left( \int_0^{t-s} \int_{t-s - u}^{t+\delta-s - u}   \sum_{\ell'=1}^L p_{\ell, \ell'}(u)   q_{\ell', i}(v) H(du,dv)   \right)  d \bar{A}^N_\ell(s) \right)^2 \right]   \\
& \le   \frac{2}{N}  \lambda_\ell  \bar \kappa_{\ell}\int_0^t  \left( \int_0^{t-s} \int_{t-s - u}^{t+\delta-s - u}   \sum_{\ell'=1}^L p_{\ell, \ell'}(u)   q_{\ell', i}(v) H(du,dv)   \right)^2  d s  \\
& \quad +  2 \left(  \lambda_\ell  \bar\kappa_{\ell} \int_0^t  \left( \int_0^{t-s} \int_{t-s - u}^{t+\delta-s - u}   \sum_{\ell'=1}^L p_{\ell, \ell'}(u)   q_{\ell', i}(v) H(du,dv)   \right)  d s\right)^2.
\end{align*}
The first terms on the right hand sides of both converge to zero as $N \to \infty$, while the second terms are bounded by a constant times $\delta^2$.  
Thus we have shown that 
\begin{align*}
\lim_{\delta \to 0} \limsup_{N\to\infty} \left[\frac{T}{\delta} \right] \sup_{t \in [0,T]}
\P \left(\sup_{s \in [0,\delta]} \big|(\bar{I}_{\ell,i}^N(t+s)- \breve{I}^N_{\ell,i}(t+s)) - (\bar{I}^N_{\ell,i}(t)- \breve{I}^N_{\ell,i}(t)) \big|> \ep \right)  = 0.
\end{align*}
Therefore, we have shown that \eqref{eqn-I-diff-conv-SEIR} holds. 
\end{proof}

By the above two lemmas we have shown the following result. 

\begin{lemma} \label{lem-barEI-conv}
With the limit $\big(\bar{A}_1, \dots,\bar{A}_L \big)$ of the convergent subsequence of  $\big\{\big(\bar{A}^N_1, \dots,\bar{A}^N_L \big)\big\}$,  
 under Assumption \ref{AS-FLLN-SEIR}, 
\begin{align} \label{eqn-barEI-conv}
\big(\bar{E}^N_{\ell,i},\bar{I}^{N}_{\ell,i},\, \ell,i\in \mathcal{L}\big) \Rightarrow \big(\bar{E}_{\ell,i}, \I_{\ell,i}, \ell,i\in \mathcal{L}\big) \qinq D^{2L^2} \qasq N\to\infty,
\end{align}
where $\bar{E}_{\ell,i}$ and $\bar{I}_{\ell,i}$ are given in \eqref{eqn-barE-SEIR} and \eqref{eqn-barI-SEIR}, respectively. 
\end{lemma}

We are now ready to prove the convergence of $\big(\bar{E}^N_i, \bar{I}^N_i, \bar{R}^N_i\big)$. 

\begin{lemma} \label{lem-barIn-conv-SEIR}
With the limit $\big(\bar{A}_1, \dots, \bar{A}_L\big)$ of the convergent subsequence of $\{(\bar{A}^N_1, \dots, \bar{A}^N_L)\}$,  under Assumption \ref{AS-FLLN-SEIR}, 
$$
\big(\bar{E}^N_i, \bar{I}^N_i, \bar{R}^N_i, \, i\in \mathcal{L}\big) 
\RA \big(\bar{E}_i, \bar{I}_i, \bar{R}_i, \, i\in \mathcal{L}\big) \qinq D^{3L} \qasq N\to\infty,
$$
where the limits are  the unique solution to the systems of ODEs: for $i\in \mathcal{L}$,
\begin{align} \label{barE-SEIR-p1}
\bar{E}_i(t) &\;=\; \bar{E}_i(0)  + \bar{A}_i(t) -    \sum_{\ell=1}^{L} \big( \bar{E}^{0}_{\ell,i}(t) +  \bar{E}_{\ell,i}(t)  \big)   +\sum_{\ell=1,\ell\neq i}^L \int_0^t (\nu_{E,\ell,i}\bar{E}_\ell(s)- \nu_{E,i,\ell}\bar{E}_i(s)) ds\,, 
\end{align}
\begin{align}\label{barI-SEIR-p1}
\bar{I}_i(t) &=  \bar{I}_i(0) +   \sum_{\ell=1}^{L} \big( \bar{E}^{0}_{\ell,i}(t) +  \bar{E}_{\ell,i}(t)  \big)  - \sum_{\ell=1}^L \left(\bar{I}^{0,1}_{\ell,i}(t)  +\bar{I}^{0,2}_{\ell,i}(t) + \bar{I}_{\ell,i}(t) \right)  
  \non\\
& \qquad +\sum_{\ell=1,\ell\neq i}^L \int_0^t (\nu_{I,\ell,i}\bar{I}_\ell(s)- \nu_{I,i,\ell}\bar{I}_i(s)) ds ,
\end{align}
\begin{align}\label{barR-SEIR-p1}
\bar{R}_i(t)&=
\sum_{\ell=1}^L \left(\bar{I}^{0,1}_{\ell,i}(t)  +\bar{I}^{0,2}_{\ell,i}(t) + \bar{I}_{\ell,i}(t) \right) + \sum_{\ell \neq i} \int_0^t \left( \nu_{R,\ell,i}  \rR_\ell(s)-  \nu_{R,i,\ell}\rR_i(s)  \right) ds ,
 \end{align}
 with $\bar{E}^{0}_{\ell,i}$, $\I^{0,1}_{\ell,i}$  and $\I^{0,2}_{\ell,i}$  being given in \eqref{eqn-barI01-SEIR} and \eqref{eqn-barI02-SEIR},  and with
$\bar{E}_{\ell,i}$ and $\I_{\ell,i}$ being defined in \eqref{eqn-barE-SEIR} and \eqref{eqn-barI-SEIR}, respectively. 
\end{lemma}

\begin{proof}
The proof for the process $E^N_i(t)$ is similar to that of $I^N_i(t)$, so we focus on $I^N_i(t)$. 
By the representations of $I^N_i(t)$ in \eqref{I-rep-2-SEIR}, we have 
\begin{align} \label{barIi-rep}
\I^N_i(t)&= \I^N_i(0)  + \sum_{\ell=1}^{L} \big( \bar{E}^{N,0}_{\ell,i}(t) +  \bar{E}^{N}_{\ell,i}(t)  \big)  - \sum_{\ell=1}^L \left(\bar{I}^{N,0,1}_{\ell,i}(t)  +\bar{I}^{N,0,2}_{\ell,i}(t) +  \bar{I}^N_{\ell,i}(t) \right) 
  \non\\
 & \qquad  + \sum_{\ell\neq i} \left( \bar{M}^N_{I,\ell,i}(t) - \bar{M}^N_{I,i,\ell}(t) \right)  +   \sum_{\ell \neq i} \int_0^t \left( \nu_{I,\ell,i}  \I^N_\ell(s)-  \nu_{I,i,\ell}\I^N_i(s)  \right) ds, 
\end{align}
where 
 for $\ell\neq i$,
\begin{align} \label{eqn-barM-I-def}
\bar{M}^N_{I,\ell,i} (t) &=  \frac{1}{N} \left(P_{I,\ell,i}\left(\nu_{I, \ell,i} \int_0^t I^N_\ell(s)ds\right) - \nu_{I,\ell,i}\int_0^t I^N_\ell (s)ds \right). 
\end{align}
As in the proof of Lemma \ref{lem-barS-conv} for the convergence of $\big(\bar{M}^N_{S,\ell,i},\, \ell, i\in \mathcal{L},\ell\neq i\big) $, we obtain that for any $\ell, i\in \mathcal{L},\ell\neq i$, as $N\to\infty$,
\begin{equation} \label{eqn-barM-I-conv}
\bar{M}^N_{I,\ell,i}(t) \to 0  \ \text{ in probability, locally uniformly in }t. 
\end{equation}
From Assumption \ref{AS-FLLN-SEIR} and Lemmas \ref{lem-barEI-0-conv} and \ref{lem-barEI-conv}, we obtain 
\begin{align*}
 &  \I^N_i(0)  + \sum_{\ell=1}^{L} \big( \bar{E}^{N,0}_{\ell,i}(t) +  \bar{E}^{N}_{\ell,i}(t)  \big)  - \sum_{\ell=1}^L \left(\bar{I}^{N,0,1}_{\ell,i}(t)  +\bar{I}^{N,0,2}_{\ell,i}(t) +  \bar{I}^N_{\ell,i}(t) \right) \\
  &\xrightarrow{N\to\infty}   \I_i(0)  + \sum_{\ell=1}^{L} \big( \bar{E}^{0}_{\ell,i}(t) +  \bar{E}_{\ell,i}(t)  \big)  - \sum_{\ell=1}^L \left(\bar{I}^{0,1}_{\ell,i}(t)  +\bar{I}^{0,2}_{\ell,i}(t) +  \bar{I}_{\ell,i}(t) \right) 
\end{align*}
in probability, locally uniformly in $t$. 
Hence, it follows from \eqref{barIi-rep},  Lemma \ref{lem-barEI-conv}, Lemma \ref{le:EqLin} again with $d=L$ and $A$ satisfying $A_{i,j} = \nu_{I,j,i}$ for $j\not=i$ and $A_{i,i}=-\sum_{j\not=i}\nu_{I,i,j}$ and the continuous mapping theorem that, along any subsequence along which  $(\bar{A}^N_1, \dots, \bar{A}^N_L)\Rightarrow(\bar{A}_1, \dots, \bar{A}_L)$ in $D^L$, $\big(\bar{I}^N_1, \dots, \bar{I}^N_L\big) 
\RA \big(\bar{I}_1,\dots, \bar{I}_L\big)$, where the limit is the unique solution of the limiting system of equations.

We next prove the convergence of $(\rR^N_1,\dots, \rR^N_L)$. Similar to \eqref{I-rep-2-SEIR},  we obtain the following representations for the process $R^N_i(t)$: 
\begin{align} \label{R-rep-2-SEIR}
R^N_i(t) &=   \sum_{\ell=1}^L \sum_{k=1}^{I^N_\ell (0)} {\bf1}_{\zeta_{k,\ell}^0 \le t}{\bf1}_{Y^{0,k}_\ell(\zeta_{k,\ell}^0)=i}  +\sum_{\ell=1}^L \sum_{k=1}^{E^N_\ell (0)} {\bf1}_{\eta_{k,\ell}^0 \le t}\left( \sum_{\ell'=1}^L {\bf1}_{X^{0,k}_\ell(\eta_{k,\ell}^0)=\ell'}  {\bf1}_{\eta_{k,\ell}^0 + \zeta_{-k,\ell} \le t} {\bf1}_{Y^{-k,\ell}_{\ell'}( \zeta_{-k,\ell})=i}\right)  \non \\
& \quad + \sum_{\ell=1}^L \sum_{j=1}^{A^N_\ell(t)} {\bf 1}_{\tau^N_{j,\ell} + \eta_{j,\ell} \le t}  \left( \sum_{\ell'=1}^L {\bf1}_{X^j_\ell( \eta_{j,\ell})=\ell'}  {\bf 1}_{\tau^N_{j,\ell} + \eta_{j,\ell}  + \zeta_{j,\ell}  \le t} {\bf1}_{Y^{j,\ell}_{\ell'}( \zeta_{j,\ell} )=i}  \right) \non  \\
& \quad + \sum_{\ell \neq i} P_{R,\ell,i}\left(\nu_{R,\ell,i} \int_0^t R^N_\ell (s)ds\right)  -\sum_{\ell\neq i}P_{R,i,\ell}\left(\nu_{R,i,\ell} \int_0^t R^N_i(s)ds\right).  
\end{align}

Thus,  we can represent the processes $\rR^N_i(t)$ by
\begin{align} \label{barR1-rep-SEIR}
\rR^N_i(t)&= \sum_{\ell=1}^L \left(\bar{I}^{N,0,1}_{\ell,i}(t)  +\bar{I}^{N,0,2}_{\ell,i}(t) + \bar{I}^N_{\ell,i}(t) \right)
  + \sum_{\ell\neq i} \left( \bar{M}^N_{R,\ell,i}(t) - \bar{M}^N_{R,i,\ell}(t) \right)  
    \non\\
 & \qquad +   \sum_{\ell \neq i} \int_0^t \left( \nu_{R,\ell,i}  \rR^N_\ell(s)-  \nu_{R,i,\ell}\rR^N_i(s)  \right) ds, 
\end{align}
where 
for $\ell\neq i$,
\begin{align} \label{eqn-barM-R-def}
\bar{M}^N_{R,\ell,i} (t) &=  \frac{1}{N} \left(P_{R,\ell,i}\left(\nu_{R, \ell,i} \int_0^t R^N_\ell(s)ds\right) - \nu_{R,\ell,i}\int_0^t R^N_\ell (s)ds \right). 
\end{align}
Arguments very similar to those in the above proof allow us to conclude.
\end{proof}

%



From the above arguments, since we have the joint convergence in  Lemmas \ref{lem-barEI-0-conv} and \ref{lem-barEI-conv}, we can conclude the joint convergence of $(\S^N_i, \bar{E}^N_i, \I^N_i, \rR^N_i, \, i\in \mathcal{L})$. However, we 
have not yet quite explicited the limiting equations, since we have not expressed $\bar{A}_i(t)$ in terms of 
$(\bar{S}_i(t), \bar{E}_i(t), \bar{I}_i(t),\bar{R}_i(t), \\ \bar{I}_\ell(t),\ell\not=i)$. It seems easy to do that since we know that
\[ \bar\Upsilon^N_i(t)=
\frac{\bar{S}^N_i(t)\sum_{\ell=1}^L\kappa_{i\ell}\bar{I}^N_\ell(t)}{(\bar{S}^N_i(t)+ \bar{E}^N_i(t) +\bar{I}^N_i(t)+\bar{R}^N_i(t))^\gamma}\,.\]
However, there is a difficulty in the case where both $\gamma=1$ and $\sum_{\ell\not=i}\kappa_{i\ell}\not=0$.
Define the function $\psi(s,e,i,r,u)=\frac{s(i+u)}{(s+e+i+r)^\gamma}$ on $[0,1]^4\times[0,\bar\kappa]$. If either $0\le\gamma<1$ or $\sup_i\sum_{\ell\not=i}\kappa_{i\ell}=0$, $\psi$ is continuous. However, if both $\gamma=1$ and 
$\sup_i\sum_{j\not=i}\kappa_{ij}>0$, then $\psi$ is not continuous at any point of the form $(0,0,0,0,u)$, with $u>0$. Hence, if we want to include that case in our model, we need to prove that for any $i \in \LL$ and $T>0$, 
$\inf_{0\le t\le T}(\bar{S}_i(t)+ \bar{E}_i(t) +\bar{I}_i(t)+\bar{R}_i(t))>0$. Fortunately, we can prove such an estimate, although we do not have yet established the exact system of equations of the $(\bar{S}_i(t), \bar{E}_i(t), \bar{I}_i(t),\bar{R}_i(t))$, $\, i\in \LL$. 

\begin{lemma}\label{le:lowbd}
Let $(\bar{S}_i(t),\bar{E}_i(t), \bar{I}_i(t),\bar{R}_i(t), \, i\in \LL)$, $0\le t\le T$,  be any weak limit as $N\to\infty$ of 
$(\bar{S}^N_i(t), \bar{E}^N_i(t), \bar{I}^N_i(t),\bar{R}^N_i(t), \, i\in \LL)$. For any $\, i\in \LL$ and $T>0$, there exists a constant $C_{i,T}>0$
which is such that for $0\le t\le T$,
\begin{equation}\label{eq:lowbd} 
\bar{S}_i(t)+\bar{E}_i(t) +\bar{I}_i(t)+\bar{R}_i(t)\ge C_{i,T}\,.
\end{equation}
\end{lemma}
\begin{proof}
Let $\bar{U}_i(t):=\bar{S}_i(t)+ \bar{E}_i(t) +\bar{I}_i(t)+\bar{R}_i(t)$ for $ i\in \LL$ and $0\le t\le T$. For any $ i,\ell\in \LL$, let 
\[ \bar{\nu}_{i,\ell}=\begin{cases}\nu_{S,i,\ell}\vee\nu_{E,i,\ell} \vee\nu_{I,i,\ell}\vee\nu_{R,i,\ell},&\text{if $i\not=\ell$},\\
                                                  0,& \text{if $i=\ell$}\, .
                                                  \end{cases}
                                                  \]
We know that $(\bar{S}_i(t),\bar{E}_i(t), \bar{I}_i(t),\bar{R}_i(t))$
is a solution of  \eqref{eq:barS}, \eqref{barE-SEIR-p1}, \eqref{barI-SEIR-p1} and \eqref{barR-SEIR-p1}. Then we have
\begin{equation}\label{eq:Ubar}
 \bar{U}_i(t)=\bar{U}_i(0)- \int_0^t\Big(\sum_{\ell=1}^L\bar{\nu}_{i,\ell}\Big)\bar{U}_i(s)ds+\int_0^t \bar{V}_i(s)ds,
 \end{equation}
where 
\begin{align*}
\bar{V}_i(t)&=\sum_{\ell\not=i}\left(\nu_{S,\ell,i}\bar{S}_\ell(t)+\nu_{E,\ell,i}\bar{E}_\ell(t) +\nu_{I,\ell,i}\bar{I}_\ell(t)+\nu_{R,\ell,i}\bar{R}_\ell(t)\right)\\&\quad+
\sum_{\ell\not=i}\left([\bar{\nu}_{i,\ell}-\nu_{S,i,\ell}]\bar{S}_i(t)+[\bar{\nu}_{i,\ell}-\nu_{E,i,\ell}]\bar{E}_i(t) +[\bar{\nu}_{i,\ell}-\nu_{I,i,\ell}]\bar{I}_i(t)+[\bar{\nu}_{i,\ell}-\nu_{R,i,\ell}]\bar{R}_i(t)\right)\,.
\end{align*}
Differentiating equation \eqref{eq:Ubar} and exploiting the inequality $\bar{V}_i(t)\ge0$ for all $t\ge0$, we deduce that 
\[ \bar{U}_i(t)\ge \bar{U}_i(0)e^{-(\sum_{\ell=1}^L\bar{\nu}_{i,\ell})t}\ge\bar{U}_i(0)e^{-(\sum_{\ell=1}^L\bar{\nu}_{i,\ell})T}:=C_{i,T},\]
for any $0\le t\le T$.
\end{proof}
We can now explicit the processes $\bar{A}_i(t)$.
\begin{lemma} \label{lem-3.11}
Let $(\bar{A}_i(t),\bar{S}_i(t),\bar{E}_i(t), \bar{I}_i(t),\bar{R}_i(t),\, i\in \LL)$, $0\le t\le T$ be any weak limit as $N\to\infty$ of 
$(\bar{A}_i^N(t),\bar{S}^N_i(t), \bar{E}^N_i(t), \bar{I}^N_i(t),\bar{R}^N_i(t),\, i\in \LL)$. Then for any $ i\in \LL$ and $0 \le t \le T$,
\[\bar{A}_i(t)= \lambda_i \int_0^t\bar\Upsilon_i(s)ds,\]
where
\[\bar\Upsilon_i(t):=\frac{\S_i(t)\sum_{\ell=1}^L\kappa_{i\ell}\I_\ell(t)}{(\S_i(t)+\bar{E}(t) +\I_i(t)+\rR_i(t))^\gamma}\,.\]
\end{lemma}
\begin{proof}
Let $\mathcal{D}^4$ denote the set of c\`adl\`ag functions from $\R_+$ into
$[0,1]^4\times\big[0,\bar\kappa\big]$.
 For any $\gamma \in [0,1]$, the function $\psi(s,e,i,r,u)=\frac{s(i+u)}{(s+e+i+r)^\gamma}$
is  continuous for the Skorohod topology, on the subset of $\mathcal{D}^4$ which is such that for any $T>0$,
$\inf_{0\le t\le T} \{s(t)+e(t)+i(t)+r(t)\}  >0$. 
Thus, we deduce from the joint convergence
\[ (\bar{A}_i^N,\bar{S}^N_i,\bar{E}^N_i, \bar{I}^N_i,\bar{R}^N_i,\, i\in \LL)\Rightarrow(\bar{A}_i,\bar{S}_i, \bar{E}_i, \bar{I}_i,\bar{R}_i, \, i\in \LL)\]
and Lemma \ref{le:lowbd} that
\begin{align*}
\bar\Upsilon_i^N&\Rightarrow\bar\Upsilon_i:=\frac{\S_i\sum_{\ell=1}^L\kappa_{i\ell}\I_\ell}{(\S_i+\bar{E}_i+\I_i+\rR_i)^\gamma} \qinq D \qasq N \to \infty. 
\end{align*}
Consequently, given Lemma \ref{lem-barAn-tight}, 
$$
(\bar{A}^N_1,\dots, \bar{A}^N_L) \RA  (\bar{A}_1, \dots,\bar{A}_L)= \left( \lambda_1\int_0^\cdot \bar\Upsilon_1(s)ds, \dots,\lambda_L\int_0^\cdot \bar\Upsilon_L(s)ds \right)\qinq D^L \qasq N \to\infty.
$$
Therefore, any limit satisfies the system of integral equations given in Theorem \ref{thm-FLLN-SEIR}. 
\end{proof}

The uniqueness of solutions to the set of integral equations in Theorem \ref{thm-FLLN-SEIR} follows from the 
next Lemma, from which we deduce that the whole sequence converges, and since the limit is deterministic, the convergence is in probability.
This completes the proof of Theorem \ref{thm-FLLN-SEIR}. 

\begin{lemma}\label{le:uniq}
The system of equations \eqref{barS-SEIR}--\eqref{eqn-Phi-SEIR}
 has at most one solution. 
\end{lemma}
\begin{proof}
In view of Lemma \ref{le:lowbd},
 if we take the difference between two solutions, any convex combination of those two solutions satisfies the  lower bound \eqref{eq:lowbd}. Since  
 \[\bar\Upsilon_i(t) =\psi\bigg(\bar{S}_i(t), \bar{E}_i(t), \bar{I}_i(t),\bar{R}_i(t),\sum_{\ell\not=i}\kappa_{i\ell}\bar{I}_\ell(t)\bigg)\,,\]
and at each time $t\in[0,T]$, the derivatives of $\psi$ with respect to each of its variables is bounded in absolute value by the supremum of $1$ and
$\bar\kappa_i \bar{U}_i^{-\gamma}(t)\le \bar\kappa_i C_{i,T}^{-\gamma}$, we can now apply a standard argument based upon Gronwall's Lemma in order to deduce uniqueness. 
\end{proof}

\section{Proof of the FCLT}
\label{sec-proof-FCLT-SEIR}

In this section, we prove Theorem \ref{thm-FCLT-SEIR}.  We thus generalize the approach in \cite{PP-2020} to the multi-patch model, by employing the standard technique of convergence of finite dimensional distributions and tightness as exposed in \cite{billingsley1999convergence}. The migration processes require subtle care in proving the finite dimensional distribution convergence as shown in Lemmas \ref{lem-I0-ij-conv-SEIR} and \ref{lem-I-ij-conv-SEIR} for the key components. The tightness proofs require a moment estimate for the supremum of the processes, which is challenging due to the formula $\Upsilon^N(t)$ for infection, see Lemmas \ref{estimhatUps} and \ref{lem-estimate-SIR}.

We first give the following representations of the diffusion-scaled processes. 
The process $\hat{A}^N_i(t)$ can be decomposed as: 
\begin{align} \label{hatA-rep}
\hat{A}^N_i(t) &= \lambda_i \int_0^t \hat{\Upsilon}_i^N(s) ds + \hat{M}^N_{A,i}(t), \quad t\ge 0, 
\end{align}
where
\begin{align} \label{hatPhi-rep}
 \hat{\Upsilon}_i^N(t) &= \sqrt{N} (\bar\Upsilon_i^N(t) -\bar{ \Upsilon}_i(t) ),
\end{align}
and 
\begin{align} \label{hatM-A-def}
\hat{M}^N_{A,i}(t) := \frac{1}{\sqrt{N}} \int_0^t\int_0^\infty{\bf1}_{\afrak\le \lambda_i \Upsilon^N_i(s^-)}\bar{Q}_{i}(ds,d\afrak). 
\end{align}


For the process $\hat{S}_i^N(t)$,  we have
\begin{align} \label{hatS-i-rep}
\hat{S}_i^N(t) &= \hat{S}_i^N(0) -\lambda_i \int_0^t \hat{\Upsilon}_i^N(s) ds  + \sum_{\ell=1, \ell\neq i}^L \int_0^t (\nu_{S,\ell,i}\hat{S}^N_\ell(s)- \nu_{S,i, \ell}\hat{S}^N_i(s)) ds   \non\\
&  \qquad - \hat{M}_{A,i}(t)  +\sum_{\ell=1, \ell\neq i}^L \big( \hat{M}_{S,\ell,i}(t) - \hat{M}_{S,i,\ell}(t) \big), 
\end{align}
where
\begin{equation*} 
\hat{M}^N_{S,\ell,i}(t) :=\frac{1}{\sqrt{N}} \left( P_{S,\ell,i}\left( \nu_{S,\ell,i} \int_0^tS^N_\ell(s)ds\right) - \nu_{S,\ell,i} \int_0^tS^N_\ell(s)ds \right). 
\end{equation*}

For the process $\hat{E}^N_i(t)$, by the representation in \eqref{eqn-E-rep-2-SEIR}, using the definitions of $E^{N,0}_{\ell,i}(t)$ and $E^N_{\ell,i}(t)$ in \eqref{eqn-EI-ij-def-SEIR}, we obtain
\begin{align} \label{hatE-rep-2-SEIR}
\hat{E}^N_i(t) &= \hat{E}^N_i(0) - \sum_{\ell=1}^L \hat{E}^N_\ell(0) \int_0^t p_{\ell,i}(s) G_0(ds)  + \lambda_i \int_0^t \hat{\Upsilon}^N_i(s) ds 
-  \sum_{\ell=1}^L \lambda_\ell \int_0^t \int_0^{t-s} p_{\ell,i}(u) G(du)  \hat\Upsilon^N_\ell(s) ds \non\\
 & \quad  + \sum_{\ell=1, \ell\neq i}^L \int_0^t (\nu_{E,\ell,i}\hat{E}^N_\ell(s)- \nu_{E,i, \ell}\hat{E}^N_i(s)) ds - \sum_{\ell=1}^L \big(  \hat{E}^{N,0}_{\ell,i}(t) + \hat{E}^{N}_{\ell,i}(t) \big)   \non\\
&\quad + \hat{M}^N_{A,i}(t)    +   \sum_{\ell=1, \ell \neq i}^L \big( \hat{M}^N_{E,\ell,i}(t) - \hat{M}^N_{E,i,\ell}(t) \big)\,,
\end{align}
where for $\ell,i\in \mathcal{L}$, 
\begin{align} 
\hat{E}^{N,0}_{\ell,i}(t) 
& := 
 \frac{1}{\sqrt{N}}  \sum_{k=1}^{E^N_\ell (0)}  \left(  {\bf1}_{\eta_{k,\ell}^0 \le t}{\bf1}_{X^{0,k}_\ell(\eta_{k,\ell}^0)=i} - \int_0^t p_{\ell,i}(s) G_0(ds) \right) \,, \label{hatE0-ij-def-SEIR} \\
 \hat{E}^{N}_{\ell,i}(t) 
 & :=  \frac{1}{\sqrt{N}}  \left(  \sum_{j=1}^{A^N_\ell(t)} {\bf 1}_{\tau^N_{j,\ell} + \eta_{j,\ell} \le t} {\bf1}_{X^j_\ell(\eta_{j,\ell})=i} - N \lambda_\ell \int_0^t \left( \int_0^{t-s}p_{\ell,i}(u) G(du) \right) \bar\Upsilon^N_\ell(s) ds   \right)\,,  \label{hatE-ij-def-SEIR}
\end{align}
and for $\ell\neq i$, 
$$
\hat{M}^N_{E,\ell,i}(t)  :=   \frac{1}{\sqrt{N}}  \left( P_{E,\ell,i}\left(\nu_{E,\ell,i}\int_0^t E^N_\ell(s)ds\right) - \nu_{E,\ell,i}\int_0^t E^N_\ell(s)ds \right)\,. 
$$

For the process $\hat{I}^N_i(t)$, 
we obtain
\begin{align} \label{hatI-rep-2-SEIR}
\hat{I}^N_i(t)&= \hat{I}^N_i(0)-  \sum_{\ell=1}^L \hat{I}^N_\ell(0) \int_0^t q_{\ell,i}(s) F_0(ds)  + \hat{E}^N_i(0) \sum_{\ell=1}^L  \left( \int_0^t p_{\ell,i}(s) G_0(ds) - \Phi^0_{\ell,i}(t) \right)   \non\\
& \quad +  \sum_{\ell=1}^L \lambda_\ell \int_0^t \int_0^{t-s} p_{\ell,i}(u) G(du)  \hat\Upsilon^N_\ell(s) ds    - \sum_{\ell=1}^L \lambda_\ell  \int_0^t \Phi_{\ell,i}(t-s)  \hat{\Upsilon}^N_\ell(s) ds \non\\
 & \quad    +   \sum_{\ell \neq i} \int_0^t \left( \nu_{I,\ell,i}  \hat{I}^N_\ell(s)-  \nu_{I,i,\ell}\hat{I}^N_i(s)  \right) ds + \sum_{\ell=1}^L \left( \hat{M}^N_{I,\ell,i}(t) - \hat{M}^N_{I,i,\ell}(t) \right) \non\\
& \quad + \sum_{\ell=1}^{L} \big( \hat{E}^{N,0}_{\ell,i}(t) +  \hat{E}^{N}_{\ell,i}(t)  \big)  - \sum_{\ell=1}^L \left(\hat{I}^{N,0,1}_{\ell,i}(t)  +\hat{I}^{N,0,2}_{\ell,i}(t) + \hat{I}^N_{\ell,i}(t) \right) ,
\end{align}
where  $ \hat{E}^{N,0}_{\ell,i}(t)$ and $\hat{E}^{N}_{\ell,i}(t)$ are defined in  \eqref{hatE0-ij-def-SEIR} and \eqref{hatE-ij-def-SEIR}, respectively, and 
\begin{align}
\hat{I}^{N,0,1}_{\ell,i}(t)&:= \frac{1}{\sqrt{N}}  \sum_{k=1}^{I^N_\ell (0)} \left( {\bf1}_{\zeta_{k,\ell}^0 \le t}{\bf1}_{Y^{0,k}_\ell(\zeta_{k,\ell}^0)=i} - \int_0^t q_{\ell,i}(s) F_0(ds) \right) \,, \label{eqn-hatI-01-def} \\
 \hat{I}^{N,0,2}_{\ell,i}(t) &:= \frac{1}{\sqrt{N}}   \sum_{k=1}^{E^N_\ell (0)} \Bigg( \sum_{\ell'=1}^L {\bf1}_{X^{0,k}_\ell(\eta_{k,\ell}^0)=\ell'}  {\bf1}_{\eta_{k,\ell}^0 + \zeta_{-k,\ell} \le t} {\bf1}_{Y^{-k,\ell}_{\ell'}( \zeta_{-k,\ell})=i}  -   \Phi^0_{\ell,i}(t) \Bigg)\,,  \label{eqn-hatI-02-def}\\
\hat{I}^N_{\ell,i}(t) &:=  \frac{1}{\sqrt{N}} \Bigg(  
\sum_{j=1}^{A^N_\ell(t)}   \sum_{\ell'=1}^L {\bf1}_{X^j_\ell( \eta_{j,\ell})=\ell'}  {\bf 1}_{\tau^N_{j,\ell} + \eta_{j,\ell}  + \zeta_{j,\ell}  \le t} {\bf1}_{Y^{j,\ell}_{\ell'}( \zeta_{j,\ell} )=i}  
-N \lambda_\ell \int_0^t \Phi_{\ell,i}(t-s)  \bar{\Upsilon}^N_\ell(s) ds \Bigg) \,,
\label{hatI-ij-def-SEIR} 
\end{align}
and 
for $\ell\neq i$, 
\begin{equation}\label{hatM-I-def}
\hat{M}^N_{I,\ell,i}(t)  :=   \frac{1}{\sqrt{N}}  \left( P_{I,\ell,i}\left(\nu_{I,\ell,i}\int_0^t I^N_\ell(s)ds\right) - \nu_{I,\ell,i}\int_0^t I^N_\ell(s)ds \right). 
\end{equation}


For the process $\hat{R}^N_i(t)$, we have
\begin{align} \label{hatR-rep-2-SEIR}
\hat{R}^N_i(t)&=   \sum_{\ell=1}^L \hat{I}^N_\ell(0) \int_0^t q_{\ell,i}(s) F_0(ds)  + \hat{E}^N_i(0) \sum_{\ell=1}^L  \Phi^0_{\ell,i}(t)    \non\\ 
& \quad +    \sum_{\ell=1}^L \lambda_\ell  \int_0^t \Phi_{\ell,i}(t-s)  \hat{\Upsilon}^N_\ell(s) ds   +   \sum_{\ell \neq i} \int_0^t \left( \nu_{R,\ell,i}  \hat{R}^N_\ell(s)-  \nu_{R,i,\ell}\hat{R}^N_i(s)  \right) ds  \non\\
 & \quad   + \sum_{\ell=1}^L \left( \hat{M}^N_{R,\ell,i}(t) - \hat{M}^N_{R,i,\ell}(t) \right)  + \sum_{\ell=1}^L \left(\hat{I}^{N,0,1}_{\ell,i}(t)  +\hat{I}^{N,0,2}_{\ell,i}(t) + \hat{I}^N_{\ell,i}(t) \right), 
\end{align}
where 
for $\ell,i\in \mathcal{L}$, and $\ell \neq i$, 
\begin{equation}\label{hatM-R-def}
\hat{M}^N_{R,\ell,i}(t)  :=   \frac{1}{\sqrt{N}}  \left( P_{R,\ell,i}\left(\nu_{R,\ell,i}\int_0^t R^N_\ell(s)ds\right) - \nu_{R,\ell,i}\int_0^t R^N_\ell(s)ds \right). 
\end{equation}

We establish the convergence of some key components in these representations  in  Lemmas \ref{lem-Mi-conv-SEIR}, \ref{lem-I0-ij-conv-SEIR} and \ref{lem-I-ij-conv-SEIR}.

\begin{lemma} \label{lem-Mi-conv-SEIR}
Under Assumption \ref{AS-FCLT-SEIR},
\begin{align*}
& \big( \hat{M}^N_{A,i}, \hat{M}^N_{E,\ell,i},    \hat{M}^N_{S,\ell,i},   \hat{M}^N_{I,\ell,i},  \hat{M}^N_{R,\ell,i},\, \ell, i\in \mathcal{L}, \ell\neq i \big)
 \\
 & \RA \big( \hat{M}_{A,i},  \hat{M}_{E,\ell,i},   \hat{M}_{S,\ell,i},   \hat{M}_{I,\ell,i},  \hat{M}_{R,\ell,i},\, \ell, i\in \mathcal{L}, \ell\neq i \big)
 \qinq D^{L+4L(L-1)} \qasq N \to \infty,
\end{align*}
where the limits are as given in Theorem \ref{thm-FCLT-SEIR}. 
\end{lemma}

\begin{proof}
This follows from a standard martingale convergence argument, see, e.g.,  Theorem 1.4 in Chapter 7 of \cite{ethier-kurtz}. The main step consists in proving that the quadratic variations converge (involving the convergence of fluid-scaled processes). We omit the details for brevity. 
\end{proof}

\begin{lemma} \label{lem-I0-ij-conv-SEIR}
Under Assumption \ref{AS-FCLT-SEIR}, 
\begin{align*}
\big(\hat{E}^{N,0}_{\ell,i}, \,\hat{I}^{N,0,1}_{\ell,i},\, \hat{I}^{N,0,2}_{\ell,i},\,\ell,i\in \mathcal{L}\big) \RA\big(\hat{E}^{0}_{\ell,i}, \,\hat{I}^{0,1}_{\ell,i},\, \hat{I}^{0,2}_{\ell,i},\,  \ell,i\in \mathcal{L}\big) 
\qinq D^{3L^2} \qasq N \to \infty,
\end{align*}
where the limits are as given in Theorem \ref{thm-FCLT-SEIR}. 
\end{lemma}

\begin{proof}
Since the proofs for the convergence of $\big(\hat{E}^{N,0}_{\ell,i}, \ell,i\in \mathcal{L}\big)$ and $\big(\hat{I}^{N,0,1}_{\ell,i}, \ell,i\in \mathcal{L}\big)$ follow from the same argument, 
we only prove the convergence of  $\big(\hat{I}^{N,0,1}_{\ell,i}, \ell,i\in \mathcal{L}\big)\RA \big(\hat{I}^{0,1}_{\ell,i}, \ell,i\in \mathcal{L}\big)$ in $D^{L^2}$.
Define $\tilde{I}^{N,0,1}_{\ell,i}(t) $ by replacing  $I^N_\ell(0)$ with $N \bar{I}_\ell(0)$ in \eqref{eqn-hatI-01-def} for $ \ell, i\in \mathcal{L}$.

Let us consider the convergence of $\tilde{I}^{N,0,1}_{1,1}$. 
Observe that the pairs $\big(\zeta_{k,1}, Y^{0,k}_1(\cdot)\big)$ and $\big(\zeta_{k',1}, Y^{0,k'}_1(\cdot)\big)$ are independent and have the same law. 
Thus, its proof follows in a similar approach for empirical processes, see, e.g., Theorem 14.3 in  \cite{billingsley1999convergence}. 
There are some differences due to the Markov process $Y^{0,k}_1$, which we highlight below. 
So, we apply Theorem 13.5 in \cite{billingsley1999convergence}.

For each $t>0$ and $\alpha\in\R$, we have 
\begin{align*}
&\E\left[\exp \left( \hat\imath \alpha \tilde{I}^{N,0,1}_{1,1}(t) \right)\right] = \E\left[ \prod_{k=1}^{N\I_1(0)} \exp\left(  \hat\imath \alpha \frac{1}{\sqrt{N}}  \left(  {\bf1}_{\zeta_{k,1}^0\le t}{\bf1}_{Y^{0,k}_1(\zeta_{k,1}^0)=1} - \int_0^t  q_{1,1}(s) F_0(ds) \right) \right) \right] \\
&= \prod_{k=1}^{N\I_1(0)} \E\left[ \exp\left(  \hat\imath \alpha \frac{1}{\sqrt{N}}  \left(  {\bf1}_{\zeta_{k,1}^0\le t}{\bf1}_{Y^{0,k}_1(\zeta_{k,1}^0)=1} - \int_0^t  q_{1,1}(s) F_0(ds) \right) \right) \right] \\
&= \left( 1 -\frac{\alpha^2}{2N} \E\left[\ \left(  {\bf1}_{\zeta_{k,1}^0\le t}{\bf1}_{Y^{0,k}_1(\zeta_{k,1}^0)=1} -  \int_0^t  q_{1,1}(s) F_0(ds) \right)^2  \right] + o(N^{-1})\right)^{N\I_1(0)} \\
&= \left( 1 -\frac{\alpha^2}{2N}   \int_0^t  q_{1,1}(s) F_0(ds) \left( 1 -  \int_0^t  q_{1,1}(s) F_0(ds) \right)  + o(N^{-1})\right)^{N\I_1(0)} \\
&\xrightarrow{N\to\infty}  \exp\left( - \frac{\alpha^2}{2}  \I_1(0)  \int_0^t  q_{1,1}(s) F_0(ds) \left( 1 -  \int_0^t  q_{1,1}(s) F_0(ds) \right)\right) = \E\left[\exp \left( \hat\imath \alpha \hat{I}^{0,1}_{1,1}(t) \right)\right]. 
\end{align*}
Similarly, it can be also shown that 
 for any $0 <s <t<r$ and $\alpha_1,\alpha_2 \in \R$, 
\begin{align*}
& \lim_{N\to\infty} \E\left[\exp \left( \hat\imath \alpha_1 \left(\tilde{I}^{N,0,1}_{1,1}(t) - \tilde{I}^{N,0,1}_{1,1}(s) \right)+  
\hat\imath \alpha_2 \left(\tilde{I}^{N,0,1}_{1,1}(r) - \tilde{I}^{N,0,1}_{1,1}(t) \right) \right)\right]  \\
&= \E\left[\exp \left( \hat\imath\alpha_1 \left(\hat{I}^{0,1}_{1,1}(t) - \hat{I}^{0,1}_{1,1}(s)\right) + \hat\imath\alpha_2 \left(\hat{I}^{0,1}_{1,1}(r) - \hat{I}^{0,1}_{1,1}(t)\right) \right)\right]  \\
& = \exp\Bigg( - \frac{\alpha^2_1}{2}  \I_1(0)  \int_s^t  q_{1,1}(u) F_0(du) \left( 1 -  \int_s^t  q_{1,1}(u) F_0(du) \right)  \\
& \qquad \qquad -  \frac{\alpha^2_2}{2}  \I_1(0)  \int_t^r  q_{1,1}(u) F_0(du) \left( 1 -  \int_t^r  q_{1,1}(u) F_0(du) \right)\\
& \qquad \qquad - \alpha_1\alpha_2  \I_1(0)  \int_s^t  q_{1,1}(u) F_0(du)  \int_t^r  q_{1,1}(u) F_0(du) \Bigg). 
\end{align*}
Thus, for
the convergence of finite dimensional distributions, with $t_1<t_2<\cdots<t_k$ and $\alpha_\ell, \ell =1,\dots,k$, we can write $\sum_{\ell=1}^k \hat\imath \alpha_\ell \tilde{I}^{N,0,1}_{1,1}(t_\ell)$ using the increments $\tilde{I}^{N,0,1}_{1,1}(t_\ell) - \tilde{I}^{N,0,1}_{1,1}(t_{\ell-1})$, and carry out the calculations as shown above. 

Next, to prove tightness, we employ  Theorem 13.5 and verify condition (13.14) in 
\cite{billingsley1999convergence}. 
We show that  for $r \le s \le t$ and for $N\ge 1$,
$$
\E\left[ \big| \tilde{I}^{N,0,1}_{1,1}(s) - \tilde{I}^{N,0,1}_{1,1}(r)\big|^2 \big| \tilde{I}^{N,0,1}_{1,1}(t) - \tilde{I}^{N,0,1}_{1,1}(s)\big|^2\right] \le   C(\phi(s) - \phi(r))  (\phi(t) - \phi(s)) \le C  (\phi(t) - \phi(r))^2
$$
for some constant $C$ and $\phi(t) =  \int_0^t  q_{1,1}(u) F_0(du)$ which is a nonnegative, nondecreasing and continuous function. Recall that $F_0$ is assumed to be continuous. This will enforce condition (13.14) in 
\cite{billingsley1999convergence}, which according to Theorem 13.5 implies tightness in $D$.
 Let 
$$
\Delta I^k_{r,s} =  {\bf1}_{r <\zeta_{k,1}^0\le s}{\bf1}_{Y^{0,k}_1(\zeta_{k,1}^0)=1} - \int_r^s  q_{1,1}(u) F_0(du),
$$
and
$$
\Delta I^k_{s,t} =  {\bf1}_{s<\zeta_{k,1}^0\le t}{\bf1}_{Y^{0,k}_1(\zeta_{k,1}^0)=1} -  \int_s^t  q_{1,1}(u) F_0(du). 
$$
Note that $\E[\Delta I^k_{r,s}] =0$, $\E[\Delta I^k_{s,t}] =0$,
$$
\E[(\Delta I^k_{r,s})^2] = \int_r^s  q_{1,1}(u) F_0(du)  \left( 1- \int_r^s  q_{1,1}(u) F_0(du) \right), 
$$
and
$$
\E[(\Delta I^k_{s,t})^2] = \int_s^t  q_{1,1}(u) F_0(du)  \left( 1- \int_s^t  q_{1,1}(u) F_0(du) \right). 
$$
By direct calculations, following similar steps in the proof of (14.9) in \cite{billingsley1999convergence},  we obtain 
\begin{align*}
\E\left[ \left|\sum_{k=1}^{N\I_1(0)} \Delta I^k_{r,s}\right|^2  \left|\sum_{k=1}^{N\I_1(0)} \Delta I^k_{s,t}\right|^2  \right] 
&= N\I_1(0) \E[ (\Delta I^1_{r,s})^2 (\Delta I^1_{s,t})^2]  \\
& \quad + N\I_1(0) (N\I_1(0)-1) \E[(\Delta I^1_{r,s})^2] \E[(\Delta I^1_{s,t})^2] \\
& \quad   + 2 N\I_1(0) (N\I_1(0)-1)    (\E[ \Delta I^1_{r,s}\Delta I^1_{s,t}])^2\,, \\
N^{-2}\E\left[ \left|\sum_{k=1}^{N\I_1(0)} \Delta I^k_{r,s}\right|^2  \left|\sum_{k=1}^{N\I_1(0)} \Delta I^k_{s,t}\right|^2  \right]& \le C  \int_r^s  q_{1,1}(u) F_0(du)  \int_s^t  q_{1,1}(u) F_0(du)\,.
\end{align*}
Thus we have shown the convergence $ \tilde{I}^{N,0,1}_{1,1} \RA  \hat{I}^{0,1}_{1,1}$ in $D$.

\smallskip

For the joint convergence $\big(\tilde{I}^{N,0,1}_{\ell,i},\, \ell,i=1,\dots L\big) $, since the variables and processes associated with patch $\ell$ and patch $\ell'$ are independent, it suffices to show the joint convergences $\big(\tilde{I}^{N,0,1}_{\ell,i}, \,i\in \mathcal{L}\big) $ for different $\ell$'s  separately. 
For the joint convergence $\big(\tilde{I}^{N,0,1}_{\ell,i}, \,i\in \mathcal{L}\big) $, 
we obtain tightness from that of each process as established above, so it suffices to show the joint convergence of their finite dimensional distributions. Take $\ell=1, i=1,2$ as an example. 
For $0<t_1< t_2$ and $\alpha_1,\alpha_2\in\R$,
\begin{align*}
&\E\left[\exp \left( \hat\imath \alpha_1 \tilde{I}^{N,0,1}_{1,1}(t_1) + \hat\imath \alpha_2 \tilde{I}^{N,0,1}_{1,2}(t_2)  \right)\right] \\
&= \E\bigg[ \prod_{k=1}^{N\I_1(0)} \exp\bigg(  \hat\imath \alpha_1 \frac{1}{\sqrt{N}}  \Big(  {\bf1}_{\zeta_{k,1}^0\le t_1}{\bf1}_{Y^{0,k}_1(\zeta_{k,1}^0)=1} - \int_0^{t_1}  q_{1,1}(s) F_0(ds) \Big)\\
& \qquad \qquad \qquad \qquad  +   \hat\imath \alpha_2 \frac{1}{\sqrt{N}}  \Big(  {\bf1}_{\zeta_{k,1}^0\le t_2}{\bf1}_{Y^{0,k}_1(\zeta_{k,1}^0)=2} - \int_0^{t_2}  q_{1,2}(s) F_0(ds)  \Big) \bigg) \bigg] \\
&= \bigg( 1 -\frac{1}{2N} \E\bigg[\bigg(  \alpha_1 \Big(  {\bf1}_{\zeta_{k,1}^0\le t_1}{\bf1}_{Y^{0,k}_1(\zeta_{k,1}^0)=1} -  \int_0^{t_1}  q_{1,1}(s) F_0(ds) \Big)  \\
& \qquad \qquad \qquad \qquad + \alpha_2 \Big( {\bf1}_{\zeta_{k,1}^0\le t_2}{\bf1}_{Y^{0,k}_1(\zeta_{k,1}^0)=2} -  \int_0^{t_2}  q_{1,2}(s) F_0(ds) 
\Big) \bigg)^2 \bigg] + o(N^{-1})\bigg)^{N\I_1(0)} \\
&= \bigg( 1 -\frac{\alpha_1^2}{2N}   \int_0^{t_1}  q_{1,1}(s) F_0(ds) \left( 1 -  \int_0^{t_1}  q_{1,1}(s) F_0(ds) \right)   \\
& \qquad - \frac{\alpha_2^2}{2N}   \int_0^{t_2}  q_{1,2}(s) F_0(ds) \left( 1 -  \int_0^{t_2}  q_{1,2}(s) F_0(ds) \right)  \\
& \qquad  + \frac{\alpha_1\alpha_2}{N}\int_0^{t_1}  q_{1,1}(s) F_0(ds) \int_0^{t_2}  q_{1,2}(s) F_0(ds)  + o(N^{-1})  \bigg)^{N\I_1(0)}  \\
&\xrightarrow{N\to\infty}  
\exp\bigg( - \frac{\alpha_1^2}{2} \I_1(0)  \int_0^{t_1}  q_{1,1}(s) F_0(ds) \left( 1 -  \int_0^{t_1}  q_{1,1}(s) F_0(ds) \right) \\
& \qquad \qquad \qquad- \frac{\alpha_2^2}{2}  \I_1(0)  \int_0^{t_2}  q_{1,2}(s) F_0(ds) \left( 1 -  \int_0^{t_2}  q_{1,2}(s) F_0(ds) \right)   \\
& \qquad \qquad \qquad   + \alpha_1\alpha_2 \I_1(0) \int_0^{t_1}  q_{1,1}(s) F_0(ds) \int_0^{t_2}  q_{1,2}(s) F_0(ds)   \bigg) \\ 
&\qquad \qquad = \E\left[\exp \left( \hat\imath\alpha_1 \hat{I}^{0,1}_{1,1}(t_1) + \hat\imath \alpha_2 \hat{I}^{0,1}_{1,2}(t_2)  \right)\right]. 
\end{align*}
This calculation can be extended to the computation of finite dimensional distributions of $\big(\tilde{I}^{N,0,1}_{1,1}, \tilde{I}^{N,0,1}_{1,2}\big)$, and then that of $\big(\tilde{I}^{N,0,1}_{\ell,i}, \,i\in \mathcal{L}\big)$. Therefore, we can conclude the joint convergence of $\big(\tilde{I}^{N,0,1}_{\ell,i}, \,i\in \mathcal{L}\big)$.

\smallskip 

Next, to prove the convergence of $ \hat{I}^{N,0,1}_{\ell,i}$, 
 it suffices to show that 
$ \hat{I}^{N,0,1}_{\ell,i} -  \tilde{I}^{N,0,1}_{\ell,i} \RA 0$ in $D$ for each $\ell,i\in \mathcal{L}$. 
We consider  the convergence  $ \hat{I}^{N,0,1}_{1,1} -  \tilde{I}^{N,0,1}_{1,1} \RA 0$ in $D$. 
We have
\begin{align} \label{eqn-hat-tilde-I0-diff}
& \text{sign} (I^N_1(0)- N \bar{I}_1(0))  \Big(\hat{I}^{N,0,1}_{1,1} (t) -  \tilde{I}^{N,0,1}_{1,1}(t)\Big) \non  \\
& =  \frac{1}{\sqrt{N}}  \sum_{k= I^N_1(0) \wedge N \bar{I}_1(0) }^{I^N_1(0) \vee N\bar{I}_1(0)} \left(    {\bf1}_{\zeta_{k,1}^0\le t}{\bf1}_{Y^{0,k}_1(\zeta_{k,\ell}^0)=1} -  \int_0^t  q_{1,1}(s) F_0(ds) \right)\non \\
&=  \frac{1}{\sqrt{N}}  \sum_{k= I^N_1(0) \wedge N \bar{I}_1(0) }^{I^N_1(0) \vee N\bar{I}_1(0)}     {\bf1}_{\zeta_{k,1}^0\le t}{\bf1}_{Y^{0,k}_1(\zeta_{k,\ell}^0)=1} -   |\hat{I}^N_1(0)| \int_0^t  q_{1,1}(s) F_0(ds)\,. 
\end{align}
It is clear that by Assumption \ref{AS-FCLT-SEIR}, 
$$\E\Big[ \big( \hat{I}^{N,0,1}_{1,1} (t) -  \tilde{I}^{N,0,1}_{1,1}(t)\big)^2 \Big]
= \int_0^t  q_{1,1}(s) F_0(ds) \left( 1- \int_0^t  q_{1,1}(s) F_0(ds)\right) \E \big[|\bar{I}^N_1(0) - \bar{I}_1(0)|\big]\to 0
$$
as $N\to \infty$. To show that  $\big\{ \hat{I}^{N,0,1}_{1,1} -  \hat{I}^{N,0,1}_{1,1}\big\}_N$ is tight, by Assumption \ref{AS-FCLT-SEIR} and \eqref{eqn-hat-tilde-I0-diff}, it suffices to prove the tightness of the first term on the right hand side of \eqref{eqn-hat-tilde-I0-diff}, which we denote as $\Delta^{N,0,1}_{1,1}(t)$. By  the Corollary of Theorem 7.4 in \cite{billingsley1999convergence},  it suffices to show that  for all $\ep>0$, 
\begin{align}\label{eqn-hat-tilde-I0-diff-p1}
\limsup_{N}\sup_{0\le t\le T}\frac{1}{\delta}\P\bigg(\sup_{0\le u\le \delta}\big|\Delta^{N,0,1}_{1,1}(t+u)-\Delta^{N,0,1}_{1,1}(t)\big|>\ep\bigg)\to0, \qasq \delta\to 0. 
\end{align}
Since $\Delta^{N,0,1}_{1,1}(t+u)$ is increasing in $u$, we only need to consider 
\begin{align*}
\E \big[ \big|\Delta^{N,0,1}_{1,1}(t+\delta)-\Delta^{N,0,1}_{1,1}(t)\big|^2 \big] = \E \big[|\bar{I}^N_1(0) - \bar{I}_1(0)|\big] \int_t^{t+\delta}   q_{1,1}(s) F_0(ds)   \le  \E \big[|\bar{I}^N_1(0) - \bar{I}_1(0)|\big] \delta 
\end{align*}
whose $\limsup_N$ is equal to zero under Assumption \ref{AS-FCLT-SEIR}, and thus
we conclude that \eqref{eqn-hat-tilde-I0-diff-p1} holds. Therefore we have shown  $ \hat{I}^{N,0,1}_{1,1} -  \tilde{I}^{N,0,1}_{1,1} \to 0$ in $D$ in probability, and conclude the joint convergence
 $\big(\hat{I}^{N,0,1}_{\ell,i}, \ell,i\in \mathcal{L}\big)\RA \big(\hat{I}^{0,1}_{\ell,i}, \ell,i\in \mathcal{L}\big)$ in $D^{L^2}$. 
To prove that the limit processes $ \big(\hat{I}^{0,1}_{\ell,i}, \ell,i\in \mathcal{L}\big)$ are continuous when $F_0$ is continuous, since they are Gaussian, it suffices to show continuity in the quadratic mean \cite{hahn1978central}, that is, for all $t>0$, $\lim_{s\to t} \E\big[\big|\hat{I}^{0,1}_{\ell,i}(t) - \hat{I}^{0,1}_{\ell,i}(s)\big|^2\big] =0$.  This is easily checked from the continuity of the covariance functions.

\smallskip

We next focus on the processes $\big(\hat{I}^{N,0,2}_{\ell,i}, \ell,i\in \mathcal{L} \big)$. 
Define $\tilde{I}^{N,0,2}_{\ell,i}(t)$ by replacing  $E^N_\ell(0)$ with $N \bar{E}_\ell(0)$ in the expression of $\tilde{I}^{N,0,2}_{\ell,i}(t)$ in  \eqref{eqn-hatI-02-def}. 
We first prove the joint convergence $\big(\tilde{I}^{N,0,2}_{\ell,i}, \ell,i\in \mathcal{L} \big)\RA \big(\hat{I}^{0,2}_{\ell,i}, \ell,i\in \mathcal{L} \big)$ in $D^{L^2}$ as $N\to \infty$. 
 We again apply   Theorem 13.5 in \cite{billingsley1999convergence}. 
By direct calculations, we obtain for $t\ge 0$,
\begin{align*}
\E\left[ \exp\left( \hat\imath \alpha\tilde{I}^{N,0,2}_{\ell,i}(t) \right)\right]
& \xrightarrow{N\to\infty}  \E\left[ \exp\left( \hat\imath \alpha\hat{I}^{0,2}_{\ell,i}(t) \right)\right] =  \exp\bigg( -\frac{\alpha^2}{2} \bar{E}_\ell(0) \Phi^0_{\ell,i}(t) (1-\Phi^0_{\ell,i}(t))   \bigg) 
\end{align*}
and 
 for $t' \le t \le t''$ and $\alpha_1, \alpha_2 \in \R$, 
\begin{align*}
& \E\left[ \exp\left( \hat\imath \alpha_1 \left( \tilde{I}^{N,0,2}_{\ell,i}(t) - \tilde{I}^{N,0,2}_{\ell,i}(t') \right) +  \hat\imath \alpha_2 \left( \tilde{I}^{N,0,2}_{\ell,i}(t'') - \tilde{I}^{N,0,2}_{\ell,i}(t) \right)\right)\right] \\
& \xrightarrow{N\to\infty}\E\left[ \exp\left( \hat\imath \alpha_1 \left( \tilde{I}^{0,2}_{\ell,i}(t) - \tilde{I}^{0,2}_{\ell,i}(t') \right) +  \hat\imath \alpha_2 \left( \tilde{I}^{0,2}_{\ell,i}(t'') - \tilde{I}^{0,2}_{\ell,i}(t) \right)\right)\right] \\
&\quad \quad =  \exp\Bigg( -\frac{\alpha_1^2}{2} \bar{E}_\ell(0) (\Phi^0_{\ell,i}(t)-\Phi^0_{\ell,i}(t') )\left[1-(\Phi^0_{\ell,i}(t)-\Phi^0_{\ell,i}(t') ) \right] \\
& \qquad \qquad \qquad  -\frac{\alpha_2^2}{2} \bar{E}_\ell(0) (\Phi^0_{\ell,i}(t'')-\Phi^0_{\ell,i}(t) )\left[1-(\Phi^0_{\ell,i}(t'')-\Phi^0_{\ell,i}(t) ) \right]  \\
& \qquad \qquad\qquad -\alpha_1\alpha_2 \bar{E}_\ell(0) (\Phi^0_{\ell,i}(t'')-\Phi^0_{\ell,i}(t) )(\Phi^0_{\ell,i}(t')-\Phi^0_{\ell,i}(t') ) \Bigg)\,. 
\end{align*} 
Hence, we can establish the convergence of finite dimensional distributions of $\hat{I}^{N,0,2}_{\ell,i}$ similarly as that of $\tilde{I}_{1,1}^{N,0,1}(t)$ above.  
For tightness, we obtain
for $t' \le t \le t''$ and for $N\ge 1$,
\begin{align*}
& \E\left[ \big| \hat{I}^{N,0,2}_{\ell,i}(t') - \hat{I}^{N,0,2}_{\ell,i}(t)\big|^2 \big| \hat{I}^{N,0,2}_{\ell,i}(t'') - \hat{I}^{N,0,2}_{\ell,i}(t)\big|^2\right] \le   C(\phi(t) - \phi(t'))  (\phi(t'') - \phi(t)) \le C  (\phi(t'') - \phi(t'))^2
\end{align*}
where $\phi(t) =\int_0^t  \sum_{\ell'=1}^L p_{\ell, \ell'}(u) \int_0^{t-u} q_{\ell', i}(v) F(dv|u)  G_0(du)$. Note that since $G_0$ is continuous, this function $\phi(t)$ is a nonnegative, nondecreasing and continuous function.
This proves the convergence of  $\tilde{I}^{N,0,2}_{\ell,i} \RA  \hat{I}^{0,2}_{\ell,i}$ in $D$ as $N\to\infty$.

For the joint convergence of  $\tilde{I}^{N,0,2}_{\ell,i}$ and  $\tilde{I}^{N,0,2}_{\ell',i'}$, we can follow a similar argument as the joint convergence $\big(\tilde{I}^{N,0,1}_{\ell,i},\, \ell,i=1,\dots L\big)$  above.   Thus we have shown the joint convergence  $\big(\tilde{I}^{N,0,2}_{\ell,i}, \ell,i\in \mathcal{L} \big)\RA \big(\hat{I}^{0,2}_{\ell,i}, \ell,i\in \mathcal{L} \big)$ in $D^{L^2}$. 
To conclude $\big(\hat{I}^{N,0,2}_{\ell,i}, \ell,i\in \mathcal{L} \big)\RA \big(\hat{I}^{0,2}_{\ell,i}, \ell,i\in \mathcal{L} \big)$ in $D^{L^2}$, it remains to show that $\hat{I}^{N,0,2}_{\ell,i} - \tilde{I}^{N,0,2}_{\ell,i}\to 0 $ in $D$ in probability for each $\ell, i\in \mathcal{L}$. 
We have 
\begin{align*}
\E \Big[ \big( \hat{I}^{N,0,2}_{\ell,i}(t) - \tilde{I}^{N,0,2}_{\ell,i}(t)\big)^2 \Big] = \Phi^0_{\ell,i}(t) (1- \Phi^0_{\ell,i}(t)) \E\big[\big|\bar{E}^N_\ell(0) - \bar{E}_\ell(0) \big| \big] \to 0 \qasq N \to \infty, 
\end{align*}
under Assumption \ref{AS-FCLT-SEIR}. To prove tightness of $\big\{\hat{I}^{N,0,2}_{\ell,i} - \tilde{I}^{N,0,2}_{\ell,i}\big\}_N$, observing that 
\begin{align*}
&\text{sign} (E^N_\ell(0) - N \bar{E}_\ell(0))  \big(\hat{I}^{N,0,2}_{\ell,i}(t) - \tilde{I}^{N,0,2}_{\ell,i}(t) \big) \\
& = \frac{1}{\sqrt{N}}   \sum_{k=E^N_\ell (0) \wedge N \bar{E}_\ell(0)}^{E^N_\ell (0) \vee N \bar{E}_\ell(0)}  \sum_{\ell'=1}^L {\bf1}_{X^{0,k}_\ell(\eta_{k,\ell}^0)=\ell'}  {\bf1}_{\eta_{k,\ell}^0 + \zeta_{-k,\ell} \le t} {\bf1}_{Y^{-k,\ell}_{\ell'}( \zeta_{-k,\ell})=i}  -   \big| \hat{E}^N_\ell(0)\big| \Phi^0_{\ell,i}(t)\,,  
\end{align*} 
it suffices, applying the Corollary of Theorem 7.4 in \cite{billingsley1999convergence} to the first term, denoted by $\Delta^{N,0,2}_{\ell,i}(t)$, to show that 
\begin{align}\label{eqn-hat-tilde-I02-diff-p1}
\limsup_{N}\sup_{0\le t\le T}\frac{1}{\delta}\P\bigg(\sup_{0\le u\le \delta}\big|\Delta^{N,0,2}_{\ell,i}(t+u)-\Delta^{N,0,2}_{\ell,i}(t)\big|>\ep\bigg)\to0, \qasq \delta\to 0. 
\end{align}
 Since it is increasing in $t$, we consider for $\delta>0$,
\begin{align*}
\E \big[ \big|\Delta^{N,0,2}_{\ell,i}(t+\delta) - \Delta^{N,0,2}_{\ell,i}(t) \big|^2 \big]  = \E\big[\big|\bar{E}^N_\ell(0) -  \bar{E}_\ell(0)\big| \big] \big( \Phi^0_{\ell,i}(t+\delta) - \Phi^0_{\ell,i}(t)\big).
\end{align*}
The $\limsup_N$ of the above is equal to zero 
by Assumption \ref{AS-FCLT-SEIR}, so it is clear that \eqref{eqn-hat-tilde-I02-diff-p1} holds. Thus we have shown  $\big(\hat{I}^{N,0,2}_{\ell,i}, \ell,i\in \mathcal{L} \big)\RA \big(\hat{I}^{0,2}_{\ell,i}, \ell,i\in \mathcal{L} \big)$ in $D^{L^2}$. 


For the joint convergence of $\big(\hat{E}^{N,0}_{\ell,i}, \,\hat{I}^{N,0,1}_{\ell,i},\, \hat{I}^{N,0,2}_{\ell,i},\,\ell,i\in \mathcal{L}\big) $, by the independence of the variables associated with $I^N_\ell(0)$ and $E^N_\ell(0)$, it suffices to show the joint convergence of $\big(\hat{E}^{N,0}_{\ell,i},  \hat{I}^{N,0,2}_{\ell,i},\,\ell,i\in \mathcal{L}\big)$. 
We define $\tilde{E}^{N,0}_{\ell,i}(t)$ by replacing $E^N(0)$ with $N \bar{E}(0)$ in the expression of $\tilde{E}^{N,0}_{\ell,i}(t)$ in \eqref{hatE0-ij-def-SEIR}. 
Similar to the proof above, we have  $\hat{E}^{N,0}_{\ell,i} - \tilde{E}^{N,0}_{\ell,i}\to 0$ in $D$ in probability. 
It then suffices to show the joint convergence of $\big(\tilde{E}^{N,0}_{\ell,i},  \tilde{I}^{N,0,2}_{\ell,i},\,\ell,i\in \mathcal{L}\big)$ and moreover, since they are tight individually, it suffices to show the convergence of their joint finite dimensional distributions. Note that $\tilde{E}^{N,0}_{\ell,i}(t)$ and $  \tilde{I}^{N,0,2}_{\ell',i}(t')$ are independent for $\ell \neq \ell'$. We calculate that for $\alpha, \alpha' \in \R$ and $t, t' >0$, 
\begin{align*}
& \lim_{N\to\infty}\E\Big[\exp\Big( \hat\imath \alpha \tilde{E}^{N,0}_{\ell,i}(t) + \hat\imath \alpha' \tilde{I}^{N,0,2}_{\ell,i'}(t')\Big) \Big] = \E\Big[\exp\Big( \hat\imath \alpha \tilde{E}^{0}_{\ell,i}(t) + \hat\imath \alpha' \tilde{I}^{0,2}_{\ell,i'}(t')\Big) \Big] \\
&= \exp\Bigg(-\frac{\alpha^2}{2} \bar{E}_\ell(0)\int_0^t p_{\ell,i}(s) d G_0(s) \bigg(1-\int_0^t p_{\ell,i}(s) d G_0(s)\bigg)   -\frac{(\alpha')^2}{2} \bar{E}_\ell(0) \Phi^0_{\ell,i'}(t') (1-\Phi^0_{\ell,i'}(t'))  \\
& \qquad\quad \quad  - \alpha \alpha' \bar{E}_\ell(0)  \bigg( \int_0^t   p_{\ell, i}(u) \int_0^{t'-u} q_{i, i'}(v) H_0(du,dv)  -  \int_0^t  p_{\ell,i}(s) G_0(ds)\Phi^0_{\ell,i'}(t') \bigg)   \Bigg). 
\end{align*}
This can be extended easily to finite dimensional distributions of $\big(\tilde{E}^{N,0}_{\ell,i}(t_1),\dots, \tilde{E}^{N,0}_{\ell,i}(t_k),   \tilde{I}^{N,0,2}_{\ell,i'}(t_1),\\  \dots,  \tilde{I}^{N,0,2}_{\ell,i'}(t_k),\,\ell,i\in \mathcal{L}\big)$ for $t_1< t_2 < \cdots < t_k$, $k\ge 1$. 
This requires calculations of the covariances of the cross terms for the increments of $\tilde{E}^{N,0}_{\ell,i}(t)$ and $\tilde{I}^{N,0,2}_{\ell,i'}(t)$. 
For instance, we have for $\alpha, \alpha' \in \R$ and $t_1<t_2$, 
\begin{align*}
 &\lim_{N\to\infty}\E\Big[\exp\Big( \hat\imath \alpha_1 \big(  \tilde{E}^{N,0}_{\ell,i}(t_2) - \tilde{E}^{N,0}_{\ell,i}(t_1) \big) + \hat\imath \alpha_2 \big( \tilde{I}^{N,0,2}_{\ell,i'}(t_2) -\tilde{I}^{N,0,2}_{\ell,i'}(t_1)\big) \Big) \Big] \\
 & =\E\Big[\exp\Big( \hat\imath \alpha_1 \big(  \tilde{E}^{0}_{\ell,i}(t_2) - \tilde{E}^{0}_{\ell,i}(t_1) \big) + \hat\imath \alpha_2 \big( \tilde{I}^{0,2}_{\ell,i'}(t_2) -\tilde{I}^{0,2}_{\ell,i'}(t_1)\big) \Big) \Big] \\
 &= \exp\Bigg(-\frac{\alpha_1^2}{2} \bar{E}_\ell(0)\int_{t_1}^{t_2} p_{\ell,i}(s) d G_0(s) \bigg(1-\int_{t_1}^{t_2} p_{\ell,i}(s) d G_0(s)\bigg)  \\
 & \qquad  -\frac{\alpha_2^2}{2} \bar{E}_\ell(0) \big(\Phi^0_{\ell,i'}(t_2) - \Phi^0_{\ell,i'}(t_1)\big)  \big[1-  \big(\Phi^0_{\ell,i'}(t_2) - \Phi^0_{\ell,i'}(t_1)\big)  \big]  \\
& \qquad  - \alpha_1 \alpha_2 \bar{E}_\ell(0)  \bigg( \int_{t_1}^{t_2}   p_{\ell, i}(u) \int_{t_1-u}^{t_2-u} q_{i, i'}(v) H_0(du,dv)  -  \int_{t_1}^{t_2}  p_{\ell,i}(u) G_0(du)\big( \Phi^0_{\ell,i'}(t_2)-  \Phi^0_{\ell,i'}(t_1) \big)  \bigg)   \Bigg).
\end{align*}
Therefore we have shown the joint convergence of $\big(\hat{E}^{N,0}_{\ell,i},  \hat{I}^{N,0,2}_{\ell,i},\,\ell,i\in \mathcal{L}\big)$. 
Finally for the continuity of the limit processes, it suffices to show the continuity in the quadratic mean  \cite{hahn1978central}, which follows from the continuity of the covariance functions. 
This completes the proof of the lemma. 
\end{proof}

For the next lemma on the moment estimates, we shall need the following technical result. 

\begin{lemma}\label{estimhatUps}
In the two cases $\gamma\in[0,1)$ and $\sum_{\ell\not=i} \kappa_{i\ell}=0$,
there exists a constant $C$ such that for any $0\le t\le T$, $1\le i\le L$,
\begin{equation}\label{eq:esimhatUps}
|\hat{\Upsilon}_i^N(t)|\le C\left(|\hat{S}_i^N(t)| + |\hat{E}_i^N(t)|+|\hat{I}_i^N(t)|+|\hat{R}_i^N(t)|\right)
+\sum_{\ell \neq i}\kappa_{i\ell}|\hat{I}_\ell^N(t)|\,.
\end{equation}
\end{lemma}

\begin{proof}
We consider again the map $\psi:[0,1]^4\times[0,\bar\kappa]\mapsto\R_+$:
\[ \psi(s,e,i,r,u)=\frac{s(i+u)}{(s+e+i+r)^\gamma}\,.\]
We have 
\begin{align*}
0\le \psi'_s(s,e,i,r,u)&=\frac{((1-\gamma)s+e+i+r)(i+u)}{(s+e+i+r)^{1+\gamma}}\le \frac{ 1+ \bar\kappa  }{(s+e+i+r)^\gamma},\\
0\ge \psi'_e(s,e,i,r,u)&=-\gamma\frac{s(i+u)}{(s+e+i+r)^{1+\gamma}}\ge-\frac{1+\bar\kappa}{(s+e+i+r)^\gamma},\\
\psi'_i(s,e,i,r,u)&=\frac{s(s+e+i+r)-\gamma s(i+u)}{(s+e+i+r)^{1+\gamma}},\quad |\psi'_i(s,e,i,r,u)|\le \frac{1+\bar\kappa}{(s+e+i+r)^\gamma},\\
0\ge \psi'_r(s,e,i,r,u)&=-\gamma\frac{s(i+u)}{(s+e+i+r)^{1+\gamma}}\ge-\frac{1+\bar\kappa}{(s+e+i+r)^\gamma},\\
0\le \psi'_u(s,e,i,r,u)&=\frac{s}{(s+e+i+r)^\gamma}\le 1\,.
\end{align*}
Moreover, if we define for $0\le a\le 1$, 
$g(a)=\psi(s+a(s'-s),e+a(e'-e),i+a(i'-i), r+a(r'-r), u+a(u'-u))$, we have 
\begin{align*}
& \psi(s',e',i',r',u')-\psi(s,e,i,r,u)=\int_0^1g'(a)da\\
&=\left(\int_0^1\psi'_s(s+a(s'-s),e+a(e'-e),i+a(i'-i), r+ a(r'-r), u+a(u'-u))da\right) [s'-s]\\&
\quad+\left(\int_0^1\psi'_e(s+a(s'-s),e+a(e'-e),i+a(i'-i), r+ a(r'-r), u+a(u'-u))da\right) [e'-e]\\ &
\quad +\left( \int_0^1\psi'_i(s+a(s'-s),e+a(e'-e),i+a(i'-i), r+ a(r'-r), u+a(u'-u))da\right) [i'-i]\\ &
\quad + \left( \int_0^1\psi'_r(s+a(s'-s),e+a(e'-e),i+a(i'-i), r+ a(r'-r), u+a(u'-u))da\right) [r'-r]\\ &
\quad +  \left(\int_0^1\psi'_u(s+a(s'-s),e+a(e'-e),i+a(i'-i), r+ a(r'-r), u+a(u'-u))da\right) [u'-u]\,. 
\end{align*}
We have 
\begin{align*}
\bar{\Upsilon}_i^N(t)&=\psi\bigg(\bar{S}_i^N(t),\bar{E}_i^N(t), \bar{I}_i^N(t),\bar{R}_i^N(t),\sum_{\ell\not=i}\kappa_{i\ell} \bar{I}_\ell^N(t) \bigg), \quad 
\bar\Upsilon_i(t) =\psi\bigg(\bar{S}_i(t),\bar{E}_i(t),\bar{I}_i(t),\bar{R}_i(t),\sum_{\ell\not=i}\kappa_{i\ell} \bar{I}_\ell(t) \bigg)\,.
\end{align*}
Suppose first that $\gamma<1$.
Clearly, the result will follow from the last formulas, if we prove that there exists $C>0$ such that for all $t\in[0,T]$, $N\ge1$,
\[\int_0^1((1-a)\bar{S}_i(t)+a\bar{S}^N_i(t)+(1-a)\bar{E}_i(t)+a\bar{E}^N_i(t)+(1-a)\bar{I}_i(t)+a\bar{I}^N_i(t)+(1-a)\bar{R}_i(t)+a\bar{R}^N_i(t))^{-\gamma}da\le C\,.\]
We have
\begin{align*}
\int_0^1&\Big((1-a)\bar{S}_i(t)+a\bar{S}^N_i(t)+(1-a)\bar{E}_i(t)+a\bar{E}^N_i(t)+(1-a)\bar{I}_i(t)+a\bar{I}^N_i(t) \\
& \qquad \qquad  +(1-a)\bar{R}_i(t)+a\bar{R}^N_i(t) \Big)^{-\gamma}da\\
&\le (\bar{S}_i(t)+\bar{E}_i(t)+\bar{I}_i(t)+\bar{R}_i(t))^{-\gamma}\int_0^1\frac{da}{a^\gamma}\\
&=\frac{ (\bar{S}_i(t)+\bar{E}_i(t)+\bar{I}_i(t)+\bar{R}_i(t))^{-\gamma}}{1-\gamma}\le \frac{C_{i,T}^{-\gamma}}{1-\gamma}\, ,
\end{align*}
for any $0\le t\le T$, where we have used Lemma \ref{le:lowbd} for the last inequality. 

It is easy to check that in the case where the variable $u$ disappear from the above formulas, the derivatives of $\psi$ are bounded on $[0,1]^3$, and the result holds in the case $\gamma=1$ as well.
\end{proof}

We will now prove the following estimate. 

\begin{lemma} \label{lem-estimate-SIR}
For each $i \in \mathcal{L}$, 
\begin{align}  \label{hatPhi-2-bound-sup}
\sup_N  \E\bigg[ \sup_{t \in [0,T]} \big| \hat\Upsilon^N_i(t)\big|^2\bigg] < \infty.
\end{align}
\end{lemma}
\begin{proof}
We first show that for each $i \in \mathcal{L}$, 
 \begin{equation} \label{hatPhi-2-bound}
 \sup_N \sup_{t \in [0,T]} \E\left[  \big| \hat\Upsilon^N_i(t)\big|^2\right] < \infty.
 \end{equation}
 We shall use \eqref{eq:esimhatUps}.
  In the representations of $\hat{S}^N_i(t)$,  $\hat{E}^N_i(t)$, $\hat{I}^N_i(t)$ and $\hat{R}^N_i(t)$
  in \eqref{hatS-i-rep}, \eqref{hatE-rep-2-SEIR}, \eqref{hatI-rep-2-SEIR} and \eqref{hatR-rep-2-SEIR}, respectively, the following hold: there exists a constant $C>0$ such that for each $i, \ell \in \LL$, 
    \begin{align} \label{hatPhi-2-bound-p2}
  & \sup_N \E\big[ \big|Z_0\big|^2\big]  \le  C, \qforq Z_0= \hat{S}_i^N(0), \hat{E}_i^N(0), \hat{I}_i^N(0), \hat{R}_i^N(0),  \non\\
  & \sup_N \sup_{t \in [0,T]}  \E\big[ (\hat{M}_{A,i}^N(t))^2 \big] \le \lambda_i \int_0^T \sup_N\bar\Upsilon^N_i(s) ds \le  \lambda_i  \bar\kappa_i T, \non\\
  &  \sup_N \sup_{t \in [0,T]}  \E\big[ (Z(t))^2 \big] \le C \nu_Z T, \qforq Z(t) = \hat{M}_{S,i,\ell}^N(t), \hat{M}_{E,i,\ell}^N(t), \hat{M}_{I,i,\ell}^N(t), \hat{M}_{R,i,\ell}^N(t),  
  \end{align}
  with the corresponding $\nu_Z = \nu_{S,i,\ell},  \nu_{E,i,\ell},  \nu_{I,i,\ell},  \nu_{R,i,\ell}$. 
  Moreover, it is easy to check that 
    \begin{align*}
 &    \sup_N  \sup_{t\in [0,T]}\E\big[\big(Z(t)\big)^2 \big] \le C, \qforq Z(t) = \hat{E}^{N,0}_{\ell,i}(t),   \hat{I}^{N,0,1}_{\ell,i}(t),  \hat{I}^{N,0,2}_{\ell,i}(t)\,,\\
   &  \sup_N  \sup_{t\in [0,T]} \E\big[\big(\hat{E}^{N}_{\ell,i}(t)\big)^2 \big] \le  \lambda_\ell  \bar\kappa_\ell T\, \quad 
  \sup_N  \sup_{t\in [0,T]} \E\big[\big( \hat{I}^{N}_{\ell,i}(t)\big)^2 \big]  \le \lambda_\ell  \bar\kappa_\ell T\,. 
  \end{align*}
  Thus, by taking squares of the processes $\hat{S}^N_i(t)$,  $\hat{E}^N_i(t)$,  $\hat{I}^N_i(t)$ and $\hat{R}^N_i(t)$
  in \eqref{hatS-i-rep}, \eqref{hatE-rep-2-SEIR}, \eqref{hatI-rep-2-SEIR} and \eqref{hatR-rep-2-SEIR}, we can apply Cauchy-Schwartz inequality and Gronwall's inequality to conclude that 
   \begin{align} \label{hatPhi-2-bound-p3}
 & \sup_N \sup_{t \in [0,T]} \E\left[  \big| Z(t)\big|^2\right]<\infty, \qforq Z(t) = \hat{S}^N_i(t), \hat{E}^N_i(t),  \hat{I}^N_i(t), \hat{R}^N_i(t),\,\,  i\in \mathcal{L}, 
  \end{align}
 and thus \eqref{hatPhi-2-bound}  follows from \eqref{eq:esimhatUps}.

  We next prove \eqref{hatPhi-2-bound-sup}. 
  By \eqref{hatPhi-2-bound-p2} and Doob's inequality, we obtain the martingale terms satisfy 
   \begin{align} \label{hatPhi-2-bound-p4}
   & \sup_N   \E\bigg[ \sup_{t \in [0,T]}(Z(t))^2 \bigg] <\infty, \qforq Z(t) = \hat{M}_{A,i}^N(t), \hat{M}_{S,i,\ell}^N(t), \hat{M}_{E,i,\ell}^N(t), \hat{M}_{I,i,\ell}^N(t), \hat{M}_{R,i,\ell}^N(t), \quad i,\ell \in \mathcal{L}. 
  \end{align}
  Then, by the expression of $\hat{S}^N_i(t)$ in \eqref{hatS-i-rep}, the bounds in \eqref{eq:esimhatUps} and \eqref{hatPhi-2-bound-p3}, we easily obtain that the property in \eqref{hatPhi-2-bound-sup} holds for $ \hat{S}^N_i(t)$.  Indeed, we note in particular that 
  \begin{align*} 
  \E&\left[\sup_{t\in[0,T]}\left|\lambda_i\int_0^t\hat{\Upsilon}^N_i(s)ds+\sum_{\ell\not=i}\int_0^t\{\nu_{S,\ell,i}\hat{S}^N_\ell(s)-\nu_{S,i,\ell}\hat{S}^N_i(s)\}ds\right|^2\right]\\
 &\qquad \le2\lambda_i^2T\int_0^T\E(|\hat{\Upsilon}^N_i(s)|^2)ds+4L\sum_{\ell\not=i}T\int_0^T\{\nu_{S,\ell,i}^2\E(|\hat{S}^N_\ell(s) |^2)+\nu_{S,i,\ell}^2\E(|\hat{S}^N_i(s) |^2\}ds,
  \end{align*}
  hence 
  \eqref{hatPhi-2-bound-p3} and \eqref{hatPhi-2-bound}, which we have already established, allow us to bound this expectation. 

For $\hat{E}^N_i(t)$,  given the bounds in  \eqref{eq:esimhatUps}, \eqref{hatPhi-2-bound-p3} and \eqref{hatPhi-2-bound-p4}, 
it suffices to show that the property in \eqref{hatPhi-2-bound-sup} holds for $ \hat{E}^{N,0}_{\ell,i}(t)$ and  $ \hat{E}^{N}_{\ell,i}(t)$. 
Similarly, 
 for $\hat{I}^N_i(t)$, 
it suffices to show that the property in \eqref{hatPhi-2-bound-sup} holds for $ \hat{I}^{N,0,1}_{\ell,i}(t)$, $ \hat{I}^{N,0,2}_{\ell,i}(t)$ and  $ \hat{I}^{N}_{\ell,i}(t)$. 
We will first treat the processes associated with the initial quantities $ \hat{E}^{N,0}_{\ell,i}(t)$,
$ \hat{I}^{N,0,1}_{\ell,i}(t)$ and $ \hat{I}^{N,0,2}_{\ell,i}(t)$. The processes $ \hat{E}^{N,0}_{\ell,i}(t)$,
$ \hat{I}^{N,0,1}_{\ell,i}(t)$ can be treated in the same way, so we will only prove $ \hat{I}^{N,0,1}_{\ell,i}(t)$.

Recall the expression of $ \hat{I}^{N,0,1}_{\ell,i}(t)$ in \eqref{eqn-hatI-01-def}. Also recall the process $ \tilde{I}^{N,0,1}_{\ell,i}(t)$ by defined in the proof of Lemma \ref{lem-I0-ij-conv-SEIR}, that is, replacing $I^N_\ell(0)$ in \eqref{eqn-hatI-01-def} with $N \bar{I}_\ell(0)$, so that $\hat{I}^{N,0,1}_{\ell,i}(t)= \tilde{I}^{N,0,1}_{\ell,i}(t) + (\hat{I}^{N,0,1}_{\ell,i}(t)-\tilde{I}^{N,0,1}_{\ell,i}(t))$. 
The process $\hat{I}^{N,0,1}_{\ell,i}(t)$ is driven by the sequence of two dimensional r.v.'s $(\zeta^0_{k,\ell}, Y^{0,k}_\ell(\zeta^0_{k,\ell}))$, $k\ge 1$. We add a sequence of i.i.d. r.v.'s, globally independent of the above sequence, $U_k$, $k\ge1$, which all have the uniform distribution on the interval $[T,T+1]$. 
 We define a sequence of r.v.'s $\tilde\zeta^0_{k,\ell}$, $k\ge 1$, as follows  (noting that $\tilde\zeta^0_{k,\ell}$ depends on $i$, which we omit for brevity) 
\[\tilde{\zeta}^0_{k,\ell}=\begin{cases} \zeta^0_{k,\ell},&\text{if \quad $Y^{0,k}_\ell(\zeta^0_{k,\ell})=i$, }\\
                                                         \zeta^0_{k,\ell}+U_k,&\text{if \quad $Y^{0,k}_\ell(\zeta^0_{k,\ell})\not=i$}\,.
                                                         \end{cases}
                                                         \]
 We have that for all $t\in[0,T]$, 
     \[ {\bf1}_{\zeta^0_{k,\ell}\le t}    {\bf1}_{ Y^{0,k}_\ell(\zeta^0_{k,\ell})=i} = {\bf1}_{\tilde{\zeta}^0_{k,\ell}\le t}\,,\]     
and  \[\E\big[{\bf1}_{\tilde{\zeta}^0_{k,\ell}\le t}\big] =  \int_0^t  q_{\ell,i}(s) F_0(ds). \]  
By writing 
\[
\frac{1}{\sqrt{\bar{I}_\ell(0)}} \tilde{I}^{N,0,1}_{\ell,i}(t) = \frac{1}{\sqrt{N\bar{I}_\ell(0)}} \sum_{k=1}^{N \bar{I}_\ell(0)} \big(  {\bf1}_{\tilde{\zeta}^0_{k,\ell}\le t} - \E\big[{\bf1}_{\tilde{\zeta}^0_{k,\ell}\le t}\big] \big),
\]
we apply the  Dvoretsky--Kiefer--Wolfowitz inequality (with Massart's optimal constant \cite{Massart90}) and obtain 
\[
\frac{1}{\bar{I}_\ell(0)} \E\bigg[\sup_{0 \le t \le T} \big( \tilde{I}^{N,0,1}_{\ell,i}(t)\big)^2   \bigg] \le  2 \int_0^\infty e^{-2x}dx = 1.
\]
On the other hand, 
\begin{align*}
\hat{I}^{N,0,1}_{\ell,i}(t)-\tilde{I}^{N,0,1}_{\ell,i}(t) &=  \frac{1}{\sqrt{N}} \sum_{k=N \bar{I}_\ell(0) \wedge I^N_\ell(0)+1}^{N \bar{I}_\ell(0) \vee I^N_\ell(0)} \big(  {\bf1}_{\tilde{\zeta}^0_{k,\ell}\le t} - \E\big[{\bf1}_{\tilde{\zeta}^0_{k,\ell}\le t}\big] \big),\\
\sup_{t\ge0}|\hat{I}^{N,0,1}_{\ell,i}(t)-\tilde{I}^{N,0,1}_{\ell,i}(t)|&\le \sqrt{N}|\bar{I}_\ell(0)-\bar{I}^N_\ell(0)|
=|\hat{I}^N_\ell(0)|,
\end{align*}
hence from Assumption  \ref{AS-FCLT-SEIR},
\[
\sup_N  \E\bigg[\sup_{ t \ge 0} \big( \hat{I}^{N,0,1}_{\ell,i}(t)-\tilde{I}^{N,0,1}_{\ell,i}(t) \big)^2   \bigg]  <\infty. 
\]
Combining the above, we have shown that the property in \eqref{hatPhi-2-bound-sup} holds for $ \hat{I}^{N,0,1}_{\ell,i}(t)$. 

For the process $ \hat{I}^{N,0,2}_{\ell,i}(t) $, we can extend the approach above as follows. Define $ \tilde{I}^{N,0,2}_{\ell,i}(t)$ by replacing $E^N_\ell(0)$ with $N \bar{E}_\ell(0)$ in the definition of $ \hat{I}^{N,0,2}_{\ell,i}(t)$ in \eqref{eqn-hatI-02-def}. Write $ \hat{I}^{N,0,2}_{\ell,i}(t) =  \tilde{I}^{N,0,2}_{\ell,i}(t) +  \hat{I}^{N,0,2}_{\ell,i}(t) -  \tilde{I}^{N,0,2}_{\ell,i}(t)$. 
For fixed $\ell, \ell'$, we have a sequence of  i.i.d. random vectors $\big(\eta_{k,\ell}^0, X^{0,k}_\ell(\eta_{k,\ell}^0),  \zeta_{-k,\ell}, Y^{-k,\ell}_{\ell'}( \zeta_{-k,\ell}) \big)_{k\ge 1}$.
We also add a sequence of i.i.d. r.v.'s globally independent of the previous sequence, $U_k$, $k\ge 1$, uniformly distributed on $[T,T+1]$. Define 
\[
\varsigma^0_{k,\ell, \ell'} = \begin{cases} \eta_{k,\ell}^0 + \zeta_{-k,\ell}, & \quad \text{if} \quad X^{0,k}_\ell(\eta_{k,\ell}^0)=\ell', \qandq Y^{-k,\ell}_{\ell'}( \zeta_{-k,\ell}) = i\,, \\
  \eta_{k,\ell}^0 + \zeta_{-k,\ell} +U_k, & \quad \text{if} \quad X^{0,k}_\ell(\eta_{k,\ell}^0)\neq \ell', \quad \text{or} \quad Y^{-k,\ell}_{\ell'}( \zeta_{-k,\ell}) \neq i \,. \\ \end{cases}
\]
 (Note that $\varsigma^0_{k,\ell, \ell'}$ depends on $i$, which we omit for brevity.) 
Then we have
\[
{\bf1}_{X^{0,k}_\ell(\eta_{k,\ell}^0)=\ell'}  {\bf1}_{\eta_{k,\ell}^0 + \zeta_{-k,\ell} \le t} {\bf1}_{Y^{-k,\ell}_{\ell'}( \zeta_{-k,\ell})=i}= {\bf1}_{\varsigma^0_{k,\ell, \ell'} \le t}\,,
\]
and
\[
\E\Big[ {\bf1}_{\varsigma^0_{k,\ell, \ell'} \le t}\Big] =  \int_0^{t} p_{\ell, \ell'}(u) \int_0^{t-u} q_{\ell'i}(v) H_0(du, dv)\,.
\]
Note that $\Phi_{\ell,i}^0(t)$ is the sum of the right hand side of the above equation over $\ell'\in \mathcal{L}$, as given in \eqref{eqn-H0}. Then we can write 
\[
  \tilde{I}^{N,0,2}_{\ell,i}(t) =  \frac{1}{\sqrt{N}}  \sum_{k=1}^{N \bar{E}_\ell (0)} \sum_{\ell'=1}^L  \Big({\bf1}_{\varsigma^0_{k,\ell, \ell'} \le t} - \E\Big[ {\bf1}_{\varsigma^0_{k,\ell, \ell'} \le t}\Big]  \Big) = \sum_{\ell'=1}^L \tilde{I}^{N,0,2}_{\ell,\ell',i}(t) \,,
\]
where 
\[
  \tilde{I}^{N,0,2}_{\ell,\ell',i}(t) :=  \frac{1}{\sqrt{N}} \sum_{k=1}^{N \bar{E}_\ell (0)}  \Big({\bf1}_{\varsigma^0_{k,\ell, \ell'} \le t} - \E\Big[ {\bf1}_{\varsigma^0_{k,\ell, \ell'} \le t}\Big]  \Big)\,. 
\]
We can apply the Dvoretsky--Kiefer--Wolfowitz inequality to obtain the desired estimate for $ \tilde{I}^{N,0,2}_{\ell,\ell',i}(t)$ for each fixed $\ell, \ell'$, that is, 
\[
 \E\bigg[\sup_{0 \le t \le T} \big( \tilde{I}^{N,0,2}_{\ell,\ell',i}(t)\big)^2   \bigg] \le   \bar{E}_\ell(0).
\]
Hence,
\[
 \E\bigg[\sup_{0 \le t \le T} \big( \tilde{I}^{N,0,2}_{\ell,i}(t)\big)^2   \bigg] \le   2L \bar{E}_\ell(0).
\]
The difference $ \hat{I}^{N,0,2}_{\ell,i}(t) -  \tilde{I}^{N,0,2}_{\ell,i}(t)$ is again easy to treat. So we obtain the estimate for $ \hat{I}^{N,0,2}_{\ell,i}(t)$.

\smallskip

We next consider the processes associated with the newly infected individuals  $\hat{E}^{N}_{\ell,i}(t)$ and  $\hat{I}^{N}_{\ell,i}(t)$.  Recall the expression of  $\hat{E}^{N}_{\ell,i}(t)$ and  $\hat{I}^{N}_{\ell,i}(t)$ in \eqref{hatE-ij-def-SEIR} and \eqref{hatI-ij-def-SEIR}, respectively. 
Recall the expressions in \eqref{PRM-rep-E-SEIR} using the PRM $\check{Q}_{\ell}(ds,du, dy, d\theta)$ and \eqref{PRM-rep-I-SEIR} using the PRM $\breve{Q}_{\ell}(ds,du,dy, dz, d \vartheta,  d\theta)$. Also recall that $\overline{Q}_{\ell}$ and  $\wt{Q}_{\ell}$ are the corresponding compensated PRMs.
Thus we can write 
\begin{align} \label{eqn-hatE-1-Q}
\hat{E}^{N}_{\ell,i}(t) 
& = \frac{1}{\sqrt{N}}  \int_0^t \int_0^\infty \int_0^{t-s}\int_{\{i\}} {\bf 1}_{\afrak \le \lambda_\ell \Upsilon^N_\ell(s^-)}\overline{Q}_{\ell} (ds,d\afrak, du, d\theta), 
\end{align}
and
\begin{align} \label{eqn-hatI-1-Q}
\hat{I}^{N}_{\ell,i}(t) 
& = \frac{1}{\sqrt{N}}  \int_0^t \int_0^\infty\int_0^{t-s}\int_0^{t-s-u}   \int_{\LL}\int_{\{i\}} {\bf 1}_{\afrak \le \lambda_\ell \Upsilon^N_\ell(s^-)} \wt{Q}_{\ell}(ds,d\afrak, du, dv, d \theta,  d\vartheta).
\end{align}
Define 
\begin{align} \label{eqn-tildeE-1-Q}
\tilde{E}^{N}_{\ell,i}(t) 
& = \frac{1}{\sqrt{N}}  \int_0^t \int_0^\infty \int_0^{t-s}\int_{\{i\}} {\bf 1}_{\afrak \le \lambda_\ell N\Upsilon_\ell(s)}\overline{Q}_{\ell} (ds,d\afrak, du, d\theta), 
\end{align}
and
\begin{align} \label{eqn-tildeI-1-Q}
\tilde{I}^{N}_{\ell,i}(t) 
& = \frac{1}{\sqrt{N}}  \int_0^t \int_0^\infty \int_0^{t-s}\int_0^{t-s-u} \int_{\LL}\int_{\{i\}} {\bf 1}_{\afrak \le \lambda_\ell  N\Upsilon_\ell(s)} \wt{Q}_{\ell}(ds,d\afrak, du, dv, d \theta, d\vartheta), 
\end{align}
where $\bar\Upsilon_\ell(t)$ is given in \eqref{eqn-Phi} and is deterministic. 
It is not hard to check that $\tilde{E}^{N}_{\ell,i}(t)$ (resp. $\tilde{I}^{N}_{\ell,i}(t)$) is a martingale w.r.t. the filtration 
$\mathcal{F}^{E}_t$ (resp. $\mathcal{F}^{I}_t$), where $\mathcal{F}^{E}_t$ is generated by the restriction of $\check{Q}$ to the set of 
$(s,\mathfrak{a},u,\theta)$ which are such that $s+u\le t$, and $\mathcal{F}^{I}_t$ is generated by the restriction of 
$\breve{Q}$ to the set of 
$(s,\mathfrak{a}, u,v,\theta,\vartheta)$ which are such that $s+u+v\le t$.
Those martingales have the quadratic variations 
\begin{align} \label{eqn-tilde-I-qv}
\langle \tilde{E}^{N}_{\ell,i}\rangle (t) &= \lambda_\ell \int_0^t \int_0^{t-s} q_{\ell,i}(u) F(du) \bar\Upsilon_\ell(s) ds, \quad 
\langle \tilde{I}^{N}_{\ell,i}\rangle (t) =  \lambda_\ell \int_0^t \Phi_{\ell,i}(t-s)  \bar{\Upsilon}_\ell(s) ds. 
\end{align}
Then by Doob's inequality, we obtain
\begin{align*}
& \sup_N \E\bigg[ \sup_{0 \le t \le T} \big( \tilde{E}^{N}_{\ell,i}(t)  \big)^2 \bigg] \le \sup_N 4\, \E\big[  \big( \tilde{E}^{N}_{\ell,i}(T)  \big)^2 \big]  = 4  \lambda_\ell \int_0^T \int_0^{T-s} q_{\ell,i}(u) F(du)  \bar\Upsilon_\ell(s) ds<\infty, \\
& \sup_N \E\bigg[ \sup_{0 \le t \le T} \big( \tilde{I}^{N}_{\ell,i}(t)  \big)^2 \bigg] \le \sup_N 4\, \E\big[  \big( \tilde{I}^{N}_{\ell,i}(T)  \big)^2 \big]  = 4  \lambda_\ell \int_0^T \Phi_{\ell,i}(T-s)  \bar{\Upsilon}_\ell(s) ds <\infty.  
\end{align*}

We next show  that
\begin{equation}\label{estimdiff}
\sup_N\E\left[\sup_{0\le t\le T}\left|\hat{E}^{N}_{\ell,i}(t) -\tilde{E}^{N}_{\ell,i}(t)\right|^2\right]<\infty,\quad
\sup_N\E\left[\sup_{0\le t\le T}\left|\hat{I}^{N}_{\ell,i}(t) -\tilde{I}^{N}_{\ell,i}(t)\right|^2\right]<\infty\,.
\end{equation}
Let us establish the first estimate in \eqref{estimdiff}. The second can be obtained by the exact same argument.
We will use below the identity
\[
\frac{1}{\sqrt{N}} \int_0^\infty \Big|{\bf 1}_{\afrak \le \lambda_\ell \Upsilon^N_\ell(s^-)} -{\bf 1}_{\afrak \le \lambda_\ell N \bar{\Upsilon}_\ell(s)} \Big|  d\afrak = \lambda_\ell  \big| \hat\Upsilon^N_\ell(s^-)\big|\,.
\]

We have
\begin{align*}
\hat{E}^{N}_{\ell,i}(t) -\tilde{E}^{N}_{\ell,i}(t)&=\frac{1}{\sqrt{N}}  \int_0^t \int_0^\infty \int_0^{t-s}\int_{\{i\}} 
\left({\bf 1}_{\afrak \le \lambda_\ell \Upsilon^N_\ell(s^-)}-{\bf 1}_{\afrak \le \lambda_\ell N \bar\Upsilon_\ell(s)}\right)\overline{Q}_{\ell} (ds,d\afrak, du, d\theta), 
\end{align*}
and
\begin{align*}
& \sup_{0\le t\le T}\left|\hat{E}^{N}_{\ell,i}(t) -\tilde{E}^{N}_{\ell,i}(t)\right| \\
&\le\frac{1}{\sqrt{N}}
\int_0^T \int_0^\infty \int_0^{T-s}\int_{\{i\}} 
\left|{\bf 1}_{\afrak \le \lambda_\ell \Upsilon^N_\ell(s^-)}-{\bf 1}_{\afrak \le \lambda_\ell N\bar\Upsilon_\ell(s)}\right|\check{Q}_{\ell} (ds,d\afrak, du, d\theta) +\int_0^T\lambda_\ell  \big| \hat\Upsilon^N_\ell(t)\big|dt\\
&=\frac{1}{\sqrt{N}}
\int_0^T \int_0^\infty \int_0^{T-s}\int_{\{i\}} 
\left|{\bf 1}_{\afrak \le \lambda_\ell \Upsilon^N_\ell(s^-)}-{\bf 1}_{\afrak \le \lambda_\ell N\bar\Upsilon_\ell(s)}\right|\overline{Q}_{\ell} (ds,d\afrak, du, d\theta)+2\int_0^T\lambda_\ell  \big| \hat\Upsilon^N_\ell(t)\big|dt\,.
\end{align*}
Thus we obtain
\begin{align*}
\E\left[\sup_{0\le t\le T}\left|\hat{E}^{N}_{\ell,i}(t) -\tilde{E}^{N}_{\ell,i}(t)\right|^2\right]&\le
2\lambda_\ell\int_0^T \E\big| \bar\Upsilon^N_\ell(t)-\bar\Upsilon_\ell(t)\big|dt+8\lambda_\ell^2 T\int_0^T\E\left[\big| \hat\Upsilon^N_\ell(t)\big|^2\right]dt,
\end{align*}
hence the first part of \eqref{estimdiff}, thanks to \eqref{hatPhi-2-bound}. Plugging the above estimates in \eqref{hatS-i-rep}, \eqref{hatE-rep-2-SEIR}, \eqref{hatI-rep-2-SEIR} and \eqref{hatR-rep-2-SEIR}, using 
\eqref{eq:esimhatUps} and Gronwall's Lemma we finally establish \eqref{hatPhi-2-bound-sup}.
\end{proof}

In the next Lemma, we shall need the following well--known result on integrals with respect to a PRM, which follows
rather easily from Theorem VI.2.9 in \cite{ccinlar2011probability}. 
\begin{lemma}\label{lem:expmoment}
Let $Q$ be a PRM on some measurable space $(E,\mathcal{E})$, with mean measure $\nu$, and $\wt{Q}$ the associated compensated measure.
Let $f:E\mapsto\mathbb{C}$ be measurable and such that $e^f-1-f$ is $\nu$ integrable. Then
\[
\E\left[\exp\left(\int_E f(x)\wt{Q}(dx)\right)\right]
= \exp\left(\int_E \left[e^{f(x)}-1-f(x)\right]\nu(dx)\right)\,.
\]
\end{lemma}

We are now ready to prove the convergence of the components associated with the newly exposed individuals. 

\begin{lemma} \label{lem-I-ij-conv-SEIR}
Under Assumption \ref{AS-FCLT-SEIR}, 
\begin{align}
\big(\hat{E}^{N}_{\ell,i}, \hat{I}^{N}_{\ell,i}, \, \ell,i\in \mathcal{L}\big) \RA\big(\hat{E}_{\ell,i}, \hat{I}_{\ell,i}, \, \ell,i\in \mathcal{L}  \big) \qinq D^{2L^2} \qasq N \to \infty,
\end{align}
where the limits are as given in Theorem \ref{thm-FCLT-SEIR}. 
\end{lemma}

\begin{proof}

Recall the processes $\tilde{E}^{N}_{\ell,i}$ and $\tilde{I}^{N}_{\ell,i}$ defined in \eqref{eqn-tildeE-1-Q} and \eqref{eqn-tildeI-1-Q} using the compensated PRMs $\overline{Q}_{\ell}$  and $\wt{Q}_{\ell}$, respectively. Each of these processes being a martingale, they are easily shown to be tight. We now establish their joint final dimensional convergence. 
By their definitions of the two PRMs, we can regard  $\check{Q}_{\ell}$ (resp. $\overline{Q}_{\ell}$) as the image of $\breve{Q}_{\ell}$ (resp. $\wt{Q}_{\ell}$) by the projection $\Pi$ from $\R_+^4\times \LL^2$ onto $\R_+^3\times \LL$, defined by $\Pi(s,\afrak,u,v,\theta,\vartheta)=(s,\afrak,u,\theta)$. In other words, we can write, together with \eqref{eqn-tildeE-1-Q}, 
\begin{align*}
\tilde{E}^{N}_{\ell,i}(t) 
& = \frac{1}{\sqrt{N}}  \int_0^t \int_0^\infty \int_0^{t-s}\int_0^\infty \int_{\{i\}} \int_\LL{\bf 1}_{\afrak \le \lambda_\ell N\Upsilon_\ell(s)}\wt{Q}_{\ell} (ds,d\afrak, du, dv, d \theta,  d\vartheta). 
\end{align*}
Consequently, for any $\alpha_E, \alpha_I,\alpha'_E, \alpha'_I \in\R$,  and for each $\ell, i,i' \in \mathcal{L}$ and $t,t'>0$, 
\begin{align*}
&\alpha_E\tilde{E}^{N}_{\ell,i}(t)+\alpha_I\tilde{I}^{N}_{\ell,i'} (t)+\alpha'_E\tilde{E}^{N}_{\ell,i}(t')+\alpha'_I\tilde{I}^{N}_{\ell,i'} (t')\\
&=
\int_0^{t\vee t'} \int_0^\infty \int_0^{t\vee t'-s}\int_0^\infty \int_{\LL} \int_\LL  
\left[f_N(s,\afrak,u,v,\theta,\vartheta)+f'_N(s,\afrak,u,v,\theta, \vartheta)\right]
\wt{Q}_{\ell} (ds,d\afrak, du, dv, d \theta, d\vartheta), 
\end{align*}
where 
\begin{align*}
 f_N(s,\afrak,u,v,\theta, \vartheta)&=\frac{1}{\sqrt{N}}{\bf 1}_{\afrak \le \lambda_\ell N\Upsilon_\ell(s)}{\bf1}_{[0,t]}(s){\bf1}_{[0,t-s]}(u)\left(\alpha_E{\bf 1}_{\theta =i} +\alpha_I{\bf1}_{[0,t-s-u]}(v){\bf 1}_{\theta =i'} \right),\\
 f'_N(s,\afrak,u,v,\theta,  \vartheta)&=\frac{1}{\sqrt{N}}{\bf 1}_{\afrak \le \lambda_\ell N\Upsilon_\ell(s)}{\bf1}_{[0,t']}(s){\bf1}_{[0,t'-s]}(u)\left(\alpha'_E{\bf 1}_{\theta =i}+\alpha'_I{\bf1}_{[0,t'-s-u]}(v){\bf 1}_{\theta =i'}\right). 
 \end{align*}
By Lemma \ref{lem:expmoment},  we have
\begin{align*}
&\E\left[\exp\left\{ \hat\imath \alpha_E\tilde{E}^{N}_{\ell,i}(t)+ \hat\imath \alpha_I\tilde{I}^{N}_{\ell,i'}(t)
+ \hat\imath \alpha'_E\tilde{E}^{N}_{\ell,i}(t')+ \hat\imath\alpha'_I\tilde{I}^{N}_{\ell,i'} (t')\right\}\right]\\
&=\exp\Bigg(\int_0^{t\vee t'}  \int_0^\infty \int_0^{t\vee t'-s} \int_0^\infty \int_{\LL}\int_\LL \left(e^{\hat\imath[f_N+f'_N]}-1- \hat\imath [f_N+f'_N]\right)(s,\afrak,u,v,\theta,\vartheta) \\
& \qquad \qquad  \qquad \qquad 
dsd\afrak H(du,dv) \mu^X_\ell(u,d\theta)\mu^Y_\theta(v,d\vartheta)\Bigg) \non\\
& = \exp\Bigg(\int_0^{t \vee t'}  \int_0^\infty \int_0^{t\vee t'-s} \int_0^\infty \int_{\LL}\int_\LL \left(-\frac{1}{2}  [f_N+f'_N]^2 +o(N^{-1})\right)(s,\afrak,u,v,\theta,\vartheta) \\
& \qquad \qquad  \qquad \qquad 
dsd\afrak H(du,dv) \mu^X_\ell(u,d\theta)\mu^Y_\theta(v,d\vartheta)\Bigg)  \non\\
&\xrightarrow{N\to\infty} \exp\Bigg(-\frac{1}{2}\int_0^{t}  \lambda_\ell\Upsilon_\ell(s)\Bigg\{\alpha^2_E \int_0^{t -s}\!\!\!\!\!p_{\ell,i}(u)G(du)+\alpha^2_I \Phi_{\ell,i'}(t-s) \non\\
& \qquad \qquad \qquad \qquad \qquad \qquad    +2\alpha_E\alpha_I\int_0^{t-s}\!\!\int_0^{t-s-u}\!\!\!\!\!\!\!\! p_{\ell,i}(u)q_{i,i'}(v) H(du,dv) \Bigg\}ds \\
&\qquad \qquad -\frac{1}{2}\int_0^{t'}  \lambda_\ell\Upsilon_\ell(s)\Bigg\{(\alpha'_E)^2 \int_0^{t' -s}\!\!\!\!\!p_{\ell,i}(u)G(du)  +(\alpha'_I)^2\Phi_{\ell,i'}(t'-s)  \non \\
& \qquad \qquad \qquad \qquad \qquad \qquad \qquad +2\alpha'_E\alpha'_I\int_0^{t'-s}\!\!\int_0^{t'-s-u}\!\!\!\!\!\!\!\! p_{\ell,i}(u)q_{i,i'}(v)H(du,dv)\Bigg\}ds\\
&\qquad \qquad  -\int_0^{t\wedge t'}\lambda_\ell\Upsilon_\ell(s)\Bigg\{
\alpha_E\alpha'_E \int_0^{t\wedge t' -s}\!\!\!\!\!p_{\ell,i}(u)G(du)
+\alpha_I\alpha'_I \Phi_{\ell,i'}(t\wedge t'-s)\\
&\qquad \qquad \qquad \qquad  
+\alpha_E\alpha'_I\int_0^{t\wedge t'-s}\!\!\int_0^{t'-s-u}\!\!\!\!\!\!\!\! p_{\ell,i}(u)q_{i,i'}(v)H(du,dv) \non\\
& \qquad \qquad \qquad \qquad +\alpha_I\alpha'_E\int_0^{t\wedge t'-s}\!\!\int_0^{t-s-u}\!\!\!\!\!\!\!\! p_{\ell,i}(u)q_{i,i'}(v)H(du,dv)\Bigg\}ds\Bigg) \\
& = \E\left[\exp\left\{\hat\imath\alpha_E\hat{E}_{\ell,i}(t)+ \hat\imath\alpha_I\hat{I}_{\ell,i'}(t)+\hat\imath\alpha'_E\hat{E}_{\ell,i}(t')+\hat\imath\alpha'_I\hat{I}_{\ell,i'}(t') \right\}\right].
\end{align*}
Would we consider more distinct times, we would clearly deduce that the whole vector converges to a Gaussian 
random vector. The only point which requests a detailed computation is the determination of the covariances, which can be deduced from the above formula, and obvious similar formulas. In particular, it is easily seen for 
$\ell\not=\ell'$, that the covariance of $\hat{E}_{\ell,i}(t)$ and $(\hat{E}_{\ell',i'}(t'),\hat{I}_{\ell',i''}(t''))$ is zero. 
Similarly, for $i'\not=i$, the covariances of $\hat{E}_{\ell,i}(t)$ and $\hat{E}_{\ell,i'}(t')$, of $\hat{I}_{\ell,i}(t)$ and
$\hat{I}_{\ell,i'}(t')$  are zero.
(This is a difference with the covariances of $\hat{E}^{0}_{\ell,i}(t)$ and $\hat{E}^{0}_{\ell,i'}(t')$, of $\hat{I}^{0,1}_{\ell,i}(t)$ and $\hat{I}^{0,1}_{\ell,i'}(t')$, of $\hat{I}^{0,2}_{\ell,i}(t)$ and $\hat{I}^{0,2}_{\ell,i'}(t')$ for $i'\neq i$, see the calculations in Lemma~\ref{lem-I0-ij-conv-SEIR} and the formulas in the statement of Theorem~\ref{thm-FCLT-SEIR}. The zero covariance of  $\hat{E}_{\ell,i}(t)$ and $\hat{E}_{\ell,i'}(t')$, of $\hat{I}_{\ell,i}(t)$ and
$\hat{I}_{\ell,i'}(t')$, follows from the Poisson random measure construction of $\hat{E}^N_{\ell,i}(t)$ and $\hat{I}^N_{\ell,i}(t)$, while the processes  $\hat{E}^{N,0}_{\ell,i}(t)$, $\hat{I}^{N,0,1}_{\ell,i}(t)$ and  $\hat{I}^{N,0,2}_{\ell,i}(t)$ have a different structure.)
 The formulas for the covariances of the pair $(\hat{E},\hat{I})$ in the statement of Theorem \ref{thm-FCLT-SEIR} are easy to deduce from the above computation.

\medskip

It then remains to show that,  for each $\ell,i\in \mathcal{L}$, as $ N\to\infty$, 
 $\hat{E}^{N}_{\ell,i}(t) - \tilde{E}^{N}_{\ell,i}(t) \to 0 $ and  $\hat{I}^{N}_{\ell,i}(t) - \tilde{I}^{N}_{\ell,i}(t) \to 0 $
 in probability, locally uniformly in $t$.

We focus on the process $\hat{E}^{N}_{\ell,i} - \tilde{E}^{N}_{\ell,i}$. 
It is clear that 
\begin{align*}
\E\big[ \hat{E}^{N}_{\ell,i}(t) - \tilde{E}^{N}_{\ell,i} (t)  \big] &= 0,\\
\E\big[ \big( \hat{E}^{N}_{\ell,i}(t) - \tilde{E}^{N}_{\ell,i} (t) \big)^2 \big] 
&=  \lambda_\ell \E  \int_0^t \int_0^{t-s} q_{\ell,i}(u) F(du) \big| \bar\Upsilon^N_\ell(s) - \bar\Upsilon_\ell(s)\big| ds  \to 0 \qasq N \to \infty,
\end{align*}
where the convergence holds by  Lemma \ref{lem-3.11} 
 and the dominated convergence theorem. 
We next show that the sequence $\{\hat{E}^{N}_{\ell,i} - \tilde{E}^{N}_{\ell,i}\}_N$ is tight. 
Observe that 
\begin{align*}
\hat{E}^{N}_{\ell,i}(t) - \tilde{E}^{N}_{\ell,i}(t) 
&= \frac{1}{\sqrt{N}}  \int_0^t \int_{\lambda_\ell N (\bar\Upsilon_\ell^N(s)\wedge \bar\Upsilon_\ell(s))}^{\lambda_\ell N (\bar\Upsilon_\ell^N(s)\vee \bar\Upsilon_\ell(s))} \int_0^{t-s}\int_{\{i\}} \text{sign}(\bar\Upsilon_\ell^N(s) - \bar\Upsilon_\ell(s) ) \check{Q}_{\ell}(ds,d\afrak, du, d\theta) \\
& \quad - \lambda_\ell \int_0^t \int_0^{t-s} q_{\ell,i}(u) F(du)  \hat{\Upsilon}_\ell^N(s) ds. 
\end{align*}
We can decompose $\text{sign}(\bar\Upsilon_\ell^N(s) - \bar\Upsilon_\ell(s) ) = {\bf 1}_{\bar\Upsilon_\ell^N(s) - \bar\Upsilon_\ell(s)>0} - {\bf 1}_{ \bar\Upsilon_\ell^N(s) -\bar\Upsilon_\ell(s)<0}$, and write 
  $\hat\Upsilon^N_\ell(s) = (\hat\Upsilon^N_\ell(s)\vee0) - (- \hat\Upsilon^N_\ell(s))\vee 0 $,
hence each of the terms on the right of the last identity can be expressed as a difference of two functions which are nondecreasing in $t$. 
It is also clear that tightness of these processes will be implied by the tightness of the following processes:
\begin{align*}
\Xi^{N}_{1}(t) &= \frac{1}{\sqrt{N}}  \int_0^t \int_{\lambda_\ell N (\bar\Upsilon_\ell^N(s)\wedge \bar\Upsilon_\ell(s))}^{\lambda_\ell N (\bar\Upsilon_\ell^N(s)\vee \bar\Upsilon_\ell(s))} \int_0^{t-s}\int_{\{i\}} \check{Q}_{\ell}(ds,d\afrak, du, d\theta) ,\\
\Xi^{N}_{2}(t) &=  \lambda_\ell \int_0^t \int_0^{t-s} q_{\ell,i}(u) F(du)  |\hat\Upsilon_\ell^N(s)|ds.
\end{align*}
Since these two processes are nondecreasing in $t$, by the Corollary on page 83 in \cite{billingsley1999convergence},  it suffices to show  that for any $\ep>0$, and $\iota=1,2$,
\begin{align} \label{Xi-diff-conv}
\limsup_{N\to\infty} \frac{1}{\delta} \P\left( \left| \Xi^{N}_{\iota}(t+\delta) - \Xi^{N}_{\iota}(t)  \right| >\ep \right) \to 0 \qasq \delta \to 0. 
\end{align}
For the process $\Xi^{N}_1(t)$, we have 
\begin{align} \label{Xi-diff-conv-1-p1}
 & \E\left[ \big| \Xi^{N}_{1}(t+\delta) - \Xi^{N}_{1}(t)  \big|^2\right] \non \\
 &= \E\bigg[ \bigg(  \frac{1}{\sqrt{N}}  \int_t^{t+\delta} \int_{\lambda_\ell N (\bar\Upsilon_\ell^N(s)\wedge \bar\Upsilon_\ell(s))}^{\lambda_\ell N (\bar\Upsilon_\ell^N(s)\vee \bar\Upsilon_\ell(s))} \int_0^{t+\delta-s}\int_{\{i\}} \check{Q}_{\ell}(ds,d\afrak, du, d\theta)  \non\\
& \quad  +  \frac{1}{\sqrt{N}}  \int_0^t \int_{\lambda_\ell N (\bar\Upsilon_\ell^N(s)\wedge \bar\Upsilon_\ell(s))}^{\lambda_\ell N (\bar\Upsilon_\ell^N(s)\vee \bar\Upsilon_\ell(s))} \int_{t-s}^{t+\delta-s}\int_{\{i\}} \check{Q}_{\ell}(ds,d\afrak, du, d\theta)
   \bigg)^2\bigg] \non \\
   & \le  2 \E\bigg[ \bigg(  \frac{1}{\sqrt{N}}  \int_t^{t+\delta} \int_{\lambda_\ell N (\bar\Upsilon_\ell^N(s)\wedge \bar\Upsilon_\ell(s))}^{\lambda_\ell N (\bar\Upsilon_\ell^N(s)\vee \bar\Upsilon_\ell(s))} \int_0^{t+\delta-s}\int_{\{i\}} \check{Q}_{\ell}(ds,d\afrak, du, d\theta)  \bigg)^2\bigg] \non\\
& \quad  + 2 \E\bigg[ \bigg(  \frac{1}{\sqrt{N}}  \int_0^t \int_{\lambda_\ell N (\bar\Upsilon_\ell^N(s)\wedge \bar\Upsilon_\ell(s))}^{\lambda_\ell N (\bar\Upsilon_\ell^N(s)\vee \bar\Upsilon_\ell(s))} \int_{t-s}^{t+\delta-s}\int_{\{i\}} \check{Q}_{\ell}(ds,d\afrak, du, d\theta)
   \bigg)^2\bigg] \non \\
   &\le  4 \E\bigg[ \bigg(  \frac{1}{\sqrt{N}}  \int_t^{t+\delta} \int_{\lambda_\ell N (\bar\Upsilon_\ell^N(s)\wedge \bar\Upsilon_\ell(s))}^{\lambda_\ell N (\bar\Upsilon_\ell^N(s)\vee \bar\Upsilon_\ell(s))} \int_0^{t+\delta-s}\int_{\{i\}} \overline{Q}_{\ell}(ds,d\afrak, du, d\theta)  \bigg)^2\bigg] \non \\
   &\quad + 4 \E\left[ \left(\lambda_\ell \int_t^{t+\delta} \int_0^{t+\delta-s} q_{\ell,i}(u) F(du)   \big| \hat\Upsilon^N_\ell(s) \big| ds\right)^2 \right] \non \\
   & \quad + 4 \E\bigg[ \bigg(  \frac{1}{\sqrt{N}}  \int_0^{t} \int_{\lambda_\ell N (\bar\Upsilon_\ell^N(s)\wedge \bar\Upsilon_\ell(s))}^{\lambda_\ell N (\bar\Upsilon_\ell^N(s)\vee \bar\Upsilon_\ell(s))} \int_{t-s}^{t+\delta-s}\int_{\{i\}} \overline{Q}_{\ell}(ds,d\afrak, du, d\theta)  \bigg)^2\bigg]  \non \\
   &\quad + 4 \E\left[ \left(\lambda_\ell \int_0^{t} \int_{t-s}^{t+\delta-s} q_{\ell,i}(u) F(du)  \big| \hat\Upsilon^N_\ell(s) \big|  ds\right)^2 \right] \non \\
   & \le 4  \lambda_\ell \int_t^{t+\delta} \int_0^{t+\delta-s} q_{\ell,i}(u) F(du)
   \E\left[ \big| \bar\Upsilon^N_\ell(s) - \bar\Upsilon_\ell(s)\big|  \right]  ds + 4 \lambda_\ell^2 \delta^2  \E\left[  \sup_{s \in [0,T]}\big| \hat\Upsilon^N_\ell(s) \big|^2\right]   \non \\
   & \quad + 4  \lambda_\ell \int_0^{t} \int_{t-s}^{t+\delta-s} q_{\ell,i}(u) F(du)
   \E\left[ \big| \bar\Upsilon^N_\ell(s) - \bar\Upsilon_\ell(s)\big|  \right]  ds \non \\
   & \quad + 4  \E\left[ \left(\lambda_\ell \int_0^{t} \int_{t-s}^{t+\delta-s} q_{\ell,i}(u) F(du) \big| \hat\Upsilon^N_\ell(s) \big| ds\right)^2 \right] . 
\end{align}
It is clear that $\E\left[ \big| \bar\Upsilon^N_\ell(s) - \bar\Upsilon_\ell(s)\big|  \right]  \to 0 $ as $N\to \infty$ by the convergence $ \bar\Upsilon^N_\ell\RA \bar\Upsilon_\ell$ and the dominated convergence theorem. Thus, the first and third terms converge to zero as $N\to\infty$. 
  Thanks to \eqref{hatPhi-2-bound-sup}, $\delta^{-1}$ times the second term converges to zero as $\delta \to 0$, which is exactly what we wish.  
 Thus, in order to prove \eqref{Xi-diff-conv} for $\Xi^N_1(t)$, it suffices to show that 
\begin{align} \label{Xi-diff-conv-p4}
\lim_{\delta \to 0}\limsup_{N\to\infty} \frac{1}{\delta}\E\left[ \left(\lambda_\ell \int_0^{t} \int_{t-s}^{t+\delta-s} q_{\ell,i}(u) F(du) \big| \hat\Upsilon^N_\ell(s) \big| ds\right)^2 \right] = 0. 
\end{align}
The expectation is bounded by
\begin{align*}
& \E\bigg[ \sup_{0 \le s \le T} \big| \hat\Upsilon^N_\ell(s) \big|^2 \bigg]  \left(\lambda_\ell \int_0^{t} \int_{t-s}^{t+\delta-s} q_{\ell,i}(u) F(du) ds\right)^2.
\end{align*}
Hence we obtain \eqref{Xi-diff-conv-p4} by the same argument as in \eqref{eqn-conv-without-F-condition}, and the bound in Lemma \ref{lem-estimate-SIR}. 

For the process $\Xi^{N}_2(t)$, we have 
\begin{align*}
 & \E\left[ \big| \Xi^{N}_{2}(t+\delta) - \Xi^{N}_{2}(t)  \big|^2\right] 
  \le  2 \E\left[ \left(\lambda_\ell \int_t^{t+\delta} \int_0^{t+\delta-s} q_{\ell,i}(u) F(du) \big| \hat\Upsilon^N_\ell(s)\big| ds\right)^2 \right] \\
      &\qquad \qquad\qquad \qquad \qquad \qquad  + 2 \E\left[ \left(\lambda_\ell \int_0^{t} \int_{t-s}^{t+\delta-s} q_{\ell,i}(u) F(du) \big| \hat\Upsilon^N_\ell(s)\big| ds\right)^2 \right] \\
   & \le 4 \lambda_1^2 \delta^2 \sup_{s \in [0,T]} \E\left[  \big| \hat\Upsilon^N_\ell(s)\big|^2\right]  + 4  \E\left[ \left(\lambda_\ell \int_0^{t} \int_{t-s}^{t+\delta-s} q_{\ell,i}(u) F(du) \big| \hat\Upsilon^N_\ell(s)\big| ds\right)^2 \right] . 
\end{align*}
The argument for these two terms follow from that for the second and fourth terms above for $\Xi^N_1(t)$.

We next prove that as $N\to\infty$,
$
\hat{I}^{N}_{\ell,i}(t) - \tilde{I}^{N}_{\ell,i}(t) \to 0 $ in probability,  locally uniformly in $t$
for each $\ell,i\in \mathcal{L}$.  It follows a similar argument so we only highlight differences below. 
It is clear that 
\begin{align*}
\E\big[ \hat{I}^{N}_{\ell,i}(t) - \tilde{I}^{N}_{\ell,i} (t)  \big] &= 0,\\
\E\big[ \big( \hat{I}^{N}_{\ell,i}(t) - \tilde{I}^{N}_{\ell,i} (t) \big)^2 \big] 
&=  \lambda_\ell \int_0^t \Phi_{\ell,i}(t-s) \big| \bar\Upsilon^N_\ell(s) -\bar\Upsilon_\ell(s)\big| ds \to 0 \qasq N \to \infty,
\end{align*}
where the convergence holds by Theorem \ref{thm-FLLN-SEIR} and the dominated convergence theorem. 
To show that the sequence $\{\hat{I}^{N}_{\ell,i} - \tilde{I}^{N}_{\ell,i}\}$ is tight, 
we write 
\begin{align*}
& \hat{I}^{N}_{\ell,i}(t) - \tilde{I}^{N}_{\ell,i}(t)  \\
&= \frac{1}{\sqrt{N}} \int_0^t \int_0^\infty \int_0^{t-s}  \int_0^{t-s-u}  \int_\LL   \int_{\{i\}} \Big( {\bf 1}_{\afrak \le \lambda_\ell  \Upsilon^N_\ell(s^-)}-  {\bf 1}_{\afrak \le \lambda_\ell N \bar\Upsilon_\ell(s)}\Big)  \breve{Q}_{\ell}(ds,d\afrak, du, dv, d\theta, d \vartheta)\,, \\
& \quad - \lambda_\ell \int_0^t \Phi_{\ell,i}(t-s)     \hat{\Upsilon}_\ell^N(s) ds, 
\end{align*}
and observe that it suffices  to prove tightness of the following processes
\begin{align*}
\mathcal{I}^N_1(t) &=  \frac{1}{\sqrt{N}} \int_0^t  \int_{\lambda_\ell N (\bar\Upsilon_\ell^N(s)\wedge \bar\Upsilon_\ell(s))}^{\lambda_\ell N (\bar\Upsilon_\ell^N(s)\vee \bar\Upsilon_\ell(s))}  \int_0^{t-s}   \int_0^{t-s-u}\int_\LL   \int_{\{i\}}   \breve{Q}_{\ell} (ds,d\afrak, du, dv,d\theta,  d \vartheta)\\
\mathcal{I}^N_2(t) &= \lambda_\ell \int_0^t \Phi_{\ell,i}(t-s)  \big|   \hat{\Upsilon}_\ell^N(s) \big| ds.
\end{align*}
By the monotone property of these two processes in $t$, we then show that 
for any $\ep>0$, and $\iota=1,2$,
\begin{align} \label{Xi-diff-conv-SEIR}
\limsup_{N\to\infty} \frac{1}{\delta} \P\left( \left| \mathcal{I}^{N}_{\iota}(t+\delta) - \mathcal{I}^{N}_{\iota}(t)  \right| >\ep \right) \to 0 \qasq \delta \to 0. 
\end{align}
Similar to the derivation in \eqref{Xi-diff-conv-1-p1}, we obtain
\begin{align} \label{Xi-diff-conv-SEIR-p1}
 & \E\left[ \big| \mathcal{I}^{N}_{1}(t+\delta) - \mathcal{I}^{N}_{1}(t)  \big|^2\right]  \le 4  \lambda_\ell \int_t^{t+\delta} \Phi_{\ell,i}(t+\delta-s)
   \E\left[ \big| \bar\Upsilon^N_\ell(s) - \bar\Upsilon_\ell(s)\big|  \right]  ds \non\\
   &\quad  + 4 \lambda_\ell^2 \delta^2 \E\left[  \sup_{s \in [0,T]} \big| \hat\Upsilon^N_\ell(s) \big|^2\right]  + 4  \lambda_\ell \int_0^{t} \big(\Phi_{\ell,i}(t+\delta-s) - \Phi_{\ell,i}(t-s) \big)
   \E\left[ \big| \bar\Upsilon^N_\ell(s) - \bar\Upsilon_\ell(s)\big|  \right]  ds \non \\
   & \quad + 4  \E\left[ \left(\lambda_\ell \int_0^{t}  \big(\Phi_{\ell,i}(t+\delta-s) - \Phi_{\ell,i}(t-s) \big) \big| \hat\Upsilon^N_\ell(s) \big| ds\right)^2 \right] . 
\end{align}
The first and third terms converge to zero as $N\to \infty$ by the convergence 
$\E\left[ \big| \bar\Upsilon^N_\ell(s) - \bar\Upsilon_\ell(s)\big|  \right] \to 0$ and applying the dominated convergence theorem.
For the last term in \eqref{Xi-diff-conv-SEIR-p1}, we have
\begin{align}\label{Xi-diff-conv-SEIR-p2}
&  \E\left[ \left(\lambda_\ell \int_0^{t}  \big(\Phi_{\ell,i}(t+\delta-s) - \Phi_{\ell,i}(t-s) \big) \big| \hat\Upsilon^N_\ell(s) \big| ds\right)^2 \right] \non \\
& \le  2\E\left[ \left(\lambda_\ell \int_0^{t} \left(\int_0^{t+\delta-s}  \sum_{\ell'=1}^L p_{\ell, \ell'}(u) \int_{t-s-u}^{t+\delta-s-u} q_{\ell'i}(v) H(du,dv) \right)      \big| \hat\Upsilon^N_\ell(s) \big| ds\right)^2 \right]  \non \\
& \quad +  2 \E\left[ \left(\lambda_\ell \int_0^t \left(  \int_{t-s}^{t+\delta-s}  \int_0^{t-s-u} \sum_{\ell'=1}^L p_{\ell, \ell'}(u) q_{\ell'i}(v) H(du,dv)\right)      \big| \hat\Upsilon^N_\ell(s) \big| ds\right)^2 \right] \non\\
& \le  2\E\bigg[ \sup_{s \in [0,T]} \big| \hat\Upsilon^N_i(s)\big|^2\bigg] \left(\lambda_\ell \int_0^{t}  \left( \int_0^{t+\delta-s}  \sum_{\ell'=1}^L p_{\ell, \ell'}(u) \int_{t-s-u}^{t+\delta-s-u} q_{\ell'i}(v) H(du,dv) \right)    ds\right)^2 \non\\
& \quad + 2 \E\bigg[ \sup_{s \in [0,T]} \big| \hat\Upsilon^N_i(s)\big|^2\bigg]\left(\lambda_\ell   \int_0^t \left(  \int_{t-s}^{t+\delta-s}  \int_0^{t-s-u} \sum_{\ell'=1}^L p_{\ell, \ell'}(u) q_{\ell'i}(v) H(du,dv)\right)  ds\right)^2  \non\\
& \le 2\E\bigg[ \sup_{s \in [0,T]} \big| \hat\Upsilon^N_i(s)\big|^2\bigg] \left(\lambda_\ell  \int_0^{t} \int_0^{t+\delta-s}
 (F(t+\delta-s-u|u) - F(t-s-u|u) )  G(du) ds\right)^2 \non\\
& \quad +2 \E\bigg[ \sup_{s \in [0,T]} \big| \hat\Upsilon^N_i(s)\big|^2\bigg] \left(\lambda_\ell   \int_0^t ( G(t+\delta-s) - G(t-s) )
     ds\right)^2 .
\end{align}
Then the first term is treated with the same argument as in \eqref{I-diff-inc-p1-2} while the second as in \eqref{eqn-conv-without-F-condition}.

We then consider 
\begin{align*}
 \E\left[ \big| \mathcal{I}^{N}_{2}(t+\delta) - \mathcal{I}^{N}_{2}(t)  \big|^2\right] 
 & \le  2 \E\left[ \left(\lambda_\ell  \int_t^{t+\delta} \Phi_{\ell,i}(t+\delta-s)  \big| \hat{\Upsilon}_\ell^N(s) \big|ds\right)^2 \right] \\
      &\quad + 2 \E\left[ \left(\lambda_\ell \int_0^{t} (\Phi_{\ell,i}(t+\delta-s) - \Phi_{\ell,i}(t-s)) \big| \hat{\Upsilon}_\ell^N(s) \big| ds\right)^2 \right]. 
\end{align*}
Here the first term is bounded by 
$$2 \E\bigg[ \sup_{s \in [0,T]} \big| \hat\Upsilon^N_\ell(s)\big|^2\bigg]  \left(\lambda_\ell  \int_t^{t+\delta} \Phi_{\ell,i}(t+\delta-s) ds\right)^2 \le 2 \lambda_\ell^2 \delta^2   \E\left[   \sup_{s\in [0,T]}\big| \hat\Upsilon^N_\ell(s) \big|^2 \right]$$
 and the second term is treated as above in \eqref{Xi-diff-conv-SEIR-p2}. 
This completes the proof of \eqref{Xi-diff-conv-SEIR}, and thus $
\hat{I}^{N}_{\ell,i}(t) - \tilde{I}^{N}_{\ell,i}(t) \to 0 $ in probability, locally uniformly in $t$. 
 This completes the proof.
\end{proof}

\begin{proof}[\bf Completing the proof of Theorem \ref{thm-FCLT-SEIR}]


Let $\wt{S}^N_i(t), \wt{E}^N_i(t), \wt{I}^N_i(t), \wt{R}^N_i(t)$ be defined as in \eqref{hatS-i-rep}, \eqref{hatE-rep-2-SEIR}, \eqref{hatI-rep-2-SEIR} and \eqref{hatR-rep-2-SEIR}  correspondingly with 
$\wt{S}^N_i(0) =\hat{S}^N_i(0), \wt{E}^N_i(0) =\hat{E}^N_i(0), \wt{I}^N_i(0) = \hat{I}^N_i(0)$ and $\hat\Upsilon^N_i(t)$ being replaced by $\wt\Upsilon^N_i(t)$ defined by 
\begin{align*}
\wt{\Upsilon}^N_i(t)& =
\psi_s(\bar{S}_i(t),\bar{I}_i(t),\bar{R}_i(t),\sum_{\ell\not=i}\kappa_{i\ell}\bar{I}_\ell(t))\wt{S}^N_i(t) 
+ \psi_e(\bar{S}_i(t),\bar{I}_i(t),\bar{R}_i(t),\sum_{\ell\not=i}\kappa_{i\ell}\bar{I}_\ell(t))\wt{E}^N_i(t) \non\\
& \quad + \psi_i(\bar{S}_i(t),\bar{I}_i(t),\bar{R}_i(t),\sum_{\ell\not=i}\kappa_{i\ell}\bar{I}_\ell(t))\wt{I}^N_i(t) 
+ \psi_r(\bar{S}_i(t),\bar{I}_i(t),\bar{R}_i(t),\sum_{\ell\not=i}\kappa_{i\ell}\bar{I}_\ell(t))\wt{R}^N_i(t) \non\\
& \quad 
+\psi_u(\bar{S}_i(t),\bar{I}_i(t),\bar{R}_i(t),\sum_{\ell\not=i}\kappa_{i\ell}\bar{I}_\ell(t))
\sum_{\ell\not=i}\kappa_{i\ell}\wt{I}^N_\ell(t). 
\end{align*}
and the other components remain unchanged.  
Then by Lemma  \ref{lem-map-digamma-SEIR} below (with $m=L$), and by the convergence results in  Lemmas \ref{lem-Mi-conv-SEIR}, \ref{lem-I0-ij-conv-SEIR} and \ref{lem-I-ij-conv-SEIR}, we obtain that $(\wt{S}^N_i, \wt{E}^N_i, \wt{I}^N_i, \wt{R}^N_i, i \in \mathcal{L}) \RA (\hat{S}_i, \hat{E}^N_i, \hat{I}_i, \hat{R}_i, i \in \mathcal{L})$ in $D^{3L}$. 
It remains to show that 
$(\wt{S}^N_i -\hat{S}^N_i, \wt{E}^N_i -\hat{E}^N_i, \wt{I}^N_i -\hat{I}^N_i, \wt{R}^N_i - \hat{R}^N_i, i \in \mathcal{L}) \RA 0$. 
We have
\begin{align*}
\hat{S}_i^N(t) - \wt{S}^N_i(t) &=  -\lambda_i \int_0^t  (\hat{\Upsilon}_i^N(s) - \wt{\Upsilon}_i^N(s))  ds  \non \\
& \quad + \sum_{\ell=1, \ell\neq i}^L \int_0^t (\nu_{S,\ell,i} (  \hat{S}^N_\ell(s)-  \wt{S}^N_\ell(s))-  \nu_{S,i, \ell} (\hat{S}^N_i(s) -  \wt{S}^N_i(s)) ) ds  \,,
\end{align*}
\begin{align*}
\hat{E}^N_i(t) - \wt{E}^N_i(t)  &=  \lambda_i \int_0^t (\hat{\Upsilon}^N_i(s)- \wt{\Upsilon}_i^N(s)) ds 
-  \sum_{\ell=1}^L \lambda_\ell \int_0^t \int_0^{t-s} p_{\ell,i}(u) G(du) ( \hat\Upsilon^N_\ell(s)- \wt{\Upsilon}_i^N(s)) ds \non\\
 & \quad  + \sum_{\ell=1, \ell\neq i}^L \int_0^t (\nu_{E,\ell,i}(\hat{E}^N_\ell(s)-\wt{E}^N_\ell(s)) - \nu_{E,i, \ell} (\hat{E}^N_i(s)-\wt{E}^N_\ell(s)) ) ds\,,
\end{align*}
\begin{align*}
\hat{I}^N_i(t) -\wt{I}^N_i(t)&= 
  \sum_{\ell=1}^L \lambda_\ell \int_0^t \int_0^{t-s} p_{\ell,i}(u) G(du)  ( \hat\Upsilon^N_\ell(s)- \wt\Upsilon^N_\ell(s))  ds   \non\\
  &\quad    - \sum_{\ell=1}^L \lambda_\ell  \int_0^t \Phi_{\ell,i}(t-s) ( \hat{\Upsilon}^N_\ell(s) - \wt\Upsilon^N_\ell(s)) ds \non\\
 & \quad    +   \sum_{\ell \neq i} \int_0^t \left( \nu_{I,\ell,i} ( \hat{I}^N_\ell(s) - \wt{I}^N_\ell(s)) -  \nu_{I,i,\ell} (\hat{I}^N_i(s) - \wt{I}^N_\ell(s))  \right) ds ,
\end{align*}
\begin{align*}
\hat{R}^N_i(t) - \wt{R}^N_i(t)&=
   \sum_{\ell=1}^L \lambda_\ell  \int_0^t \Phi_{\ell,i}(t-s) ( \hat{\Upsilon}^N_\ell(s) - \wt{\Upsilon}^N_\ell(s)) ds     \non\\
 & \quad   + \sum_{\ell=1, \ell\neq i}^L \int_0^t (\nu_{R,\ell,i}(\hat{R}^N_\ell(s) - \wt{R}^N_\ell(s)) - \nu_{R,i, \ell} (\hat{R}^N_i(s) - \wt{R}^N_i(s)) ) ds. 
\end{align*}


Let
\begin{align*}
\psi^N_{s,a}(t) & := 
\psi'_s \bigg(\bar{S}_i(t)+a(\bar{S}^N_i(t)-\bar{S}_i(t)), \bar{E}_i(t)+a(\bar{E}^N_i(t)-\bar{E}_i(t)), \bar{I}_i(t)+a(\bar{I}^N_i(t)-\bar{I}_i(t)), \\
& \qquad  \bar{R}_i(t)+a(\bar{R}^N_i(t)-\bar{R}_i(t)),  
 \sum_{\ell\not=i}\kappa_{i\ell} \bar{I}_\ell(t)+a(\sum_{\ell\not=i}\kappa_{i\ell}\bar{I}_\ell^N(t)-\sum_{\ell\not=i}\kappa_{i\ell}\bar{I}_\ell(t)) \bigg). 
\end{align*}
Similarly for $\psi^N_{e,a}(t)$, $\psi^N_{i,a}(t)$, $\psi^N_{r,a}(t)$  and $\psi^N_{u,a}(t)$.  
Also write $\psi_{e,0}(t), \psi_{i,0}(t), \psi_{r,0}(t), \psi_{u,0}(t), \psi_{u,0}(t)$ when $a=0$ (noting that they no longer depend on $N$ in this case). 
We then have 
\begin{align*}
& \hat{\Upsilon}_i^N(t) - \wt{\Upsilon}_i^N(t)  \\
&= \sqrt{N}\Big( \psi(\bar{S}_i^N(t),\bar{E}_i^N(t), \bar{I}_i^N(t),\bar{R}_i^N(t),\sum_{\ell\not=i}\kappa_{i\ell} \bar{I}_\ell^N(t)) - \psi(\bar{S}_i(t),\bar{E}_i(t), \bar{I}_i(t),\bar{R}_i(t),\sum_{\ell\not=i}\kappa_{i\ell} \bar{I}_\ell(t)) \Big) -   \wt{\Upsilon}_i^N(t)\\
& = \int_0^1 \psi^N_{s,a}(t) da ( \hat{S}^N_i(t) - \wt{S}^N_i(t) )+  \int_0^1 \psi^N_{e,a}(t) da (\hat{E}^N_i(t) - \wt{E}^N_i(t)) + \int_0^1 \psi^N_{i,a}(t) da (\hat{I}^N_i(t) - \wt{I}^N_i(t))  \\
& \quad  + \int_0^1 \psi^N_{r,a}(t) da (\hat{R}^N_i(t) - \wt{R}^N_i(t))  + \int_0^1 \psi^N_{u,a}(t) da \sum_{\ell\neq i} ( \hat{I}^N_\ell(t) - \wt{I}^N_j(t)) \\
& \quad + \Big( \int_0^1 \psi^N_{s,a}(t) da - \psi_{s,0}(t) \Big)  \wt{S}^N_i(t) +  \Big( \int_0^1 \psi^N_{e,a}(t) da - \psi_{e,0}(t) \Big)  \wt{E}^N_i(t)
+  \Big( \int_0^1 \psi^N_{i,a}(t) da - \psi_{i,0}(t) \Big)  \wt{I}^N_i(t)  \\
& \quad + \Big( \int_0^1 \psi^N_{r,a}(t) da -    \psi_{r,0}(t) \Big)  \wt{R}^N_i(t) +  \Big( \int_0^1 \psi^N_{u,a}(t) da - \psi_{u,0}(t) \Big) \sum_{\ell\neq i}  \wt{I}^N_\ell(t). 
\end{align*}
We can use the bounds in the proof of Lemma \ref{estimhatUps} to show that $\limsup_N \sup_{t\in [0,T]} \int_0^1 \psi^N_{s,a}(t) da <\infty$ and the same holds for  $\psi^N_{e,a}(t)$, $\psi^N_{i,a}(t)$, $\psi^N_{r,a}(t)$ and $\psi^N_{u,a}(t)$. It is also clear that 
$\int_0^1 \psi^N_{s,a}(t) da - \psi_{s,0}(t) \to 0$ as $N \to \infty$, uniformly in $t$, and similarly for the others. 
In addition, similarly as Lemma \ref{estimhatUps}, we can also show that $ \sup_N \sup_{t \in [0,T]} \E\left[  \big| \wt{S}^N_i(t)\big|^2\right] < \infty$ and the same for $\wt{E}^N_i(t)$, $\wt{I}^N_i(t)$ and $ \wt{R}^N_i(t)$.
Then by Gronwall's inequality, we conclude that $(\wt{S}^N_i -\hat{S}^N_i, \wt{I}^N_i -\hat{I}^N_i, \wt{R}^N_i - \hat{R}^N_i, i \in \mathcal{L}) \RA 0$.  This completes the proof of the theorem. 
\end{proof}

\section{Concluding Remarks}

We have extended the approach in \cite{PP-2020} to study multi-patch SEIR models with Markov migrations among patches. In this generalization we have introduced a new formulation for the infection process, and further developed the methodology to tackle the challenges arising from that and the migration processes. However we have assumed a constant infectivity rate for each
individual. In \cite{FPP2020b,PP2020-FCLT-VI}, in a model with homogeneous population, each individual is associated with a random infectivity function, for which FLLNs and FCLTs have been established. It would be interesting to study multi-patch models with varying infectivity. 
In addition, with an infection age dependent infectivity, FLLNs and PDEs have been established for the one-patch models in \cite{PP-2021}. PDEs for the multi-patch models with an infection age dependent infectivity can be also derived. Control strategies such vaccination and isolation have been developed for epidemic models \cite{abakuks1973optimal,wickwire1975optimal,hansen2011optimal,bolzoni2019optimal}. For multi-patch models, one may also consider control strategies restricting migrations in different patches.

\bigskip 

\paragraph{\bf Acknowledgement}
The authors want to thank the anonymous reviewers for their helpful comments that have led to substantial improvements in the paper. 
G. Pang was supported in part by  the US National Science Foundation grants DMS-1715875 and DMS-2108683/2216765 and Army Research Office grant W911NF-17-1-0019.

\section{Appendix} \label{sec-appendix}

We define a $4m$-dimensional integral mapping $\digamma$: given $a_i, b_i, c_i,  d_i, \phi_i, \psi_i,\varphi_i,  \chi_i\in D$, some constants $\alpha_i, \beta_i, \gamma_i, \kappa_i>0$ and  functions $F_{\ell,i}$, $G_{\ell,i}$  for $\ell,i=1,\dots,m$, let $x_i,y_i,z_i,w_i$ be the solutions to the following integral mapping: 
\begin{align*}
x_i(t) &= x_i(0) + \phi_i(t) - \int_0^t (a_{i}(s) x_i(s) + b_{i}(s) y_i(s) + c_{i}(s) z_i(s) + d_i(s) w_i(s))ds \\
& \qquad + \sum_{\ell=1,\ell\neq i}^m \int_0^t (\alpha_{\ell,i} x_\ell(s) - \alpha_{i,\ell} x_i(s)) ds\,,\\
y_i(t) &= y_i(0) + \psi_i(t) + \int_0^t (a_{i}(s) x_i(s) + b_{i}(s) y_i(s) + c_{i}(s) z_i(s)+ d_i(s) w_i(s))ds \\
& \quad - \sum_{\ell=1}^m \int_0^t F_{\ell,i}(t-s) (a_{\ell}(s) x_\ell(s) + b_{\ell}(s) y_\ell(s) + c_{\ell}(s) z_\ell(s)+ d_\ell(s) w_\ell(s) )ds \\
& \quad + \sum_{\ell=1,\ell\neq i}^m  \int_0^t (\beta_{\ell,i} y_\ell(s) - \beta_{i,\ell} y_i(s)) ds \,,\\
z_i(t) &= z_i(0) + \varphi_i(t)-  \sum_{\ell=1}^m \int_0^t F_{\ell,i}(t-s) (a_{\ell}(s) x_\ell(s) + b_{\ell}(s) y_\ell(s) + c_{\ell}(s) z_\ell(s)+ d_\ell(s) w_\ell(s) )ds  \\
& \quad 
-  \sum_{\ell=1}^m \int_0^t G_{\ell,i}(t-s) (a_{\ell}(s) x_\ell(s) + b_{\ell}(s) y_\ell(s) + c_{\ell}(s) z_\ell(s) + d_\ell(s) w_\ell(s))ds \\
& \quad + \sum_{\ell=1,\ell\neq i}^m  \int_0^t (\gamma_{\ell,i} z_\ell(s) - \gamma_{i,\ell} z_i(s)) ds, \\
w_i(t) &= w_i(0) + \chi_i(t) 
+ \sum_{\ell=1}^m \int_0^t F_{\ell,i}(t-s) (a_{\ell}(s) x_\ell(s) + b_{\ell}(s) y_\ell(s) + c_{\ell}(s) z_\ell(s)+ d_\ell(s) w_\ell(s) )ds\\
& \quad + \sum_{\ell=1}^m \int_0^t G_{\ell,i}(t-s) (a_{\ell}(s) x_\ell(s) + b_{\ell}(s) y_\ell(s) + c_{\ell}(s) z_\ell(s) + d_\ell(s) w_\ell(s))ds \\
& \quad + \sum_{\ell=1,\ell\neq i}^m  \int_0^t (\kappa_{\ell,i} w_\ell(s) - \kappa_{i,\ell} w_i(s)) ds. 
\end{align*}

 The existence and uniqueness of its solution and the continuity property are stated in the following lemma.

\begin{lemma} \label{lem-map-digamma-SEIR}
Assume that  $F_{\ell,i}$ and $G_{\ell,i}$ , $\ell,i=1,\dots,m$ are measurable, bounded and continuous functions satisfying $F_{\ell,i}(0)=0$ and $G_{\ell,i}(0)=0$, and let the constants
$\alpha_i, \beta_i, \gamma_i, \kappa_i>0$ and the functions $\phi_i, \psi_i,\varphi_i, \chi_i$ be given. There exists a unique solution $(x_i,y_i, z_i, w_i, \, i=1,\dots,m) \in D^{4m}$ to the set of integrable equations defining the mapping $\tilde\digamma$. 
The mapping is continuous in the Skorohod $J_1$ topology, that is, 
if $(a^n_i, b^n_i,c^n_i, d^n_i, \phi^N_i, \psi^n_i, \varphi^n_i, \chi^n_i \, i=1,\dots,m) 
 \to (a_i, b_i,c_i, d_i,\phi_i, \psi_i, \varphi_i, \chi_i, \,i=1,\dots,m)$ in $D([0,T], \R^{8m})$ as $n\to\infty$ and 
 $(x^n_i(0), y^n_i(0), z^n_i(0),  w^n_i(0), \, i=1,\dots,m) \to (x_i(0),  y_i(0),\\ z_i(0), w_i(0), \, i=1,\dots,m)$, 
 then  $(x^n_i,y^n_i,z^n_i, w^n_i,\, i=1,\dots,m) \to (x_i, y_i, z_i, w_i,\,i=1,\dots,m)$ in $D([0,T], \R^{4m})$ as $n\to \infty$. In addition, if $\phi_i, \psi_i,\varphi_i, \chi_i$ are continuous, then 
$(x_i, y_i, z_i, w_i\,i=1,\dots,m)\in C^{4m}$ and the mapping $\digamma$ is continuous uniformly on compact sets in $[0,T]$. 
\end{lemma}

\begin{proof}
For the existence and uniqueness of solutions, we can apply the Schauder-Tychonoff fixed point theorem, and modify the proofs in Theorems 1.2 and 2.3 in Chapter II of \cite{miller1971} (where these results are shown for Volterra integral equations with continuous functions). 
The continuity can be proved similarly as Lemma 9.1 in \cite{PP-2020}. 
\end{proof}

\bibliographystyle{plain}

\bibliography{SIR2patch}

\end{document}